\documentclass[12pt,a4paper,reqno]{amsart}
\usepackage[left=2cm, right=2cm, bottom=3cm]{geometry} 
\usepackage{xcolor}

\newcommand\dela[1]{}

\usepackage[ngerman, english]{babel}

\usepackage{palatino}

\usepackage{amsmath}
\usepackage{amsthm}
\usepackage{amssymb}
\usepackage{url}
\usepackage{hyperref}
\usepackage{graphicx}
\usepackage{url}
\usepackage{listings}
\makeindex

\allowdisplaybreaks

\newtheorem{Satz}{Theorem}[section]
\newtheorem{Theorem}[Satz]{Theorem}
\newtheorem{Lemma}[Satz]{Lemma}		
\newtheorem{Korollar}[Satz]{Corollary}	
\newtheorem{Prop}[Satz]{Proposition}	
\numberwithin{equation}{section}
\theoremstyle{definition}

\newtheorem{Definition}[Satz]{Definition}

\newtheorem{Remark}[Satz]{Remark}

\newtheorem{Assumption}[Satz]{Assumption}
\newtheorem{AssumptionNot}[Satz]{Assumption and Notation}

\newcommand{\BIGboxplus}{\mathop{\mathchoice%
		{\raise-0.35em\hbox{\huge $\boxplus$}}%
		{\raise-0.15em\hbox{\Large $\boxplus$}}{\hbox{\large $\boxplus$}}{\boxplus}}}

\newcommand{\C}{\mathbb{C}} 
\newcommand{\R}{\mathbb{R}} 
\newcommand{\Rd}{{\mathbb{R}^d}} 
\newcommand{\Q}{\mathbb{Q}} 
\newcommand{\Z}{\mathbb{Z}} 
\newcommand{\N}{\mathbb{N}} 

\newcommand{\E}{\mathbb{E}}
\newcommand{\Prob}{\mathbb{P}}
\newcommand{\F}{\mathcal{F}}
\newcommand{\Filtration}{\mathbb{F}}
\newcommand{\tildeProb}{\tilde{\Prob}}
\newcommand{\Etilde}{\tilde{\E}}
\newcommand{\HS}{\operatorname{HS}}
\newcommand{\tr}{\operatorname{tr}}

\newcommand{\Yosida}{R_\lambda}

\newcommand{\mass}{\mathcal{M}}
\newcommand{\energy}{\mathcal{E}}

\newcommand{\EA}{{E_A}}

\newcommand{\EAdual}{{E_A^*}}

\newcommand{\Hs}{{H^s(M)}}
\newcommand{\LinftyEA}{{L^\infty(0,T;\EA)}}

\newcommand{\StetigEAdual}{{C([0,T],\EAdual)}}
\newcommand{\LInfty}{{L^\infty(M)}}
\newcommand{\LzweiTimeSum}{{L^2([s,t]\times \N)}}
\newcommand{\LzweiTimeSumHminusEins}{{L^2([s,t]\times \N,\EAdual)}}
\newcommand{\weaklyContinousEA}{C_w([0,T],\EA)}
\newcommand{\stetigBall}{C([0,T],\mathbb{B}_{\EA}^r)}
\newcommand{\stetigBallX}{C([0,T],\mathbb{B}_{X}^r)}

\newcommand{\LalphaPlusEins}{{L^{\alpha+1}(M)}}
\newcommand{\LalphaPlusEinsAlphaPlusEins}{{L^{\alpha+1}(0,T;\LalphaPlusEins)}}
\newcommand{\LalphaPlusEinsDual}{{L^{\frac{\alpha+1}{\alpha}}(M)}}

\newcommand{\D}{\mathcal{D}}
\newcommand{\df }{\mathrm{d}}
\newcommand{\im }{\mathrm{i}}
\newcommand{\sumM }{\sum_{m=1}^{\infty}}
\newcommand{\Real}{\operatorname{Re}}
\newcommand{\Ima}{\operatorname{Im}}
\newcommand{\skpH}[2]{\big(#1,#2\big)_{H}}
\newcommand{\skp}[2]{\big(#1,#2\big)}
\newcommand{\skpHReal}[2]{\Real \big(#1,#2\big)_{H}}
\newcommand{\skpLzwei}[2]{\big(#1,#2\big)_{L^2}}
\newcommand{\skpHBig}[2]{\Big(#1,#2\Big)_{H}}
\newcommand{\sqrtA}{A^{\frac{1}{2}}}

\newcommand{\norm}[1]{\Vert #1 \Vert}
\newcommand{\bigNorm}[1]{\left\Vert #1 \right\Vert}
\newcommand{\quadVar}[1]{\langle \langle #1 \rangle \rangle}
\newcommand{\duality}[2]{\langle #1, #2 \rangle}
\newcommand{\dualityReal}[2]{\Real \langle #1, #2 \rangle}

\newcommand{\Fhat}{\hat{F}}
\newcommand{\rhodot}{\dot{\rho}}

\newcommand{\smin}{\frac{\alpha-1}{\alpha}}

\newcommand{\Addresses}{{
		\bigskip
		\footnotesize
		
		Zdzis{\l}aw ~Brze{\'{z}}niak, Department of Mathematics, University of York,
			Heslington, York, YO105DD, UK\par\nopagebreak
		\textit{E-mail address}: zdzislaw.brzezniak@york.ac.uk
		\medskip
		
		Fabian~Hornung, Institute for Analysis, Karlsruhe Institute for Technology (KIT), 76128 Karlsruhe, Germany\par\nopagebreak
		\textit{E-mail address}: fabian.hornung@kit.edu
		
		\medskip
		
		Lutz~Weis, Institute for Analysis, Karlsruhe Institute for Technology (KIT), 76128 Karlsruhe, Germany\par\nopagebreak
		\textit{E-mail address}: lutz.weis@kit.edu
		
	}}

\title[Stochastic nonlinear Schr\"odinger equation in the energy space]{Martingale solutions for the stochastic nonlinear Schr\"odinger equation in the energy space}

\author[Z. BRZE{\'Z}NIAK, F. HORNUNG AND L. WEIS]{Zdzis{\l}aw Brze{\'z}niak, Fabian Hornung  and  Lutz Weis	}
\thanks{We gratefully acknowledge financial support by the
Deutsche Forschungsgemeinschaft (DFG) through CRC 1173.}
\date{\today}

\begin{document}

\begin{abstract}
	We consider a stochastic nonlinear Schr\"{o}dinger equation with multiplicative noise in an abstract framework that
	covers subcritical focusing and defocusing Stochastic NLSE in $H^1$ on compact manifolds and bounded domains. We construct a martingale solution using a modified Faedo-Galerkin-method based on the Littlewood-Paley-decomposition. For the 2d manifolds with bounded geometry, we use the Strichartz estimates to show the pathwise uniqueness of solutions.
\end{abstract}
\maketitle



\section{Introduction}

The article is concerned with the following nonlinear stochastic Schr\"{o}dinger equation
\begin{equation}
\label{ProblemStratonovich}
\left\{
\begin{aligned}
\df u(t)&= \left(-\im A u(t)-\im  F(u(t))\right) dt-\im B u(t) \circ \df W(t),\qquad t> 0,\\
u(0)&=u_0,
\end{aligned}\right.
\end{equation}
in the energy space $E_A:=\D(\sqrtA),$ where $A$ is a selfadjoint, non-negative operator $A$ with a compact resolvent in an $L^2$-space $H,$ $F$ a nonlinearity,  $B$ a linear bounded operator,  $W$ is a Wiener process and the equation is understood in the  (multiplicative) Stratonovich sense. \\
Three basic examples of the operator $A$ are
\begin{itemize}
	\item the negative Laplace-Beltrami operator $-\Delta_g$ on a compact riemannian manifold $(M,g)$ without boundary,
	\item the negative Laplacian $-\Delta$ on a bounded domain of $\Rd$ with Neumann or Dirichlet boundary conditions,
	\item fractional powers of the first two examples.
\end{itemize}
The two basic model nonlinearities are
\begin{itemize}
	\item
	the defocusing power nonlinearity $F_{\alpha}^+(u):=\vert u\vert^{\alpha-1}u$ with subcritical exponents in the sense that the embedding $\EA \hookrightarrow L^{\alpha+1}$ is compact
	\item
	and the focusing  nonlinearity $F_{\alpha}^-(u):=-\vert u\vert^{\alpha-1}u$ with an additional restriction to the power $\alpha.$
\end{itemize}
The typical noise term has the form
\begin{align}\label{StratonovichDifferential}
-\im B u(t) \circ \df W(t) &=-\im  \sumM e_m u(t) \circ \df \beta_m(t)=-\frac{1}{2}\sumM e_m^2 u(t)-\im \sumM e_m u(t) \df \beta_m(t)
\end{align}
with a sequence of independent standard real Brownian motions $\left(\beta_m\right)_{m\in\N}$ and functions $\left(e_m\right)_{m\in\N}$ satisfying certain regularity and decay conditions that guarantee the convergence of the series on the RHS of \eqref{StratonovichDifferential} in the space $\EA.$

The main aim of this study  is twofold. Firstly, it proposes to
construct a martingale solution of problem
$\eqref{ProblemStratonovich}$ by a
a stochastic version of a compactness method. Secondly, it proposes to
prove the uniqueness of solutions by means of the stochastic Strichartz
estimates. In this respect it differs from many previous papers on
stochastic nonlinear Schr\"odinger equations, notably
\cite{BrzezniakStrichartz}, \cite{BouardLzwei}, \cite{FHornung}, and references therein,
in which the proofs of both the existence and the uniqueness were
obtained by means of appropriate stochastic Strichartz estimates. The
compactness approach to the existence of solutions of 1-D stochastic
Schr\"odinger equations in  variational form has recently been  used
in  a paper \cite{lisei2016stochastic} by Keller and Lisei.
Classical references for the construction of weak solutions of the deterministic NLSE by a  combination of a compactness method and the Galerkin approximation are \cite{GajewskiDeutsch} and \cite{Gajewski} for intervals and \cite{Nasibov} as well as \cite{Vladimirov} for domains of arbitrary dimension. Let us point out
that Burq, G\'{e}rard, Tzvetkov in \cite{Burq} also used a compactness method in the proof of their Theorem 3 but instead of the Galerkin approximation they used an approximation by more regular solutions. In particular, we give a new proof of these results. But we would like to emphasise  that the deterministic case is significantly simpler since our spectral theoretic methods to construct the approximations of the noise term are not needed.

In technical sense, the present paper is  motivated by the  construction of a global solution of the cubic  equation on compact 3d-manifolds $M$ generalizing the existence part, see Theorem 3 of Burq, G\'{e}rard, Tzvetkov in \cite{Burq},  to the stochastic setting.
In three dimensions, the fixed point argument from \cite{BrzezniakStrichartz} is restricted to higher regularity, because it requires the  Sobolev embeddings $H^{s,q} \hookrightarrow L^\infty,$ which are more restrictive in 3D than in 2D. Hence, this approach only yields local solutions, which is the motivation for constructing a global  solution in $H^1(M)$ with an approximation procedure based on the conservation laws of the NLSE without using the dispersive properties of the Schr\"{o}dinger group.
We remark that in \cite{Burq}, the authors also prove uniqueness for the deterministic NLSE in 3D. For the equation with noise this question will be addressed in a forthcoming paper.

In the present paper, we construct a martingale solution of problem $\eqref{ProblemStratonovich}$ by a modified Faedo-Galerkin approximation
\begin{equation}
\left\{
\begin{aligned}\label{GalerkinIntroduction}
\df u_n(t)&= \left(-\im A u_n(t)-\im P_n F\left( u_n(t)\right) \right) \df t-\im  S_n B(S_n u_n(t)) \circ \df W(t),\quad t>0,
\\
u_n(0)&=P_n u_0,
\end{aligned}\right.
\end{equation}
in finite dimensional subspaces $H_n$ of $H$ spanned by some eigenvectors of $A.$
Here,\\ $P_n: H\to H_n$ are the standard orthogonal projections and $S_n: H\to H_n$ are selfadjoint operators derived from the  Littlewood-Paley-decomposition associated to $A.$
The reason for using the operators $\left(S_n\right)_{n\in\N}$ lies in the uniform estimate
\begin{align*}
\sup_{n\in\N}\norm{S_n}_{L^p\to L^p}<\infty, \quad 1<p<\infty,
\end{align*}
which turns out to be necessary in the estimates of the noise due to the $L^p$-structure of the energy, see \eqref{energyIntro} below, and which is false if one replaces $S_n$ by $P_n.$
Using the Littlewood-Paley decomposition via the operators $(S_n)_{n\in\N}$ can be viewed as the one of the main analytical contributions of this paper. We remark that in the mean time, a similar construction has  been used in \cite{hornung2017strong} to construct a solution of a stochastic nonlinear Maxwell equation by estimates in $L^q$ for some $q>2$. This indicates that our method has potential to increase the field of application of the classical Faedo-Galerkin method significantly.

On the other hand, the orthogonal projections $P_n$ are used in the deterministic part, because they do not destroy the cancellation effects which lead to the mass and energy conservation
\begin{align}\label{energyIntro}
\norm{u}_{L^2}^2=\operatorname{const}, \qquad
\frac{1}{2}\norm{\sqrtA u}_{L^2}^2+\Fhat(u)=\operatorname{const}
\end{align}
for solutions $u$ of problem $\eqref{ProblemStratonovich}$ in the deterministic setting, where $\Fhat$ denotes the antiderivative of the nonlinearity $F.$ Note that in the case $F_\alpha^\pm(u)=\pm\vert u\vert^{\alpha-1}u$, the antiderivative is given by $\hat{F}_\alpha^\pm=\pm \frac{1}{\alpha+1}\norm{u}_{L^{\alpha+1}}^{\alpha+1}.$
In the stochastic case, the mass conservation $\norm{u_n}_{L^2}^2=\operatorname{const}$ for solutions of $\eqref{GalerkinIntroduction}$ holds almost surely due to the Stratonovich form of the noise.
Moreover, the conservation of the energy is carried over in the sense that  a Gronwall type argument yields the uniform a priori estimates,  for every $T>0$,
\begin{align}\label{uniformIntroduction}
\sup_{n\in\N}\E \Big[\sup_{t\in[0,T]} \norm{u_n(t)}_\EA^2\Big]<\infty ,\qquad
\sup_{n\in\N}\E \Big[\sup_{t\in[0,T]} \norm{u_n(t)}_\LalphaPlusEins^{\alpha+1}\Big]<\infty.
\end{align}
Combined with the Aldous condition $[A],$ see Definition \ref{DefinitionAldous},  which is a stochastic version of the equicontinuity, \dela{see Definition $\ref{DefinitionAldous},$} the   estimates $\eqref{uniformIntroduction}$ lead to the tightness of the sequence $\left(u_n\right)_{n\in\N}$ in the locally convex space
\begin{align*}
Z_T:=\StetigEAdual\cap \LalphaPlusEinsAlphaPlusEins \cap \weaklyContinousEA,
\end{align*}
where $\weaklyContinousEA$ denotes the space of continuous functions with respect to the weak topology in $\EA.$
The construction of a martingale solution is similar to \cite{BrzezniakMotyl} and employs a limit argument based on Jakubowski's extension of the Skorohod Theorem to nonmetric spaces and the Martingale Representation Theorem from \cite{daPrato}, chapter 8. Our main result is the following Theorem.

\begin{Theorem}\label{MainTheorem}
	Let $T>0$ and $u_0\in E_A.$  Under the assumptions $\ref{spaceAssumptions}$, $\ref{nonlinearAssumptions}$, $\ref{focusing},$ $\ref{stochasticAssumptions},$ there exists a martingale solution  $\left(\tilde{\Omega},\tilde{\F},\tilde{\Prob},\tilde{W},\tilde{\Filtration},u\right)$ of equation $\eqref{ProblemStratonovich}$ (see Definition $\ref{MartingaleSolutionDef}$), which satisfies
	\begin{align}\label{propertySolution}
	u\in L^q(\tilde{\Omega},\LinftyEA)
	\end{align}
	for all $q\in [1,\infty)$ and
	\begin{align*}
	\norm{u(t)}_{L^2(M)}=\norm{u_0}_{L^2(M)}\qquad \text{$\tilde{\Prob}$-a.s. for all $t\in[0,T]$.}
	\end{align*}
\end{Theorem}

As an application of Theorem $\ref{MainTheorem},$ we get the following Corollary. Note that an analogous result holds in the case of a bounded domain, see Corollary $\ref{CorollaryDomain}$.

\begin{Korollar}\label{CorollaryManifold}
	Let $(M,g)$ be a compact $d$-dimensional riemannian manifold without boundary. Let $T>0$ and $u_0\in H^1(M).$ Under assumption $\ref{stochasticAssumptions}$ and either i) or ii)
	\begin{itemize}
		\item[i)] $F(u)= \vert u\vert^{\alpha-1}u$ with $ \alpha \in \left(1,1+\frac{4}{(d-2)_+}\right),$
		\item[ii)]$F(u)= -\vert u\vert^{\alpha-1}u$ with $ \alpha \in \left(1,1+\frac{4}{d}\right),$
	\end{itemize}
	the equation
	\begin{equation}
	\label{ProblemStratonovichManifold}
	\left\{
	\begin{aligned}
	\df u(t)&= \left(\im \Delta_g u(t)-\im  F(u(t)\right) dt-\im B u(t) \circ \df W(t)\quad \text{in} \hspace{0.1cm} H^1(M),\\
	u(0)&=u_0,
	\end{aligned}\right.
	\end{equation}
	has a martingale solution with	
	\begin{align}\label{propertySolution-2}
	u\in L^q(\tilde{\Omega},{L^\infty(0,T;H^1(M))}),
	\end{align}
	for all $q\in [1,\infty)$ and
	\begin{align*}
	\norm{u(t)}_{L^2(M)}=\norm{u_0}_{L^2(M)}\qquad \text{$\tilde{\Prob}$-a.s. for all $t\in[0,T]$.}
	\end{align*}
\end{Korollar}

Furthermore, we address the question of uniqueness of the solution from Corollary $\ref{CorollaryManifold}$ in two dimensions.
\begin{Korollar}\label{CorollaryManifold2d}
	In the situation of Corollary $\ref{CorollaryManifold}$ with $d=2,$ there exists a unique strong solution of $\eqref{ProblemStratonovichManifold}$ in $H^1(M)$ and the martingale solutions are unique in law.
\end{Korollar}

We obtain pathwise uniqueness by an improvement of the regularity of solutions based on the Strichartz estimates by Bernicot and Samoyeau from \cite{Bernicot} and Brze{\'{z}}niak and Millet from \cite{BrzezniakStrichartz}. Ondrej{\'a}t showed in \cite{OndrejatUniqueness} in a quite general setting, that this is sufficient to get a strong solution. In fact, our uniqueness result is more general than we have formulated in Corollary \ref{CorollaryManifold2d}. On the one hand, we allow possibly non-compact of manifolds with bounded geometry. On the other hand, uniqueness holds in the strictly larger class $L^{r}(\Omega,L^\beta(0,T;H^s(M)))$ with $r> \alpha,$ $\beta:=\max\left\{2,\alpha\right\}$ and
\begin{align*}
s \in \begin{cases}
(\frac{2\alpha-1}{2\alpha},1] & \text{for } \alpha\in (1,3], \\
(\frac{\alpha(\alpha-1)-1}{\alpha(\alpha-1)},1] & \text{for } \alpha>3.
\end{cases}
\end{align*}
For the details, we refer to Theorem \ref{Uniqueness2d}


Let us point out that the stochastic nonlinear Schr\"odinger equations are used in the fiber optics, nonlinear photonics and optical wave turbulence, see for instance a recent review paper  \cite{Turitsyn+_2012} by S Turitsyn et al. and references therein. 
There is also an extended literature on the nonlinear Schr\"odinger equations on special manifolds, as e.g. Schwarzschild manifolds, see papers \cite{Laba+Soffer_2000}, \cite{Andresson_2015-symmetries} and \cite{Marzula+_2010-Strichartz}. In these papers the Schr\"odinger equation is somehow related to the corresponding nonlinear wave equation which in turn appears in the theory of  gravitational fields.
Furthermore, we would like to mention the article \cite{parthasarathy1968derivation} which deals with the derivation of the Schr\"odinger equation on manifolds.
From a mathematical point of view, important questions are  how the geometry of the manifold influences the 
qualitative behavior of solutions and how the geometry of the manifold and the external noise influence the well-posedness theory. 
Nonlinear Schr\"odinger equations on manifolds have been studied  e.g. by Burg et al. \cite{Burq+Gerard+Tzvetkov_2003,Burq}, see also references therein.  The motivation for these authors was  "to evaluate
the impact of geometry of the manifold  on the well-posedness theory, having in
mind the infinite propagation speed of the Schr\"odinger equation". 

The paper is organized as follows. In the Sections 2 and 3, we fix the notation\dela{s}, formulate our Assumptions and present a number of typical examples of operators $A,$ a model nonlinearity $F$ and  noise coefficients $B$ covered by our framework.  In Section 4,  we are concerned with the compactness results that we will be using later on. In Section 5, we formulate the Galerkin approximation equations and prove the a priori estimates which are sufficient for compactness in view of Section 4.  Section 6 is devoted to the proof of Theorem 1 and in Section 7, we focus on uniqueness in the case of 2d manifolds with bounded geometry.

\section{Notation\dela{s} and Assumptions }
In this section, we want to fix the notations, explain the assumptions and formulate an abstract framework for the stochastic nonlinear Schr\"{o}dinger equation. \\

Let $\left(X,\Sigma,\mu\right)$ be a $\sigma$-finite measure space with metric $\rho$ satisfying the \emph{doubling property}, i.e. $\mu(B(x,r))<\infty$ for all $x\in X$ and $r>0$ and
\begin{align}\label{doubling}
\mu(B(x,2r))\lesssim \mu(B(x,r)).
\end{align}
This estimate implies 
\begin{align}\label{doublingDimension}
\mu(B(x,tr))\lesssim t^d \mu(B(x,r)),\qquad x\in X,\quad r>0,\quad t\ge 1
\end{align}
and the number $d\in\N$ is called \emph{doubling dimension.}	
Let $M\subset X$ be an open subset with finite measure and  $L^q(M)$ for $q\in [1,\infty]$ the space of equivalence classes of $\C$-valued $q-$integrable functions. For $q\in [1,\infty],$ let  $q':=\frac{q}{q-1}\in [1,\infty]$ be the conjugate exponent. In particular, for $q\in [1,\infty]$ it holds that $\frac{1}{q}+\frac{1}{q'}=1.$
We further abbreviate ${H}:=L^2(M)$.
In the special case that $M$ is a Riemannian manifold, $H^{s,q}(M)$ denotes the fractional Sobolev space of regularity $s\in \R$ and integrability $q\in (1,\infty)$ and we shortly write $H^s(M):=H^{s,2}(M).$ For a definition of these spaces, we refer to Definition \ref{DefinitionSobolevManifold}. 

If functions $a,b\ge 0$ satisfy the inequality $a\le C(A) b$ with a constant $C(A)>0$ depending on the expression $A$, we write $a \lesssim_A b.$ If we have $a \lesssim_A b$ and $b \lesssim_A a,$ we write $a \eqsim_A b.$
For two   Banach spaces $E,F$, we denote by $\mathcal{L}(E,F)$ the space of linear bounded operators $B: E\to F$ and abbreviate $\mathcal{L}(E):=\mathcal{L}(E,E).$  Furthermore, we write $E\hookrightarrow F,$ if $E$ is continuously embedded in $F;$ i.e. $E\subset F$ with natural embedding  $j\in \mathcal{L}(E,F).$ The space $C^{1,2}([0,T]\times E,F)$ consists of all functions $\varPhi\colon [0,T]\times E\to F$ such that $\varPhi(\cdot,x)\in C^1([0,T],F)$ for every $x\in E$ and $\varPhi(t,\cdot)\in C^2(E,F)$ for every $t\in[0,T].$ For two Hilbert spaces $H_1$ and $H_2,$ the space of Hilbert-Schmidt operators $B: H_1\to H_2$ is abbreviated by $\HS(H_1,H_2).$ The resolvent set of a densely defined linear operator $A\colon E\supset \mathcal{D}(A)\to E$ on a Banach space $E$ is denoted by $\rho(A).$
For a probability space $\left(\Omega, \F, \Prob\right),$ the law of a random variable $X: \Omega \to E$ is denoted by $\Prob^X$.

\begin{AssumptionNot}\label{spaceAssumptions}
	We assume the following: 
	\begin{itemize}
		\item[i)] Let $A$ be a non-negative selfadjoint operator on ${H}$ with domain $\mathcal{D}(A).$
		\item [ii)] There is a strictly positive selfadjoint operator $S$ on ${H}$ with compact resolvent commuting with $A$ which fulfills
		$\D(S^k)\hookrightarrow E_A$
		for sufficiently large $k.$ Moreover, we assume that $S$  has \emph{generalized Gaussian $(p_0,p_0')$-bounds} for some $p_0\in [1,2),$ i.e. 
		\begin{align}\label{generalizedGaussianEstimate}
		\Vert \mathbf{1}_{B(x,t^\frac{1}{m})}e^{-tS}\mathbf{1}_{B(y,t^\frac{1}{m})}\Vert_{\mathcal{L}(L^{p_0},L^{p_0'})} \le C{\mu(B(x,t^\frac{1}{m}))}^{\frac{1}{p_0'}-\frac{1}{p_0}} \exp \left\{-c \left(\frac{\rho(x,y)^m}{t}\right)^{\frac{1}{m-1}}\right\},
		\end{align}
		for all $t>0$ and  $(x,y)\in M\times M$ with constants $c,C>0$ and $m\ge 2.$
		\item[iii)] 		
		The Hilbert space $E_A:=\D(\sqrtA)$ equipped with the inner product
		\begin{align*}
		\skp{u}{v}_{\EA}:=\skpH{u}{v}+\skpH{\sqrtA u}{\sqrtA v},\qquad u,v\in \EA,
		\end{align*}
		is called the \emph{energy space} and the induced norm $\norm{\cdot}_{E_A}$ is called the  \emph{energy norm} associated to $A.$ We denote the dual space of $E_A$ by $E_A^*$ and abbreviate the duality with $\duality{\cdot}{\cdot}:= \duality{\cdot}{\cdot}_{E_A^*,E_A},$ where the complex conjugation is taken over the second variable of the duality. 
		Note that $\left(E_A, H, E_A^*\right)$ is a Gelfand triple, i.e.
		\begin{align*}
		E_A\hookrightarrow H \cong H^* \hookrightarrow E_A^*.
		\end{align*}
		\item[iv)] Let $\alpha \in (1,p_0'-1)$ be such that $\EA$ is compactly embedded in $\LalphaPlusEins.$ We set
		\begin{align*}
		p_{\max}:= \sup \left\{p\in (1,\infty]: \EA \hookrightarrow L^p(M) \quad\text{is continuous}\right\}\dela{\in [\alpha+1,\infty].}
		\end{align*}
		and note that $p_{\max}\in [\alpha+1,\infty].$
		In the case $p_{\max}<\infty,$ we assume that $\EA \hookrightarrow L^{p_{\max}}(M)$ is continuous, but not necessarily compact.
	\end{itemize}
\end{AssumptionNot}
\begin{Remark}\label{GaussianRemark}
	\begin{itemize}
		\item[a)] The operator $S$ plays the role of an auxiliary operator to cover the different examples from Section 3 in a unified framework. Typical choices are $S:=I+A,$ $S:=A$ or $S:=I+A^{1/\beta}$ for some $\beta>0.$ 
		\item[b)] 	If $p_0=1,$ then it is proved in \cite{blunckKunstmann} that \eqref{generalizedGaussianEstimate} is equivalent to the usual upper Gaussian estimate, i.e. for all $t>0$ there is a measurable function $p(t,\cdot,\cdot): M\times M\to \R$ with
		\begin{align*}
		(e^{-tS}f)(x)= \int_M p(t,x,y) f(y) \mu(dy), \quad t> 0, \quad \text{a.e. } x\in M
		\end{align*}
		for all $f\in H$ and
		\begin{align}\label{GaussianEstimate}
		\vert p(t,x,y)\vert \le \frac{C}{\mu(B(x,t^\frac{1}{m}))} \exp \left\{-c \left(\frac{\rho(x,y)^m}{t}\right)^{\frac{1}{m-1}}\right\},
		\end{align}
		for all $t>0$ and almost all $(x,y)\in M\times M$ with constants $c,C>0$ and $m\ge 2.$
		\item[c)]The generalized Gaussian estimate \eqref{generalizedGaussianEstimate} is used in the proof of Proposition \ref{PaleyLittlewoodLemma}, where spectral multiplier theorems for $S$ in $L^p(M)$ for $p\in (p_0,p_0'),$ respectively a Mihlin $\mathcal{M}^\beta$ functional calculus of $S$ for some $\beta>0$ are employed. The Mihlin functional calculus is defined and studied in \cite{krieglerDissertation} and \cite{KrieglerWeis}.    For additional information about spectral multiplier theorems for operators with generalized Gaussian estimates, we refer to \cite{Uhl},  \cite{kunstmannUhl}. Note that spectral multiplier results with different assumptions are also sufficient for our analysis below, see e.g. \cite{OuhabazGaussianBounds}, where a result for the Laplace-Beltrami operator on a compact riemannian manifold is explicitly stated without mentioning the doubling property in this particular case. 
	\end{itemize}
	
\end{Remark}
We start with some conclusions which can be deduced  from Assumption $\ref{spaceAssumptions}.$	
\begin{Lemma}\label{spaceLemma}
	\begin{itemize}
		\item[a)] 		There is a non-negative selfadjoint operator $\hat{A}$ on $E_A^*$ with $\D(\hat{A})=E_A$ with $\hat{A}=A$ on $H.$
		\item[b)]  The embedding
		$\EA \hookrightarrow {H}$
		is compact.
		\item[c)] There is an orthonormal basis $\left(h_n\right)_{n\in \N}$ and a nondecreasing sequence $\left(\lambda_n\right)_{n\in\N}$ with $\lambda_n>0$ and  $\lambda_n\to \infty$ as $n\to \infty$ and  
		\begin{align*}
		S x=\sum_{n=1}^\infty \lambda_n \skpH{x}{h_n} h_n, \quad x\in \D(S)=\left\{x\in H: \sum_{n=1}^\infty \lambda_n^2 \vert \skpH{x}{h_n}\vert^2<\infty\right\},
		\end{align*}
	\end{itemize}
\end{Lemma}

\begin{proof}
	\emph{ad a).}
	The operator $\hat{A}$ is defined by
	\begin{align*}
	\duality{\hat{A}\varphi}{\psi}:= \skpH{\sqrtA \varphi}{\sqrtA \psi}, \quad \varphi,\psi \in E_A.
	\end{align*}
	The estimate
	\begin{align*}
	\vert \duality{\hat{A}\varphi}{\psi}\vert \le \norm{\sqrtA \varphi}_H \norm{\sqrtA \psi}_H\le \norm{\varphi}_{E_A} \norm{\psi}_{E_A}
	\end{align*}
	shows that $\hat{A}$ is well-defined and a bounded operator from $E_A$ to $E_A^*$ with
	$\norm{\hat{A}}\le 1.$  Moreover, one can apply the Lax-Milgram-Theorem to see that $I+\hat{A}$ is a surjective isometry from $\EA$ to $\EAdual.$ If one equips $\EAdual$ with the inner product
	\begin{align*}
	\skp{f^*}{g^*}_{\EAdual}:=\skp{(I+\hat{A})^{-1}f^*}{(I+\hat{A})^{-1}g^*}_\EA, \qquad f^*,g^*\in \EAdual,
	\end{align*}	
	one can show the symmetry of $\hat{A}$ as an unbounded operator in $\EAdual.$ Hence, $\hat{A}$ is selfadjoint, because $-1\in \rho(\hat{A}).$ 
	
	\emph{ad b).} The  embedding $\EA \hookrightarrow \LalphaPlusEins$ is compact by Assumption $\ref{spaceAssumptions}$ iv) and \\$\LalphaPlusEins \hookrightarrow {H}$ is continuous due to $\mu(M)<\infty.$ Hence, $\EA \hookrightarrow H$ is compact. \\	
	\emph{ad c).} Immediate consequence of the spectral theorem, since $S$ has a compact resolvent. 
\end{proof}	
In most cases where this does not cause ambiguity or confusion, we also use the notations $A$ for $\hat{A}.$ We continue with the assumptions on the nonlinear part of our problem.

\begin{Assumption}\label{nonlinearAssumptions}
	Let $\alpha\in(1,p_0'-1)$ be chosen as in Assumption $\ref{spaceAssumptions}.$ Then, we assume the following:
	\begin{itemize}
		\item[i)] Let $F: \LalphaPlusEins \to \LalphaPlusEinsDual$ be a function satisfying the following estimate
		\begin{align}\label{nonlinearityEstimate}
		\norm{F(u)}_\LalphaPlusEinsDual \lesssim \norm{u}_\LalphaPlusEins^\alpha,\quad u\in \LalphaPlusEins.
		\end{align}
		Note that this leads to $F: \EA \to \EAdual$ by Assumption $\ref{spaceAssumptions}$ iv), because $\EA\hookrightarrow\LalphaPlusEins$ implies $(\LalphaPlusEins)^*=\LalphaPlusEinsDual\hookrightarrow \EAdual.$ We further assume and $F(0)=0$ and
		\begin{align}\label{nonlinearityComplex}
		\Real \duality{\im u}{F(u)}=0, \quad u\in \LalphaPlusEins.
		\end{align}
		\item[ii)] The map $F: \LalphaPlusEins\to \LalphaPlusEinsDual$ is continuously real Fr\'{e}chet differentiable with
		\begin{align}\label{deriveNonlinearBound}
		\Vert F'[u]\Vert_{L^{\alpha+1}\to L^\frac{\alpha+1}{\alpha}} \lesssim \norm{u}_\LalphaPlusEins^{\alpha-1}, \quad u\in \LalphaPlusEins.
		\end{align}
		\item[iii)] The map $F$ has a real antiderivative $\hat{F},$ i.e. there exists a Fr\'{e}chet-differentiable map \\ $\hat{F}: \LalphaPlusEins\to \R$ with
		\begin{align}\label{antiderivative}
		\Fhat'[u]h=\Real \duality{F(u)}{h},\quad u,h\in \LalphaPlusEins.
		\end{align}
	\end{itemize}
\end{Assumption}
By Assumption $\ref{nonlinearAssumptions}$ ii) and the mean value theorem for Fr\'{e}chet differentiable maps, we get
\begin{align}\label{nonlinearityLocallyLipschitz}
\norm{  F(x)-F(y)}_{\LalphaPlusEinsDual}&\le\sup_{t\in [0,1]}\norm{F'[tx+(1-t)y]}	\norm{  x-y}_{\LalphaPlusEins}\nonumber\\
&\lesssim \left(\norm{x}_\LalphaPlusEins+\norm{y}_\LalphaPlusEins\right)^{\alpha-1} \norm{x-y}_\LalphaPlusEins,	\qquad x,y\in \LalphaPlusEins,	
\end{align}	
which means that the nonlinearity is  Lipschitz on bounded sets of $\LalphaPlusEins.$\\

We will cover the following two standard types of nonlinearities.
\begin{Definition}
	Let $F$ satisfy Assumption $\ref{nonlinearAssumptions}.$
	Then, $F$ is called \emph{defocusing}, if $\Fhat(u)\ge 0$  and \emph{focusing}, if $\Fhat(u)\le 0$ for all $u\in \LalphaPlusEins.$
	
\end{Definition}


\begin{Assumption}\label{focusing}
	We assume either i) or i'):
	\begin{itemize}
		\item[i)] Let $F$ be defocusing and satisfy
		\begin{align}\label{boundantiderivative}
		\norm{u}_\LalphaPlusEins^{\alpha+1}\lesssim \Fhat(u), \quad u\in \LalphaPlusEins.
		\end{align}
		\item[i')] Let $F$ be focusing and satisfy
		\begin{align}\label{boundantiderivativeFocusing1}
		-\Fhat(u)\lesssim\norm{u}_\LalphaPlusEins^{\alpha+1}, \quad u\in \LalphaPlusEins.
		\end{align}
		and there is $\theta \in (0,\frac{2}{\alpha+1})$ with
		\begin{align}\label{interpolationFocusing}
		\left({H},\EA\right)_{\theta,1}\hookrightarrow \LalphaPlusEins.
		\end{align}
	\end{itemize}
\end{Assumption}
Here $\left(\cdot,\cdot\right)_{\theta,1}$ denotes the real interpolation space and 
we remark that by \cite{TriebelInterpolationTheory}, Lemma 1.10.1, \eqref{interpolationFocusing} is equivalent to
\begin{align}\label{boundantiderivativeFocusing2}
\norm{u}_\LalphaPlusEins^{\alpha+1}\lesssim \norm{u}_H^{\beta_1} \norm{u}_\EA^{\beta_2},\quad u\in \EA.
\end{align}
for some $\beta_1>0$ and $\beta_2 \in (0,2)$ with $\alpha+1= \beta_1+\beta_2.$
Let us continue with the definitions and assumptions for the stochastic part.

\begin{Assumption}\label{stochasticAssumptions}
	We assume the following:
	\begin{itemize}
		\item[i)] Let $(\Omega,\F,\Prob)$ be a probability space, $Y$ a separable real Hilbert space with ONB $(f_m)_{m\in\N}$ and $W$ a $Y$-canonical cylindrical Wiener process adapted to a filtration $\Filtration$ satisfying the usual conditions.
		\item[ii)] Let $B: {H} \to \HS(Y,{H})$ be a linear
		operator and set $B_m u:=B(u)f_m$ for $u\in {H}$ and $m\in \N.$ Additionally, we assume that $B_m\in\mathcal{L}(H)$ is selfadjoint for every $m\in\N$ with
		\begin{align}\label{noiseBoundsH}
		\sumM \norm{B_m}_{\mathcal{L}({H})}^2<\infty
		\end{align}
		and assume $B_m\in \mathcal{L}(\EA)$ and $B_m\in \mathcal{L}(\LalphaPlusEins)$ for $m\in\N$ and $\alpha\in (1,p_0'-1)$ as in Assumption and Notation \ref{spaceAssumptions} with
		\begin{align}\label{noiseBoundsEnergy}
		\sumM\norm{B_m}_{\mathcal{L}(\EA)}^2<\infty,\hspace{1cm}
		\sumM \norm{B_m}_{\mathcal{L}(L^{\alpha+1})}^2<\infty.
		\end{align}
	\end{itemize}
	
\end{Assumption}

For the special case, when the $B_m$ are pointwise multiplication operators, see section \ref{NoiseSection} below.

\begin{Remark}
	The estimates $\eqref{noiseBoundsH}$ and $\eqref{noiseBoundsEnergy}$ imply
	\begin{align*}
	B\in \mathcal{L}({H},\HS(Y,{H})),\quad B\in \mathcal{L}(\EA,\HS(Y,\EA)),\quad B\in \mathcal{L}(\LalphaPlusEins,\gamma(Y,\LalphaPlusEins)),
	\end{align*}
	where $\gamma(Y,\LalphaPlusEins)$ denotes the spaces of $\gamma$-radonifying operators from $Y$ to $\LalphaPlusEins.$
\end{Remark}

Finally, we have sufficient background  to formulate the problem which we want to solve.
We investigate the following stochastic evolution equation in the Stratonovich form	
\begin{equation}\label{ProblemStratonovich2}
\left\{
\begin{aligned}
\df u(t)&= \left(-\im A u(t)-\im  F(u(t)\right) dt-\im B u(t) \circ \df W(t),\hspace{0,3 cm} t\in (0,T),\\
u(0)&=u_0,
\end{aligned}\right.
\end{equation}
where the stochastic differential is defined by	
\begin{align}
-\im B u(t) \circ\df W(t)=-\im B u(t) \df W(t)+\frac{1}{2} \tr_Y\left(\mathcal{M}(u(t))\right) \df t,
\end{align}
with the bilinear form $\mathcal{M}(u)$ on $Y\times Y$ defined by
\begin{align*}
\mathcal{M}(u)(y_1,y_2):= -\im B'[u](-\im B(u)y_1)y_2, \quad u\in {H},\quad y_1,y_2\in Y.
\end{align*}
For \dela{our}{the} purpose of giving a rigorous definition of a solution to problem \eqref{ProblemStratonovich2}, it is useful to rewrite the equation in the It\^o form. Therefore,
we first compute
\begin{align*}
\tr_Y\left(\mathcal{M}(u)\right)&=\sumM -\im B'[u](-\im B(u)f_m)f_m= -\sumM B\left(B(u)f_m\right)f_m\\
&= -\sumM B\left( B_m u\right)f_m=-\sumM B_m^2 u.
\end{align*}
Hence,  equation $\eqref{ProblemStratonovich2}$ will be understood in the following It\^o form 	
\begin{equation}\label{Problem}
\left\{
\begin{aligned}
\df u(t)&= \left(-\im A u(t)-\im  F(u(t)+ \mu \left(u(t)\right) \right) \df t-\im B u(t) \df W(t),\hspace{0,3 cm} t\in (0,T),\\
u(0)&=u_0,
\end{aligned}\right.
\end{equation}
where the linear operator $\mu$ defined by
\begin{align*}
\mu(u) := -\frac{1}{2} \sumM B_m^2 u,\qquad u\in{H},
\end{align*}
is the Stratonovich correction term. \\

Most of our paper will be concerned with the construction of a martingale solution.

\begin{Definition}\label{MartingaleSolutionDef}
	Let $T>0$ and $u_0\in E_A.$ A \emph{martingale solution} of the equation $\eqref{ProblemStratonovich}$ is a system $\left(\tilde{\Omega},\tilde{\F},\tilde{\Prob},\tilde{W},\tilde{\Filtration},u\right)$ consisting of
	\begin{itemize}
		\item a probability space $\left(\tilde{\Omega},\tilde{\F},\tilde{\Prob}\right);$
		\item a $Y$-valued cylindrical Wiener $\tilde{W}$ process on $\tilde{\Omega};$
		\item  a filtration $\tilde{\Filtration}=\left(\tilde{\F}_t\right)_{t\in [0,T]}$ with the usual conditions;
		\item a continuous, $\tilde{\Filtration}$-adapted, $\EAdual$-valued process  such that $u\in L^2(\Omega\times [0,T],\EAdual)$ and almost all paths are in $\weaklyContinousEA$,
	\end{itemize}
	such that the equality
	\begin{align}\label{ItoFormSolution}
	u(t)=  u_0+ \int_0^t \left[-\im A u(s)-\im F(u(s))+\mu(u(s))\right] \df s- \im \int_0^t B u(s) \df \tilde{W}(s)
	\end{align}
	holds almost surely in $\EAdual$ for all $t\in [0,T].$
\end{Definition}

\section{Examples}\label{ExampleSection}

In this section, we consider concrete situations and verify that they are covered by the general framework presented in the last section.

\subsection{The Model Nonlinearit\dela{y}{ies}}

The class of the general nonlinearities from the Assumptions $\ref{nonlinearAssumptions}$ and $\ref{focusing}$ covers the standard focusing and defocusing power nonlinearity.

\begin{Prop}\label{powerTypeNonlinearity}
	Let $\alpha\in(1,\infty)$ be chosen as in Assumption $\ref{spaceAssumptions}.$ Define the following function
	\begin{align*}
	F_{\alpha}^\pm(u):=\pm\vert u\vert^{\alpha-1}u, \qquad \Fhat_\alpha^\pm(u):=\pm\frac{1}{\alpha+1} \norm{u}_\LalphaPlusEins^{\alpha+1},\qquad  u\in \LalphaPlusEins.
	\end{align*}
	Then, $F_{\alpha}^\pm$ satisfies Assumption $\ref{nonlinearAssumptions}$ with antiderivative $\Fhat_\alpha^\pm.$
\end{Prop}

\begin{proof}
	Obviously, $F_{\alpha}^\pm:\LalphaPlusEins \to \LalphaPlusEinsDual$ due to
	\begin{align*}
	\norm{F_{\alpha}^\pm(u)}_\LalphaPlusEinsDual=\norm{u}_\LalphaPlusEins^\alpha,\qquad u\in \LalphaPlusEins.
	\end{align*}	
	Furthermore,
	\begin{align*}
	\Real \duality{\im v}{F_{\alpha}^\pm(v)}=\pm \Real\int_M \im v \vert v\vert^{\alpha-1}\overline{v}\df \mu=\pm \Real\left[\im \norm{v}_\LalphaPlusEins^{\alpha+1}\right]=0.
	\end{align*}
	We can apply the Lemma $\ref{FrechetNonlinear}$ below with $p=\alpha+1$ and
	\begin{align*}
	\varPhi(a,b)=\left(a^2+b^2\right)^\frac{\alpha-1}{2} \left(\begin{array}{c} a \\ b \end{array}\right),\qquad a,b\in\R,
	\end{align*}
	to obtain part ii) and iii) of Assumption $\ref{nonlinearAssumptions}.$
\end{proof}	
The next Lemma contains the differentiablity properties of the nonlinearity. For a proof, we refer to the lecture notes \cite{ISEM}, Lemma 9.1 and Lemma 9.2.
\begin{Lemma}\label{FrechetNonlinear}
	Let $\left(S,\mathcal{A},\mu\right)$ be a measure space and $\alpha>1.$
	\begin{itemize}
		\item[a)] Let $p>1.$ Then, the map $G_1: L^p(S)\to \R $ defined by $G_1(u):= \norm{u}_{L^p(S)}^p$ is continuously Fr\'{e}chet differentiable and for all $u,h\in L^p(S),$ we have
		\begin{align*}
		G_1'[u]h=\Real\int_S \vert u\vert^{p-1}u \overline{h} \df \mu.
		\end{align*}		
		\item[b)] Let $p>\alpha$ and $\varPhi=(\varPhi_1,\varPhi_2)\in C^1(\R^2,\R^2).$ Assume that there is $C>0$ with
		\begin{align*}
		\vert \varPhi(a,b) \vert \le C \left(a^2+b^2\right)^\frac{\alpha}{2}, \qquad
		\vert \varPhi'(a,b) \vert \le C \left(a^2+b^2\right)^\frac{\alpha-1}{2}, \qquad a,b\in \R.
		\end{align*}
		Then, the map
		\begin{align*}
		G: L^p(S) \rightarrow L^{\frac{p}{\alpha}}(S), \hspace{0.7cm}G(u):=\varPhi_1(\Real u,\Ima u)+\im \varPhi_2(\Real u,\Ima u)
		\end{align*}
		is continuously Fr\'{e}chet differentiable and for $u,h\in L^p(S),$ we have
		\begin{align*}
		G'[u]h=\nabla \varPhi_1(\Real u,\Ima u)\cdot \!\left(\begin{array}{c} \Real h \\ \Ima h \end{array}\!\right) +\im \nabla\varPhi_2(\Real u,\Ima u)\cdot \!\left(\begin{array}{c} \Real h \\ \Ima h \end{array}\!\right)
		\end{align*}
		and
		\begin{align*}
		\Vert G'[u]\Vert_{L^p \to L^{\frac{p}{\alpha}}}\le C \Vert u \Vert_{L^p}^{\alpha-1}.
		\end{align*}
		
	\end{itemize}	
\end{Lemma}

\subsection{The Laplace-Beltrami Operator on compact manifolds}\label{Manifold}
In this subsection, we deduce Corollary $\ref{CorollaryManifold}$ from Theorem $\ref{MainTheorem}.$
Let $(M,g)$ be a compact d-dimensional riemannian manifold without boundary and $A:=-\Delta_g$ be the Laplace-Beltrami operator on $M.$

\begin{proof}[Proof of Corollary \ref{CorollaryManifold}]
	\emph{Step 1.} Let $X=M$, $\rho$ be the geodesic distance and $\mu$ be the canonical volume measure on $X$.  From \cite{CoulhonRuss}, Section 4, p.~329, we obtain the local doubling property of $X$, i.e.  there is $C_1>0$ such that for all $x\in X$ and $r\in (0,1)$ we have
	\begin{align}\label{localDoubling}
	\mu(B(x,2r))\le C_1 \mu(B(x,r)).
	\end{align} 
	Dominated convergence implies that the function $f\colon X\times [1,\max\{1,\operatorname{diam}(M)\}]\allowbreak\to (0,\infty)$ defined by 
	\begin{align*}
	f(x,r)=\mu(B(x,r)),\qquad x\in X,\quad r\in [1,\max\{1,\operatorname{diam}(M)\}],
	\end{align*}	
	is continuous. 
	Since $X\times [1,\max\{1,\operatorname{diam}(M)\}]$ is compact, we therefore obtain that 
	\begin{align}\label{intermediateRadiusInf}
	C_2:=\inf_{x\in X, r\in [1,\max\{1,\operatorname{diam}(M)\}]}\mu(B(x,r))>0.
	\end{align}
	In particular, this yields
	\begin{align}\label{intermediateRadius}
	\mu(B(x,2r))\le \frac{\mu(M)}{C_2} \mu(B(x,r))
	\end{align} 
	for every $x\in X$ and $r\in [1,\max\{1,\operatorname{diam}(M)\}]$.
	For $x\in X$ and $r>\operatorname{diam}(M)$, we get
	\begin{align}\label{largeRadius}
	\mu(B(x,2r))=\mu(M)=\mu(B(x,r)).
	\end{align}
	Combining \eqref{localDoubling}, \eqref{intermediateRadius} and \eqref{largeRadius} implies the doubling property \eqref{doubling}.\\
	
	\emph{Step 2.} Let $S:=I-\Delta_g.$ Then, $S$ is selfadjoint, strictly positive and commutes with $A.$ Moreover, $S$ has a compact resolvent and $\D(S^k)\hookrightarrow E_A$ holds for every $k\in\N.$ Furthermore, $S$ has upper Gaussian bounds by \cite{GrigoryanLectureNotes}, Corollary 5.5 and Theorem 6.1, since these results imply
	\begin{align*}
	\vert p(t,x,y)\vert \le \frac{C}{t^{d/2}} e^{-t} \exp \left\{-c \frac{\rho(x,y)^2}{t}\right\}, \qquad t>0,\quad (x,y)\in M\times M
	\end{align*}
	for the kernel $p$ of the semigroup $\left(e^{-tS}\right)_{t\ge 0}.$  This is sufficient for \eqref{GaussianEstimate} since \eqref{doublingDimension} implies 
	\begin{align*}
	\frac{1}{t^{d/2}}\lesssim \frac{\mu(B(x,1))}{\mu(B(x,t^{1/2}))}\le \frac{\mu(M)}{\mu(B(x,t^{1/2}))},\qquad t>0.
	\end{align*}
	In particular, $S$ has generalized Gaussian bounds with $p_0=1$, see Remark \ref{GaussianRemark}. 
	Next note that by Proposition $\ref{PropertiesFractionalSobolev}$ a), the scale of Sobolev spaces on $M$ is given by
	\begin{align*}
	H^s(M)=R\left(S^{-\frac{s}{2}}\right)=\D \left(S^{\frac{s}{2}}\right)=\D \left((-\Delta_g)^{\frac{s}{2}}\right), \quad s> 0,
	\end{align*}
	where the last identity can be deduced from the spectral theorem and $(1+\lambda)^s\eqsim_s 1+\lambda^s.$
	\dela{Especially}{In particular}, we have $E_A=H^1(M).$	
	Let $1<\alpha <1+\frac{4}{(d-2)_+}.$  Then, by Proposition $\ref{PropertiesFractionalSobolev}$ c) and Lemma $\ref{spaceLemma},$ the embeddings
	\begin{align*}
	E_A=H^1(M)\hookrightarrow H^{-1}(M)=E_A^*, \qquad E_A=H^1(M)\hookrightarrow \LalphaPlusEins
	\end{align*}
	are compact. Hence, Assumption \ref{spaceAssumptions} holds with our choice of $A$ and $S.$\\
	
	\emph{Step 3.} In view of Proposition \ref{powerTypeNonlinearity}, Assumption \ref{nonlinearAssumptions} holds. Next, we check Assumption \ref{focusing}. 	Obviously, $F_\alpha^+$  fulfills i) for $ \alpha \in \left(1,1+\frac{4}{(d-2)_+}\right)$. Let us consider $F_\alpha^-$ for $ \alpha \in \left(1,1+\frac{4}{d}\right).$
	
	\emph{Case 1}. Let $d\ge 3.$ Then, $p_{\max}:=\frac{2d}{d-2}$ is the maximal exponent with $H^1(M)\hookrightarrow L^{p_{\max}}(M).$
	Since $\alpha \in (1,p_{\max}-1),$  we can interpolate $\LalphaPlusEins$ between ${H}$ and $L^{p_{\max}}(M)$ and get
	\begin{align*}
	\norm{u}_\LalphaPlusEins\le \norm{u}_{L^2}^{1-\theta} \norm{u}_{L^{p_{\max}}(M)}^\theta
	\lesssim \norm{u}_{L^2}^{1-\theta} \norm{u}_{H^1(M)}^\theta.
	\end{align*}
	with $\theta=\frac{d(\alpha-1)}{2(\alpha+1)}\in (0,1).$
	The restriction $\beta_2:=\theta (\alpha+1)<2$ from Assumption $\ref{focusing}$ i') is equivalent to $\alpha <1+\frac{4}{d}.$\\	
	\emph{Case 2}.
	In the case $d=2,$  Assumption i') is guaranteed for $\alpha \in (1,3).$ To see this, take $p>\frac{4}{3-\alpha}$ which is equivalent to $\theta (\alpha+1)<2$ when $\theta\in (0,1)$ is chosen as
	\begin{align*}
	\theta=\frac{(\alpha-1)p}{(\alpha+1) (p-2)}.
	\end{align*}
	We have $H^1(M)\hookrightarrow L^p(M)$ and as above, interpolation between ${H}$ and $L^p(M)$ yields
	\begin{align*}
	\norm{u}_\LalphaPlusEins^{\alpha+1}\lesssim \norm{u}_{L^2}^{(\alpha+1)(1-\theta)}  \norm{u}_{\EA}^{(\alpha+1)\theta}.
	\end{align*}	
	\emph{Case 3}. Let $d=1$ and fix $\varepsilon\in (0,\frac{1}{2}).$ Proposition $\ref{PropertiesFractionalSobolev}$ yields
	\begin{align*}
	H^{\frac{1}{2}+\varepsilon}(M)\hookrightarrow \LInfty,\qquad H^{\frac{1}{2}+\varepsilon}(M)=\left[L^2(M),H^1(M)\right]_{\frac{1}{2}+\varepsilon}.
	\end{align*}  Hence,
	\begin{align*}
	\norm{v}_{L^{\alpha+1}}^{\alpha+1}
	\le \norm{v}_{L^2}^2\norm{v}_{L^\infty}^{\alpha-1}
	\lesssim \norm{v}_{L^2}^2\norm{v}_{H^{\frac{1}{2}+\varepsilon}}^{\alpha-1}
	\lesssim\norm{v}_{L^2}^{2+(\frac{1}{2}-\varepsilon)(\alpha-1)}\norm{v}_{H^1}^{(\frac{1}{2}+\varepsilon)(\alpha-1)}.
	\end{align*}
	The condition $(\frac{1}{2}+\varepsilon)(\alpha-1)<2$ is equivalent to $\alpha<1+\frac{4}{1+2\varepsilon}.$ Choosing $\varepsilon$ small enough, we see that Assumption $\ref{focusing}$ i') is true for $\alpha\in (1,5).$\\
	
	\emph{Step 4.} The Steps 1-3 and Theorem \ref{MainTheorem} complete the proof of Corollary \ref{CorollaryManifold}.
\end{proof}

\begin{Remark}
	Note, that the 3-dimensional case with a cubic defocusing nonlinearity, i.e.
	\begin{align*}
	d=\alpha=3,\quad F(u)=F_3^+(u)=\vert u\vert^2 u
	\end{align*}
	is admissible in our framework. In the deterministic setting, i.e. $B=0,$ a global unique weak solution to this problem in $H^1(M)$ was constructed in \cite{Burq}, Theorem 3. Uniqueness in the stochastic case will be proved in a forthcoming paper.
	In \cite{BrzezniakStrichartz}, the authors considered the stochastic problem, but only obtained global solutions in the 2-dimensional case.
\end{Remark}

\subsection{Laplacians on bounded domains}\label{Domain}
%
%
%

We can apply Theorem $\ref{MainTheorem}$ to the stochastic NLSE on bounded domains.


\begin{Korollar}\label{CorollaryDomain}
	Let $M\subset \Rd$ be a bounded domain and $\Delta$ be the Laplacian with Dirichlet or Neumann boundary conditions. In the Neumann case, we assume that $\partial M$ is Lipschitz. Under assumption $\ref{stochasticAssumptions}$ and either i) or ii)
	\begin{itemize}
		\item[i)] $F(u)= \vert u\vert^{\alpha-1}u$ with $ \alpha \in \left(1,1+\frac{4}{(d-2)_+}\right)$,
		\item[ii)]$F(u)= -\vert u\vert^{\alpha-1}u$ with $ \alpha \in \left(1,1+\frac{4}{d}\right),$
	\end{itemize}
	the equation
	\begin{equation}
	\label{ProblemStratonovichDomain}
	\left\{
	\begin{aligned}
	\df u(t)&= \left(\im \Delta u(t)-\im  F(u(t)\right) dt-\im B u(t) \circ \df W(t)\quad \text{in} \hspace{0.1cm} H^1(M),\\
	u(0)&=u_0\in H^1(M),
	\end{aligned}\right.
	\end{equation}
	has a martingale solution which satisfies
	\begin{align*}
	u\in L^q(\tilde{\Omega},L^\infty(0,T;H^1(M)))
	\end{align*}
	for all $q\in [1,\infty).$
\end{Korollar}

We remark, that one could consider uniformly elliptic operators and more general boundary conditions, but for the sake of simplicity, we concentrate on the present two examples.

\begin{proof} In the setting of the second section, we choose $X=\Rd.$ Hence, the doubling property is fulfilled.
	We consider  the Dirichlet form
	$a_V: V \times V \to \C$ ,
	\begin{align*}
	a_V(u,v)=\int_M \nabla u \cdot \nabla v \df x, \quad u,v\in V,
	\end{align*}
	with associated operator $\left(A_V,\D(A_V)\right)$ in the following two situations:
	\begin{itemize}
		\item[i)] $V=H^1_0(M)$
		\item[ii)] $V=H^1(M)$ and $M$ has Lipschitz-boundary.
	\end{itemize}
	The operator $A_{H^1_0(M)}=\Delta_D$ is the Dirichlet Laplacian and $A_{H^1(M)}=\Delta_N$ is the Neumann Laplacian.
	In both cases, $V=E_{A_V}$ by the square root property (see \cite{OuhabazHeatEquations}, Theorem 8.1) and the  embedding $E_{A_V}\hookrightarrow \LalphaPlusEins$ is compact iff $1<\alpha <p_{\max}-1$ with $p_{\max}:=2+\frac{4}{(d-2)_+}.$ Hence, we obtain the same range of admissible powers $\alpha$ for the focusing and the defocusing nonlinearity as in the case of the Riemannian manifold without boundary. \\
	
	In the Dirichlet case, we choose $S:=A=-\Delta_D,$ which is a strictly positive operator and  \cite{OuhabazHeatEquations}, Theorem 6.10, yields the Gaussian estimate for the associated semigroup. Hence, we can directly apply Theorem $\ref{MainTheorem}$ to construct a martingale solution of problem $\eqref{ProblemStratonovichDomain}.$ \\
	
	In the Neumann case, we have $0\in \sigma(\Delta_N)$ and the kernel of the semigroup $\left(e^{-t\Delta_N}\right)_{t\ge 0}$ only satisfies the estimate
	\begin{align*}
	\vert p(t,x,y)\vert \le \frac{C_\varepsilon}{\mu(B(x,t^\frac{1}{m}))} e^{\varepsilon t}\exp \left\{-c \left(\frac{\rho(x,y)^m}{t}\right)^{\frac{1}{m-1}}\right\}
	\end{align*}
	for all $t>0$ and almost all $(x,y)\in M\times M$  with an arbitrary $\varepsilon>0,$ see \cite{OuhabazHeatEquations}, Theorem 6.10. In order to get a strictly positive operator with the Gaussian bound from Remark \ref{GaussianRemark}, we fix $\varepsilon>0$ and choose  $S:=\varepsilon I -\Delta_N.$
	Finally, the computation of the admissible range of exponents $\alpha$ in the focusing case is similar to the third step of the proof of Corollary \ref{CorollaryManifold}.
\end{proof}

\subsection{The fractional NLSE}
In this subsection, we show how the range of admissible nonlinearities change when the Laplacians in the previous examples are replaced by their fractional powers $\left(-\Delta\right)^\beta$ for $\beta>0.$
Exemplary, we treat the case of a compact riemannian manifold without boundary. Similar results are also true for the Dirichlet and the Neumann Laplacian on a bounded domain. Let us point out that there exists a huge literature on the subject of fractional NLSE apparently starting with a paper \cite{Laskin_2000} by N Laskin. \\
In the setting of Section $\ref{Manifold},$ we look at the fractional Laplace-Beltrami operator given by
$A:=\left(-\Delta_g\right)^\beta$ for $\beta>0,$ which is also a selfadjoint positive operator by the functional calculus and once again, we choose $S:=I-\Delta_g.$
We apply Theorem $\ref{MainTheorem}$ with
\begin{align*}
\EA=\D(\sqrtA)=\D\left(\left(I-\Delta_g\right)^\frac{\beta}{2}\right)=H^\beta(M),
\end{align*}
see Proposition $\ref{PropertiesFractionalSobolev}$ a).  Note that $\D(S^k)\hookrightarrow E_A$ holds for every $k\in\N$ with $k\ge \frac{\beta}{2}.$ The range of admissible pairs $(\alpha,\beta)$ in the defocusing case is given by
\begin{align*}
\beta>\frac{d}{2}-\frac{d}{\alpha+1} \quad \Leftrightarrow \quad\alpha \in \left(1,1+\frac{4\beta}{(d-2\beta)_+}\right),
\end{align*}
since this is exactly the range of $\alpha$ and $\beta$ with a compact embedding $E_A \hookrightarrow \LalphaPlusEins$ (see  Proposition $\ref{PropertiesFractionalSobolev}$ c)).
In the focusing case, analogous calculations as in the third step of the proof of Corollary \ref{CorollaryManifold}  (with the distinction of $\beta> \frac{d}{2},$ $\beta=\frac{d}{2}$ and $\beta<\frac{d}{2}$) imply that the range of exponents reduces to
\begin{align*}
\alpha\in \left(1,1+\frac{4\beta}{d}\right).
\end{align*}
Hence, we get the following Corollary.

\begin{Korollar}\label{CorollaryManifoldFractional}
	Let $(M,g)$ be a compact $d$-dimensional riemannian manifold without boundary, $\beta>0$ and \\$u_0\in H^\beta(M).$ Under assumption $\ref{stochasticAssumptions}$ and either i) or ii)
	\begin{itemize}
		\item[i)] $F(u)= \vert u\vert^{\alpha-1}u$ with $ \alpha \in \left(1,1+\frac{4\beta}{(d-2\beta)_+}\right)$,
		\item[ii)]$F(u)= -\vert u\vert^{\alpha-1}u$ with $ \alpha \in \left(1,1+\frac{4\beta}{d}\right),$
	\end{itemize}
	the equation
	\begin{equation}
	\label{ProblemStratonovichManifoldFractional}
	\left\{
	\begin{aligned}
	\df u(t)&= \left(-\im \left(-\Delta_g\right)^\beta u(t)-\im  F(u(t)\right) dt-\im B u(t) \circ \df W(t),\quad t>0,\\
	u(0)&=u_0\in H^\beta(M),
	\end{aligned}\right.
	\end{equation}
	has a martingale solution $\left(\tilde{\Omega},\tilde{\F},\tilde{\Prob},\tilde{W},\tilde{\Filtration},u\right)$ in $ H^\beta(M)$
	with
	\begin{align}\label{propertySolutionFractional}
	u\in L^q(\tilde{\Omega},L^\infty(0,T;H^\beta(M)))
	\end{align}
	for all $q\in [1,\infty).$
\end{Korollar}

\subsection{The Model Noise}\label{NoiseSection}
In Corollaries $\ref{CorollaryManifold}$ and $\ref{CorollaryDomain},$ we considered the general linear noise from Assumption $\ref{stochasticAssumptions}.$
If $M$ is either a compact riemannian manifold or a bounded domain, let us consider the following example.
Let  $\left(B_m\right)_{m\in\N}$  the multiplication operators given by
\begin{align*}
B_m u:= e_m u
\end{align*}
for $u\in {H}$ with real valued functions $e_m,$ $m\in \N,$ that satisfy
\begin{align}\label{assumptionEm}
e_m \in F:=\begin{cases}
H^{1,d}(M) \cap \LInfty, &  d\ge 3,\\
H^{1,q}(M),& d=2,\\
H^{1}(M),& d=1,\\
\end{cases}
\end{align}
for some $q>2$ in the case $d=2.$
Moreover, we assume
\begin{align*}
\sumM\norm{e_m}_{F}^2<\infty,
\end{align*}
We get
\begin{align*}
\norm{e_m u}_{L^p}\le \norm{e_m}_\LInfty \norm{u}_{L^p},\qquad u\in L^p(M),
\end{align*}
for $p\in [1,\infty].$ First, let $d\ge 3.$
The Sobolev embedding $H^1(M) \hookrightarrow L^{p_{\max}}(M)$ for $p_{\max}=\frac{2d}{d-2}$ and the H\"older inequality with $\frac{1}{2}=\frac{1}{d}+\frac{1}{p_{\max}}$ yield
\begin{align*}
\norm{\nabla \left(e_m u\right)}_{L^2} \le& \norm{ u \nabla e_m }_{L^2}+ \norm{ e_m \nabla u }_{L^2}
\le \norm{\nabla e_m}_{L^d} \norm{u}_{L^{p_{\max}}}+\norm{e_m}_\LInfty \norm{\nabla u}_{L^2}\\
\lesssim& \left(\norm{\nabla e_m}_{L^d}+\norm{e_m}_\LInfty\right) \norm{u}_{H^1},\qquad u\in H^1(M).
\end{align*}
Now, let $d=2$ and $q>2$ as in \eqref{assumptionEm}. Then, we have $F \hookrightarrow  \LInfty.$
Furthermore, we choose $p>2$ according to $\frac{1}{2}=\frac{1}{q}+\frac{1}{p}$  and observe $H^1(M)\hookrightarrow L^p(M).$   As above, we obtain
\begin{align*}
\norm{\nabla \left(e_m u\right)}_{L^2}
&\lesssim \left(\norm{\nabla e_m}_{L^q}+\norm{e_m}_\LInfty\right) \norm{u}_{H^1}\lesssim \norm{e_m}_{H^{1,q}}\norm{u}_{H^1},\qquad u\in H^1(M).
\end{align*}
Hence, we conclude in both cases
\begin{align*}
\norm{e_m u}_{H^1}\lesssim \norm{e_m}_F \norm{u}_{H^1},\qquad m\in\N,\quad u\in H^1(M).
\end{align*}
For $d=1,$ this inequality directly follows from the embedding $H^1(M) \hookrightarrow \LInfty.$ Therefore, we obtain
\begin{align*}
\sumM \norm{B_m}_{\mathcal{L}(\EA)}^2<\infty.
\end{align*}
for arbitrary dimension $d.$
The  properties of $B_m$ as operator in $\mathcal{L}(\LalphaPlusEins)$ and in $\mathcal{L}(L^2(M))$ can be deduced from the embedding $F\hookrightarrow \LInfty.$

We close this section by remarks on natural generalizations of the linear, conservative noise considered in this paper. The details have been worked out in the second author's dissertation \cite{FhornungDiss}.

\begin{Remark}
	As in \cite{BrzezniakStrichartz}, Section 8,
	it is possible to replace the linear Stratonovich noise in Theorem  \ref{MainTheorem}, see also Assumption \ref{stochasticAssumptions}, by a nonlinear one of the form
	\begin{align*}
	B_m(u):= -\im B_m\left( g(\vert u\vert^2)u\right),\qquad \mu(u):=-\frac{1}{2}\sumM B_m^2 \left(g(\vert u\vert^2)^2 u\right),
	\end{align*}
	where we assume the Lipschitz and linear growth conditions
	\begin{align*}
	\norm{g(\vert u\vert^2)^j u}_\EA\lesssim \norm{u}_\EA,\qquad \norm{g(\vert u\vert^2)^j u}_{L^p}\lesssim \norm{u}_{L^p},\qquad \norm{g(\vert u\vert^2)^j u-g(\vert v\vert^2)^jv}_{L^p}\lesssim \norm{u-v}_{L^p}
	\end{align*}
	for $j\in \left\{1,2\right\}$ and $p\in \left\{\alpha+1,2\right\}.$ In the case of $H^1$-based energy spaces, i.e. the $A=-\Delta$ on a bounded domain or $A=-\Delta_g$ on a riemannian manifold, one can take
	$g\in C^2([0,\infty),\R)$ which satisfies the following conditions:
	\begin{align}\label{assumptionSaturatedNonlinearity}
	\sup_{r>0} \vert g(r)\vert<\infty,\qquad \sup_{r>0} (1+r)\vert g'(r)\vert<\infty,\qquad  \sup_{r>0} (1+r^\frac{3}{2})\vert g''(r)\vert<\infty.
	\end{align}
	This kind of nonlinearity is often called saturated and typical examples are given by
	\begin{align*}
	g_1(r)=\frac{r}{1+\sigma r},\qquad g_2(r)=\frac{r(2+\sigma r)}{(1+\sigma r)^2},\qquad g_3(r)=\frac{\log(1+\sigma r)}{1+\log(1+\sigma r)},\qquad r\in [0,\infty),
	\end{align*}
	for a constant $\sigma>0.$
	For the Galerkin equation, we then take
	\begin{equation*}
	\left\{
	\begin{aligned}\label{GalerkinNonlinearRemark}
	\df u_n&= \left(-\im A u_n-\im P_n F\left( u_n\right) -\frac{1}{2}\sumM S_n B_m^2\left(g(\vert u_n\vert^2)^2 u_n\right)\right) \df t-\im \sumM S_n B_m( g(\vert u_n\vert^2) u_n)  \df \beta_m,
	\\
	u_n(0)&=P_n u_0.
	\end{aligned}\right.
	\end{equation*}
	Unfortunately, this approximation does not respect mass conservation, but one still has
	\begin{align}\label{massEstimateNonConservative}
	\sup_{n\in\N}\E \left[\sup_{t\in [0,T]}\norm{u_n(t)}_H^2\right]\lesssim 1,
	\end{align}
	which is enough for our purpose.
\end{Remark}

\begin{Remark}
	Another possible generalization of the noise is to drop the assumption that $B_m,$ $m\in\N,$ is selfadjoint. Then, the correction term $\mu$ has the form
	\begin{align*}
	\mu(u):= -\frac{1}{2}\sumM B_m^* B_m u.
	\end{align*}
	This kind of noise is called non-conservative and was considered in \cite{BarbuH1} and \cite{FHornung}. The existence result is then based on the approximation
	\begin{equation*}
	\left\{
	\begin{aligned}\label{GalerkinNonConservativeRemark}
	\df u_n&= \left(-\im A u_n-\im P_n F\left( u_n\right) -\frac{1}{2}\sumM S_n B_m^* B_m u_n\right) \df t-\im \sumM S_n B_m u_n  \df \beta_m,
	\\
	u_n(0)&=P_n u_0,
	\end{aligned}\right.
	\end{equation*}
	and the a priori estimates as well as the convergence results can be proved analogously. We only have to replace mass conservation by the estimate $\eqref{massEstimateNonConservative}.$ The uniqueness result in section 7, however, only holds for selfadjoint $B_m,$ since this is the crucial assumption in Lemma \ref{SolutionDifferenceLemma}.
\end{Remark}

\section{Compactness and Tightness Criteria}	\label{CompactnessSection}

This section is devoted to the compactness results which will be used to get  a martingale solution of $\eqref{ProblemStratonovich}$ by the Faedo-Galerkin method.

Let $A$ and $\alpha>1$ be chosen according to Assumption $\ref{spaceAssumptions}.$ We recall that the energy space $\EA$ is defined by $\EA:=\D(\sqrtA).$
We start with a criterion for convergence of a sequence in $\stetigBall,$ where the ball $\mathbb{B}_{\EA}^r$ is equipped with the weak topology.

\begin{Lemma}\label{convergenceStetigBall}
	Let $r>0$ and $\left(u_n\right)_{n\in\N} \subset \LinftyEA$ be a sequence with the properties
	\begin{enumerate}
		\item[a)] $\sup_{n\in\N} \norm{u_n}_\LinftyEA \le r$,
		\item[b)] $u_n\to u$ in $\StetigEAdual$ for $n\to \infty.$
	\end{enumerate}
	Then $u_n,u\in \stetigBall$ for all $n\in\N$ and $u_n \to u$ in $\stetigBall$ for $n\to \infty.$
\end{Lemma}

\begin{proof}
	The Strauss-Lemma $\ref{StraussLemma}$ and the assumptions guarantee that
	\begin{align*}
	u_n \in \StetigEAdual \cap \LinftyEA \subset \weaklyContinousEA
	\end{align*}
	for all $n\in \N$ and $\sup_{t\in [0,T]} \norm{u_n(t)}_\EA \le  r.$ Hence, we infer that $u_n\in \stetigBall$ for all $n\in \N.$ For $h\in \EA$
	\begin{align*}
	\sup_{s\in[0,T]} \left\vert \duality{u_n(s)-u(s)}{h}\right\vert \le \norm{u_n-u}_{C([0,T],\EAdual)} \norm{h}_\EA \to 0,\qquad n\to \infty.
	\end{align*}
	By a) and Banach-Alaoglu, we get a subsequence $\left(u_{n_k}\right)_{k\in\N}$ and $v\in \LinftyEA$ with $u_{n_k} \rightharpoonup^* v$ in $\LinftyEA$ and by the uniqueness of the weak star limit in  $L^\infty(0,T;\EAdual),$ we conclude  $u=v \in \LinftyEA$ with $\norm{u}_\LinftyEA\le r.$\\	
	Let $\varepsilon>0$ and $h\in \EAdual.$ By the density of $\EA$ in $\EAdual,$ we choose $h_\varepsilon\in \EA$ with $\norm{h-h_\varepsilon}_\EAdual\le \frac{\varepsilon}{4 r}$ and obtain for large $n\in \N$
	\begin{align*}
	\left\vert \duality{u_n(s)-u(s)}{h}\right\vert&\le
	\left\vert \duality{u_n(s)-u(s)}{h-h_\varepsilon}\right\vert+
	\left\vert \duality{u_n(s)-u(s)}{h_\varepsilon}\right\vert\\
	&\le \norm{u_n(s)-u(s)}_\EA \norm{h-h_\varepsilon}_\EAdual+ \left\vert \duality{u_n(s)-u(s)}{h_\varepsilon}\right\vert\\
	&\le 2 r \frac{\varepsilon}{4 r}+\frac{\varepsilon}{2}=\varepsilon
	\end{align*}
	independent of $s\in[0,T].$ This implies
	$\sup_{s\in[0,T]} \left\vert \duality{u_n(s)-u(s)}{h}\right\vert \to 0$
	for $n\to \infty$ and all $h\in  \EAdual,$ i.e. $u_n \to u$ in $\weaklyContinousEA.$ By Lemma \ref{ConvergenceInStetigBall}, we obtain the assertion.
\end{proof}

We define a  Banach space $\tilde{Z}_T$ by
\begin{align*}
\tilde{Z}_T:=\StetigEAdual\cap \LalphaPlusEinsAlphaPlusEins
\end{align*}
and a locally convex space $Z_T$ by
\begin{align*}
Z_T:=\tilde{Z}_T \cap \weaklyContinousEA.
\end{align*}
The latter is equipped  with the Borel $\sigma$-algebra, i.e. the $\sigma$-algebra generated by the open sets in the locally convex topology of $Z_T.$
In the next Proposition, we give a criterion for compactness in $Z_T.$

\begin{Prop}\label{CompactnessDeterministic}
	Let $K$ be a subset of $Z_T$ and $r>0$ such that
	\begin{enumerate}
		\item[a)] $
		\sup_{u\in K} \Vert u\Vert_{\LinftyEA}\le r ;
		$
		\item[b)] $K$ is equicontinuous in $\StetigEAdual,$ i.e.
		\begin{align*}
		\lim_{\delta \to 0} \sup_{u\in K} \sup_{\vert t-s\vert\le \delta} \Vert u(t)-u(s)\Vert_{\EAdual}=0.
		\end{align*}
	\end{enumerate}
	Then, $K$ is relatively compact in $Z_T.$
\end{Prop}

\begin{proof}
	Let $K$ be a subset of $Z_T$ such that the assumptions $a)$ and $b)$ are fullfilled and $\left(z_n\right)_{n\in\N}\subset K.$ We want to construct a subsequence converging in $\LalphaPlusEinsAlphaPlusEins,$ $\StetigEAdual$ and $\weaklyContinousEA.$\\
	
	\emph{Step 1:} By $a),$ we can choose a constant $C>0$ and for each $n\in\N$ a  null set $I_n$ with
	$\Vert z_n(t)\Vert_\EA \le C$
	for all $t\in [0,T]\setminus I_n.$ The set $I:= \bigcup_{n\in\N} I_n$ is also a nullset and for each $t\in [0,T]\setminus I,$ the sequence $\left(z_n(t)\right)_{n\in\N}$ is bounded in $\EA.$ \\
	Let $\left(t_j\right)_{j\in\N}\subset [0,T]\setminus I$ be a sequence, which is dense in $[0,T].$ By Lemma $\ref{spaceLemma},$ the embedding $\EA\hookrightarrow {H}$ is compact, which yields that $\EA \hookrightarrow \EAdual$ is also compact. Therefore, we can choose for each $j\in\N$ a Cauchy subsequence in $\EAdual$ again denoted by $\left(z_n(t_j)\right)_{n\in\N}.$ By a diagonalisation argument, one obtains a common Cauchy subsequence $\left(z_{n}(t_j)\right)_{n\in\N}.$ \\
	Let $\varepsilon>0.$ Assumption $b)$ yields $\delta>0$ with
	\begin{align}\label{EquicontArzelaAscoli}
	\sup_{u\in K} \sup_{\vert t-s\vert\le \delta} \Vert u(t)-u(s)\Vert_\EAdual\le \frac{\varepsilon}{3}.
	\end{align}
	%
	Let us choose finitely many open balls $U_\delta^1,\dots, U_\delta^L$ of radius $\delta$ covering $[0,T].$ By density, each of these balls contains an element of the sequence $\left(t_j\right)_{j\in\N},$ say $t_{j_l}\in U_\delta^l$ for $l\in \left\{1,\dots, L\right\}.$ In particular, the sequence $\left(z_{n}(t_{j_l})\right)_{n\in\N}$ is Cauchy for all $l\in \left\{1,\dots, L\right\}.$
	Hence,
	\begin{align}\label{CauchyArzelaAscoli}
	\Vert z_n(t_{j_l})-z_m(t_{j_l})\Vert_\EAdual \le \frac{\varepsilon}{3},\qquad l=1,\dots,L,
	\end{align}
	if we choose $m,n\in \N$ sufficiently large.
	Now, we fix $t\in[0,T]$ and take $l\in \{1,\dots, L\}$ with $\vert t_{j_l}-t\vert\le \delta.$ 
	We  use  \eqref{EquicontArzelaAscoli} and \eqref{CauchyArzelaAscoli} to get
	\begin{align}\label{Arzela}
	\Vert z_n(t)-z_m(t)\Vert_\EAdual \le& \Vert z_n(t)-z_n(t_{j_l})\Vert_\EAdual+\Vert z_n(t_{j_l})-z_m(t_{j_l})\Vert_\EAdual+\Vert z_m(t_{j_l})-z_m(t)\Vert_\EAdual\le \varepsilon.
	\end{align}
	This means that $\left(z_n\right)_{n\in\N}$ is a Cauchy sequence in $\StetigEAdual$ since the estimate \eqref{Arzela} is uniform in $t\in [0,T].$ \\

	\emph{Step 2:}
	The first step yields $z\in \StetigEAdual$ with $z_n \to z$ in $\StetigEAdual$ for $n\to \infty$
	and assumption a) implies, that there is $r>0$ with
	$\sup_{n\in\N} \norm{z_n}_\LinftyEA\le r.$\\
	Therefore, we obtain $z\in\stetigBall$ and $z_n\to z$ in $\stetigBall$ for $n\to \infty$  by Lemma $\ref{convergenceStetigBall}.$ Hence, $z_n\to z$ in $\weaklyContinousEA.$\\

	\emph{Step 3:} We fix again $\varepsilon>0.$ By \dela{Lions'}{the Lions} Lemma $\ref{LionsLemma}$ with  $X_0=\EA,$ $X=\LalphaPlusEins,$ \\$X_1=\EAdual,$ $p=\alpha+1$ and $\varepsilon_0=\frac{\varepsilon}{2 T (2C)^{\alpha+1}}$ we get
	\begin{align}\label{LionsEstimate}
	\Vert v \Vert_\LalphaPlusEins^{\alpha+1} \le \varepsilon_0 \Vert v\Vert_{\EA}^{\alpha+1}+C_{\varepsilon_0} \Vert v\Vert_{\EAdual}^{\alpha+1}
	\end{align}
	for all $v\in\EA.$ The first step allows us to choose $n,m\in\N$ large enough that
	\begin{align*}
	\Vert z_n-z_m\Vert_{C([0,T],\EAdual)}^{\alpha+1}\le \frac{\varepsilon}{2 C_{\varepsilon_0}T}
	\end{align*} The special choice $v=z_n(t)-z_m(t)$ for $t\in [0,T]$ in $\eqref{LionsEstimate}$ and integration with respect to time yields
	\begin{align*}
	\Vert z_n-z_m \Vert_\LalphaPlusEinsAlphaPlusEins^{\alpha+1} &\le {\varepsilon_0} \Vert z_n-z_m\Vert_{L^{\alpha+1}(0,T;\EA)}^{\alpha+1}+C_{\varepsilon_0} \Vert z_n-z_m\Vert_{L^{\alpha+1}(0,T;\EAdual)}^{\alpha+1}\\
	&\le  {\varepsilon_0} T \Vert z_n-z_m\Vert_{L^\infty(0,T;\EA)}^{\alpha+1}+C_{\varepsilon_0} T\Vert z_n-z_m\Vert_{C([0,T],\EAdual)}^{\alpha+1}\\
	&\le  {\varepsilon_0} T \left(2 C\right)^{\alpha+1}+C_{\varepsilon_0} T\Vert z_n-z_m\Vert_{C([0,T],\EAdual)}^{\alpha+1}\\
	&\le \frac{\varepsilon}{2}+\frac{\varepsilon}{2}=\varepsilon.
	\end{align*}
	Hence, the sequence $\left(z_n\right)_{n\in\N}$ is also Cauchy in $\LalphaPlusEinsAlphaPlusEins$.
\end{proof}

In the following, we want to obtain a criterion for tightness in $Z_T.$ Therefore, we introduce the Aldous condition.

\begin{Definition}\label{DefinitionAldous}
	Let $(X_n)_{n\in\N}$ be a sequence of  stochastic processes in a Banach space $E.$ Assume that for every $\varepsilon>0$ and $\eta>0$ there is $\delta>0$ such that for every sequence $(\tau_n)_{n\in\N}$ of $[0,T]$-valued stopping times one has
	\begin{align*}
	\sup_{n\in\N} \sup_{0<\theta \le \delta} \Prob \left\{ \Vert X_n((\tau_n+\theta)\land T)-X_n(\tau_n)\Vert_E\ge \eta \right\}\le \varepsilon.
	\end{align*}
	In this case, we say that $(X_n)_{n\in\N}$
	satisfies the Aldous condition $[A].$
\end{Definition}

The following Lemma (see \cite{Motyl}, Lemma A.7) gives us a useful consequence of the Aldous condition $[A].$

\begin{Lemma} \label{AldousLemma}			
	Let $(X_n)_{n\in\N}$ be a sequence of continuous stochastic processes in a Banach space $E,$ which satisfies the Aldous condition $[A].$ Then, for every $\varepsilon>0$ there exists a measurable subset $A_\varepsilon \subset C([0,T],E)$ such that
	\begin{align*}
	\Prob^{X_n}(A_\varepsilon)\ge 1-\varepsilon,\qquad
	\lim_{\delta\to 0} \sup_{u\in A_\varepsilon} \sup_{\vert t-s\vert\le \delta} \Vert u(t)-u(s)\Vert_E=0.
	\end{align*}
\end{Lemma}

The deterministic compactness result in Proposition $\ref{CompactnessDeterministic}$ and the last Lemma can be used to get the following criterion for tightness in $Z_T.$

\begin{Prop}\label{TightnessCriterion}
	Let $(X_n)_{n\in\N}$ be a sequence of continuous adapted $\EAdual$-valued processes satisfying the Aldous condition $[A]$ in $\EAdual$ and
	\begin{align*}
	\sup_{n\in\N} \E \left[\Vert X_n\Vert_\LinftyEA^2\right] <\infty.
	\end{align*}
	Then the sequence $\left({\Prob}^{X_n}\right)_{n\in\N}$ is tight in $Z_T,$ i.e. for every $\varepsilon>0$ there is a compact set $K_\varepsilon\subset Z_T$ with
	\begin{align*}
	\Prob^{X_n}(K_\varepsilon)\ge 1- \varepsilon
	\end{align*}
	for all $n\in\N.$
\end{Prop}

\begin{proof}
	Let $\varepsilon>0.$ With $R_1:= \left(\frac{2}{\varepsilon} \sup_{n\in\N} \E \left[ \Vert X_n\Vert_\LinftyEA^2\right]\right)^{\frac{1}{2}},$ we obtain
	\begin{align*}
	\Prob\left\{ \Vert X_n\Vert_\LinftyEA> R_1\right\}\le \frac{1}{R_1^2}\E \left[\Vert X_n\Vert_\LinftyEA^2\right]\le \frac{\varepsilon}{2}.
	\end{align*}
	By Lemma $\ref{AldousLemma}$, one can use the Aldous condition $[A]$ to get a Borel subset $A$ of $\StetigEAdual$ with
	\begin{align*}
	\Prob^{X_n}\left(A\right)\ge 1-\frac{\varepsilon}{2},\quad n\in\N, \qquad \lim_{\delta\to 0} \sup_{u\in A} \sup_{\vert t-s\vert\le \delta} \Vert u(t)-u(s)\Vert_\EAdual=0.
	\end{align*}
	We define $K:= \overline{A\cap B}$ where $B:= \left\{u\in Z_T: \Vert u\Vert_\LinftyEA\le R_1 \right\}.$ This set $K$ is compact by Proposition $\ref{CompactnessDeterministic}$ and we can estimate
	\begin{align*}
	\Prob^{X_n}(K)\ge \Prob^{X_n}\left(A\cap B\right)\ge \Prob^{X_n}\left(A\right)-\Prob^{X_n}\left( B^c\right)\ge 1-\frac{\varepsilon}{2}-\frac{\varepsilon}{2}=1-\varepsilon
	\end{align*}
	for all $n\in \N.$
\end{proof}

In metric spaces, one can apply Prokhorov Theorem (see \cite{parthasaraty}, Theorem II.6.7) and Skorohod Theorem (see \cite{Billingsley}, Theorem 6.7.) to obtain convergence from tightness. Since the space $Z_T$ is a locally convex space, we use the  following generalization to nonmetric spaces.

\begin{Prop}[Skorohod-Jakubowski]\label{SkohorodJakubowski}
	Let $\mathcal{X}$ be a topological space such that there is a sequence of continuous functions $f_m: \mathcal{X}\to \C$ that separates points of $\mathcal{X}.$ Let $\mathcal{A}$ be the $\sigma$-algebra  generated by $\left(f_m\right)_m.$ Then, we have the following assertions:
	\begin{enumerate}
		\item[a)] Every compact set $K\subset \mathcal{X}$ is metrizable.
		\item[b)] Let $\left(\mu_n\right)_{n\in\N}$ be a tight sequence of probability measures on $\left(\mathcal{X}, \mathcal{A}\right).$
		Then, there are a subsequence $\left(\mu_{n_k}\right)_{k\in\N},$ random variables $X_k,$ $X$ for $k\in\N$ on a common probability space $(\tilde{\Omega},\tilde{\Filtration},\tildeProb)$ with $\tildeProb^{X_k}=\mu_{n_k}$ for $k\in\N,$ and $X_k \to X$ $\tildeProb$-almost surely for $k\to \infty.$
	\end{enumerate}
\end{Prop}

We stated Proposition $\ref{SkohorodJakubowski}$ in the form of \cite{BrzezniakOndrejat} (see also \cite{Jakubowski}) where it was first used to construct martingale solutions for stochastic evolution equations.
We apply 
this result to the concrete situation and obtain the final result of this section.

\begin{Korollar}\label{corollaryEstimatesToASconvergence}
	Let $(X_n)_{n\in\N}$ be a sequence of adapted $\EAdual$-valued processes satisfying the Aldous condition $[A]$ in $\EAdual$ and
	\begin{align*}
	\sup_{n\in\N} \E \left[\Vert X_n\Vert_\LinftyEA^2\right] <\infty.
	\end{align*}
	Then, there are a subsequence $(X_{n_k})_{k\in\N}$ and random variables $\tilde{X}_k,$ $\tilde{X}$ for $k\in\N$ on a second probability space $(\tilde{\Omega},\tilde{\Filtration},\tildeProb)$ with $\tildeProb^{\tilde{X}_k}=\Prob^{X_{n_k}}$ for $k\in\N,$  and $\tilde{X}_k \to \tilde{X}$ $\tildeProb$-almost surely in $Z_T$ for $k\to \infty.$
\end{Korollar}

\begin{proof}
	We recall that $Z_T=\StetigEAdual\cap \LalphaPlusEinsAlphaPlusEins\cap \weaklyContinousEA$ is a locally convex space. Therefore, the assertion follows by an application of the Propositions $\ref{TightnessCriterion}$ and $\ref{SkohorodJakubowski}$ if for each of the spaces in the definition of $Z_T$ we find a sequence $f_m: Z_T \to \R$ of continuous functions separating points which generates the Borel $\sigma$-algebra.  The separable Banach spaces  $\StetigEAdual$ and $\LalphaPlusEinsAlphaPlusEins$ have this property. \\
	Let $\left\{h_m: m\in \N \right\}$ be a dense subset of $\EAdual.$ Then, we define the countable set \\$F:=\left\{f_{m,t}: m\in\N, t \in [0,T]\cap \Q \right\}$ of functionals on $\weaklyContinousEA$ by
	\begin{align*}
	f_{m,t}(u):= \duality{u(t)}{h_m}
	\end{align*}
	for $m\in\N,$ $t\in [0,T]\cap \Q$ and $u\in \weaklyContinousEA.$ \\
	The set $F$ separates points, since for $u,v \in \weaklyContinousEA$ with $f_{m,t}(u)=f_{m,t}(v)$ for all $m\in\ N$ and $t\in [0,T]\cap \Q,$ we get $\duality{u}{h_m}=\duality{v}{h_m}$ on $[0,T]$ for all $m\in\N$ by continuous continuation and therefore $u=v$ on $[0,T].$ \\
	Furthermore, the density of $\{h_m: m\in\N\}$ and the definition of the locally convex topology yield that $\left(f_{m,t}\right)_{m\in\N, t\in [0,T]\cap \Q}$ generate the Borel $\sigma$-algebra on $\weaklyContinousEA.$
\end{proof}

\section{The Galerkin Approximation}\label{sectionGalerkin}	
In this section, we \dela{treat}{introduce} the Galerkin approximation, which will be used for the proof of the existence of a solution to $\eqref{ProblemStratonovich}$. We prove the well-posedness of the approximated equation and uniform estimates for the solutions that are sufficient to apply Corollary $\ref{corollaryEstimatesToASconvergence}.$\\

By the functional calculus of the selfadjoint operator $S$ from Assumption and Notation \ref{spaceAssumptions}, we define the operators $P_n: H\to H$ by $P_n:=\mathbf{1}_{(0,2^{n+1})}(S)$ for $n\in\N_0.$
Recall from Lemma $\ref{spaceLemma},$ that $S$ has the representation
\begin{align*}
S x=\sum_{m=1}^\infty \lambda_m \skpH{x}{h_m} h_m, \quad x\in \D(S)=\left\{x\in H: \sum_{m=1}^\infty \lambda_m^2 \vert \skpH{x}{h_m}\vert^2<\infty\right\},
\end{align*}
with an orthonormal basis $\left(h_m\right)_{m\in\N}$ and eigenvalues $\lambda_m>0$ such that $\lambda_m\to \infty$ as $m\to \infty.$ 
For $n\in \N_0,$ we set
\begin{align*}
H_n:=\operatorname{span}\left\{h_m: m\in\N, \lambda_m< 2^{n+1}\right\}
\end{align*}
and observe that $P_n$ is the orthogonal projection from $H$ to $H_n.$ Moreover, we have 
\begin{align*}
P_n x = \sum_{\lambda_m< 2^{n+1}} \skpH{x}{h_m}h_m, \qquad x\in {H}.
\end{align*}
Note that we have $h_m\in \bigcap_{k\in\N}\D(S^k)$ for $m\in\N$ and thus, we obtain by the assumption 	$\D(S^k)\hookrightarrow E_A$
for some $k\in\N$ that $H_n$ is a closed subspace of $\EA$ for $n\in\N_0.$ In particular, $H_n$ is a closed subspace of $\EAdual.$
The fact that the operators $S$ and $A$ commute by Assumption \ref{spaceAssumptions} implies that $P_n$ and $A^\frac{1}{2}$ commute. We obtain
\begin{align}\label{PnHeinsKontraktiv}
\norm{P_n x}_{\EA}^2=\norm{P_n x}_{H}^2+\norm{\sqrtA P_n x}_{H}^2
=\norm{P_n x}_{H}^2+\norm{ P_n \sqrtA x}_{H}^2
\le \norm{x}_{\EA}^2,\qquad x\in \EA,
\end{align} 
and 
\begin{align*}
\norm{P_n v}_{\EAdual}=\sup_{\norm{x}_\EA\le 1}\vert \skpH{P_n v}{x}\vert
\le \norm{v}_\EAdual \sup_{\norm{x}_\EA\le 1} \norm{P_n x}_\EA
\le \norm{v}_\EAdual.
\end{align*}
By density, we can extend $P_n$ to an operator $P_n: \EAdual\to H_n$ with $\norm{P_n}_{\EAdual\to \EAdual}\le 1$ and
\begin{align}\label{PnInEAdual}
\duality{v}{P_n v}\in \R, \qquad \duality{v}{P_n w}=\skpH{P_n v}{w}, \qquad v\in \EAdual, \quad w\in \EA.
\end{align}

Despite their nice behaviour as orthogonal projections, it turns out that the operators $P_n,$ $n\in\N,$ lack the crucial property needed in the proof of the a priori estimates of the stochastic terms. In general, they are not uniformly bounded from $\LalphaPlusEins$ to $\LalphaPlusEins.$ To overcome this deficit, we construct another sequence $\left(S_n\right)_{n\in\N}$ of operators $S_n: H \to H_n$ using  functional calculus techniques and the general Littlewood-Paley decomposition from \cite{KrieglerWeis}.\\

%
%
We take a function $\rhodot\in C_c^\infty(0,\infty)$ with $\operatorname{supp} \rhodot \subset [\frac{1}{2},2]$ and $\sum_{m\in\Z} \rhodot(2^{-m} t)=1$ for all $t>0.$ We define
$ \rho_m= \rhodot(2^{-m} \cdot)$ for $m\in \N$ and $\rho_0:=\sum_{m=-\infty}^0 \rhodot(2^{-m} \cdot),$ 		
so that we have $\sum_{m=0}^\infty \rho_m(t)=1$ for all $t>0.$ The sequence $\left(\rho_m\right)_{m\in \N_0}$ is called $\emph{dyadic partition of unity}.$

\begin{Lemma}\label{LittlewoodPaleyLemma}
	We have the norm equivalence
	\begin{align}\label{KrieglerEquivalence}
	\norm{ x}_{\LalphaPlusEins}\eqsim \sup_{\norm{a}_{l^\infty(\N_0)}\le 1}\bigNorm{\sum_{m=0}^\infty a_m \rho_m(S) x}_{\LalphaPlusEins},
	\end{align}
	where the operators $\rho_m(S),$ $m\in\N,$ are defined by the functional calculus for selfadjoint operators.
\end{Lemma}

\begin{proof}
	By Assumption $\ref{spaceAssumptions}$ ii), we obtain that the restriction of $\left(T(t)\right)_{t\ge 0}$ to $\LalphaPlusEins$ defines a $c_0$-semigroup on $\LalphaPlusEins,$ see Theorem 7.1. in \cite{OuhabazHeatEquations}. We denote the corresponding generator by $S_{\alpha+1}.$
	Lemma 6.1. in \cite{KrieglerWeis} implies that  the operator $S_{\alpha+1}$ is $0$-sectorial and has a Mihlin $M^\beta$-calculus for some $\beta>0.$ For a definition of these properties, we refer to \cite{KrieglerWeis}, Section 2. The estimate $\eqref{KrieglerEquivalence}$ follows from Theorem 4.1 in \cite{KrieglerWeis}. \\
\end{proof}

In the next Proposition, we use the estimate from Lemma $\ref{LittlewoodPaleyLemma}$ to construct the sequence $\left(S_n\right)_{n\in\N}$ which we will employ in our Galerkin approximation of the problem $\eqref{ProblemStratonovich}.$	For a more direct proof which employs spectral multiplier theorems from \cite{Uhl}, \cite{kunstmannUhl} rather than the abstract Littlewood-Paley theory from \cite{KrieglerWeis}, we refer to  \cite{FhornungDiss}.
Moreover, we would like to remark that in the meantime, a similar construction has also been applied to use the Galerkin method in the context of stochastic Maxwell equation, see  \cite{hornung2017strong}.

\begin{Prop}\label{PaleyLittlewoodLemma}
	There \dela{is}{exists} a sequence $\left(S_n\right)_{n\in\N_0}$ of selfadjoint operators $S_n: H \to H_n$ for $n\in\N_0$ with
	$S_n \psi \to \psi$ in $E_A$ for $n\to \infty$ and $\psi \in E_A$  and the uniform norm estimates
	\begin{align}\label{SnUniformlyBounded}
	\sup_{n\in\N_0}\norm{S_n}_{{\mathcal{L}(H)}}\le 1, \quad \sup_{n\in\N_0} \norm{S_n}_{\mathcal{L}(\EA)}\le 1, \quad \sup_{n\in\N_0} \norm{S_n}_{\mathcal{L}(L^{\alpha+1})}<\infty.
	\end{align}
\end{Prop}

\begin{proof}
	%
	%
	%
	
	
	We fix $n\in\N$ and define the operators $S_n : H \to H$ for $n\in \N_0$ by
	$S_n:= \sum_{m=0}^n \rho_m(S)$
	via the functional calculus for selfadjoint operators. 	The operator $\rho_m(S)$ is selfadjoint for each $m,$ since $\rho_m$ is real-valued. Hence, $S_n$ is selfadjoint.
	By the convergence property of the functional calculus, we get $S_n \varphi \to \varphi$ in $\EA$ for all $\varphi\in \EA.$ A straightforward calculation using the properties of the dyadic partition of unity leads to
	\begin{align*}
	S_n x=\sum_{\lambda_m < 2^n} \skpH{x}{h_m} h_m+ \sum_{\lambda_m \in [ 2^n,2^{n+1})} \rho_n(\lambda_m) \skpH{x}{h_m} h_m,\quad u\in H.
	\end{align*}
	%
	%
	%

	Therefore, $S_n$ maps $H$ to $H_n$ and we have  $\sup_{n\in\N_0}\norm{S_n}_{{\mathcal{L}(H)}}\le 1.$ The second estimate in \eqref{SnUniformlyBounded} can be derived as in \eqref{PnHeinsKontraktiv}, since $S_n$ and $\sqrtA$ commute. 
	To prove the third estimate, we employ Lemma \ref{PaleyLittlewoodLemma} 				
	with $\left(a_m\right)_{m\in\N_0}$ as $a_m=1$ for $m\le n$ and $a_m=0$ for $m>n$ and obtain for $x\in\LalphaPlusEins$
	\begin{align*}
	\norm{S_n x}_\LalphaPlusEins&= \bigNorm{\sum_{m=0}^\infty a_m \rho_m(S)x }_{\LalphaPlusEins}\le \sup_{\norm{a}_{l^\infty(\N_0)}\le 1}\bigNorm{\sum_{m=0}^\infty a_m \rho_m(S)x  }_{\LalphaPlusEins}
	\lesssim \norm{x}_\LalphaPlusEins.
	\end{align*}
	
\end{proof}	

Using the operators $P_n$ and $S_n,$ $n\in\N,$ we approximate our original problem $\eqref{ProblemStratonovich}$ by the stochastic differential equation in $H_n$ given by	
\begin{equation*}
\left\{
\begin{aligned}
\df u_n(t)&= \left(-\im A u_n(t)-\im P_n F\left( u_n(t)\right) \right) \df t-\im  S_n B(S_n u_n(t)) \circ \df W(t),
\\
u_n(0)&=P_n u_0.
\end{aligned}\right.
\end{equation*}
With the Stratonovich correction term
\begin{align*}
\mu_n := -\frac{1}{2} \sumM \left(S_n B_m S_n\right)^2,
\end{align*}
the approximated problem can also be written in the It\^o form 	
\begin{equation}\label{galerkinEquation}
\left\{
\begin{aligned}
\df u_n(t)&= \left(-\im A u_n(t)-i P_n F\left(  u_n(t)\right)+ \mu_n \left(u_n(t)\right) \right) \df t-\im  S_n B (S_n u_n(t)) \df W(t),
\\
u_n(0)&=P_n u_0.
\end{aligned}\right.
\end{equation}

By the well known theory for finite dimensional stochastic differential equations with locally Lipschitz coefficients, we get a local wellposedness result for $\eqref{galerkinEquation}.$

\begin{Prop}\label{localSolutionGalerkin}
	For each $n\in\N,$ there is a unique local solution $u_n$ of $\eqref{galerkinEquation}$ with continuous paths in $H_n$ and maximal existence time $\tau_n,$ which is a blow-up time in the sense that we have $\limsup_{ t \nearrow \tau_n(\omega)} \norm{u_n(t,\omega)}_{H_n}=\infty$ for almost all $\omega\in\Omega$ with $\tau_n(\omega)<\infty.$
\end{Prop}

The global existence for equation $\eqref{galerkinEquation}$ is based on the conservation of the $L^2$-norm of solutions.

\begin{Prop}\label{MassEstimateGalerkinSolution}
	For each $n\in\N,$ there is a unique global solution $u_n$ of $\eqref{galerkinEquation}$ with continuous paths in $H_n$ and we have the estimate
	\begin{align}\label{LzweiEstimate}
	\norm{u_n(t)}_{H_n}=\norm{u_n(t)}_{H}=\norm{P_n u_0}_{H}\le \norm{u_0}_{H}
	\end{align}
	almost surely for all $t\ge 0.$
\end{Prop}

\begin{proof}
	\emph{Step 1:} We fix $n\in\N$ and take the unique maximal  solution $(u_n,\tau_n)$ from Proposition $\ref{localSolutionGalerkin}.$ We show that the estimate $\eqref{LzweiEstimate}$ holds almost surely on $\{t\le \tau_n\}.$
	The function $\varPhi: H_n \to \R$  defined by  $\varPhi(v):=\norm{v}_{H}^2$ for $v\in H_n$ is  twice continuously Fr\'{e}chet-differentiable with
	\begin{align*}
	\varPhi'[v]h_1&= 2 \Real \skpH{v}{ h_1}, \qquad
	\varPhi''[v] \left[h_1,h_2\right]= 2 \Real \skpH{ h_1}{h_2}
	\end{align*}
	for $v, h_1, h_2\in H_n.$ For the sequence $\left(\tau_{n,k}\right)_{k\in\N}$ of stopping times
	\begin{align*}
	\tau_{n,k}:=\inf \left\{t\in [0,\tau_n]: \norm{u_n(t)}_{H_n}\ge k\right\}\land \tau_n,\qquad k\in\N,
	\end{align*}
	we have  $\tau_{n,k} \nearrow \tau_n$ almost surely and the It\^o process $u_n$ has the representation
	\begin{align*}
	u_n(t)= P_n u_0&+ \int_0^t \left[-\im A u_n(s)-\im P_n F\left( u_n(s)\right)+\mu_n(u_n(s))\right] \df s- \im \int_0^t S_n B (S_n u_n(s)) \df W(s)
	\end{align*}
	almost surely on $\{t\le \tau_{n,k}\}$ for all $k\in \N.$ We fix $k\in\N.$ Since we have
	\begin{align*}
	\tr \Big(\varPhi''[u_n(s)]&\left(-\im S_n B \left(S_n u_n(s)\right),-\im S_n B \left(S_n u_n(s)\right)\right)\Big) \\
	=& \sumM 2 \Real \skpH{-\im S_n B \left(S_n u_n(s)\right)f_m}{-\im S_n B \left(S_n u_n(s)\right)f_m} \\
	=&2 \sumM \norm{S_n B_m S_n u_n(s)}_{H}^2
	\end{align*}
	for $s\in \{t\le \tau_{n,k}\},$ the It\^o lemma yields
	\begin{align*}
	\norm{u_n(t)}_{H}^2=&\norm{P_n u_0}_{H}^2+2 \int_0^t \Real \skpH{ u_n(s)}{ -\im A u_n(s)-\im P_n F\left(u_n(s)\right)+\mu_n(u_n(s))} \df s\nonumber \\
	&+2 \int_0^t \Real \skpH{u_n(s)}{ -\im S_n B(S_n u_n(s)) \df W(s)} +\sumM \int_0^t
	\Vert  S_n B_m S_n u_n(s)\Vert_{H}^2\df s\nonumber
	\end{align*}
	almost surely in $\{t\le \tau_{n,k}\}.$ We fix $v\in H_n$ and $m\in\N$ and calculate
	\begin{align*}
	\Real \skpH{ v}{ -\im A v}&=\Real \left[\im \Vert \sqrtA v\Vert_{H}^2\right]=0,\\
	\Real \skpH{ v}{ -\im P_n F\left( v\right)}&=\Real \duality{\im v}{   F\left( v\right)}=0, \\
	2\Real \skpH{ v}{\mu_n(v)}&= -\sumM\Real  \skpH{ v}{\left(S_n B_m S_n\right)^2 v}=-\sumM  \Vert{ S_n B_m S_n v}\Vert_{H}^2,
	\end{align*}
	where we used $\eqref{PnInEAdual}$ and Assumption $\ref{nonlinearAssumptions}$ i) for the second term and the fact, that the operator $S_n B_mS_n $ is selfadjoint for the third term. Analogously, we get
	\begin{align*}
	\Real \skpH{v}{ -\im S_n B(S_n v)f_m}&=\Real \skpH{v}{ -\im S_n B_m S_n v}= \Real \left[\im\skpH{ v}{S_n   B_m S_n v}\right]=0.
	\end{align*}
	Thus, we obtain
	$
	\norm{u_n(t)}_{H}^2=\norm{P_n u_0}_{H}^2\le \Vert u_0 \Vert_{H}^2
	$
	almost surely in $\{t\le \tau_{n,k}\}.$\\
	
	

	\emph{Step 2.} To show $\tau_n=\infty$ almost surely, we assume the contrary. Therefore, there is $\Omega_0 \in \F$ with $\Prob(\Omega_0)>0$ such that $\tau_n(\omega)<\infty$ and $\tau_{n,k}(\omega)\nearrow \tau_n(\omega)$ for all $\omega \in \Omega_0.$ Hence, $\tau_{n,k}<\infty$ on $\Omega_0$ and by the continuity of the paths of $u_n$ and the definition of $\tau_{n,k},$ we get \\
	$\norm{u_n(\tau_{n,k}(\omega),\omega)}_{H_n}=k$ for all $\omega\in \Omega_0$ and $k\in \N.$ This is a contradiction to Step 1, where we obtained $\norm{u_n(t)}_{H} \le \norm{u_0}_{H}$ almost surely in $\{t\le \tau_{n,k}\}.$
	Therefore, $u_n$ is a global solution and we have
	\begin{align*}
	\norm{u_n(t)}_{H_n}=\norm{u_n(t)}_{H}=\norm{P_n u_0}_{H}\le \norm{u_0}_{H}
	\end{align*}
	almost surely for all $t\ge 0.$
\end{proof}

The  next goal  is to find uniform energy estimates for the global solutions of the equation $\eqref{galerkinEquation}.$  Recall that by Assumption $\ref{nonlinearAssumptions},$ the nonlinearity $F$ has a real antiderivative denoted by $\Fhat.$

\begin{Definition}
	We define the energy $\energy(u)$ of $u\in \EA$ by
	\begin{align*}
	\energy(u):= \frac{1}{2} \Vert A^{\frac{1}{2}} u \Vert_{H}^2+\Fhat(u),\qquad u\in \EA.
	\end{align*}	
\end{Definition}
Note that $\energy(u)$ is welldefined by the  embedding $\EA \hookrightarrow L^{\alpha+1}(M).$ 
In contrast to the uniform $L^2$-estimate in $[0,\infty),$ we cannot exclude the growth of the energy in an infinity time interval. So, we fix $T>0$ from now on.	
As a preparation, we formulate a Lemma, which simplifies the arguments, when the Burkholder-Davis-Gundy inequality is used.

%

\begin{Lemma}\label{LemmaYOmegaNachLzweiZeit}
	Let $r\in [1,\infty),$ $\varepsilon>0,$ $T>0$ and $X\in L^r(\Omega,L^\infty(0,\dela{t}{T})).$ Then,
	\begin{align*}
	\norm{X}_{L^r(\Omega,L^2(0,t))}\le \varepsilon \norm{X}_{L^r(\Omega,L^\infty(0,t))} +\frac{1}{4\varepsilon}\int_0^t \norm{X}_{L^r(\Omega,L^\infty(0,s))} \df s,\qquad t\in[0,T].
	\end{align*}
\end{Lemma}

\begin{proof}
	By interpolation of $L^2(0,t)$ between $L^\infty(0,t)$ and $L^1(0,t)$ and the elementary inequality $\sqrt{a b}\le \varepsilon a+ \frac{1}{4\varepsilon}b$ for $a,b\ge 0$ and $\varepsilon>0,$ we obtain
	\begin{align*}
	\norm{X}_{L^2(0,t)}\le \norm{X}_{L^\infty(0,t)}^\frac{1}{2} \norm{X}_{L^1(0,t)}^\frac{1}{2}\le \varepsilon \norm{X}_{L^\infty(0,t)} +\frac{1}{4\varepsilon} \norm{X}_{L^1(0,t)}.
	\end{align*}
	Now, we take the $L^r(\Omega)$-norm and apply Minkowski's inequality to get
	\begin{align*}
	\norm{X}_{L^r(\Omega,L^2(0,t))}&\le \varepsilon \norm{X}_{L^r(\Omega,L^\infty(0,t))} +\frac{1}{4\varepsilon}\int_0^t \norm{X(s)}_{L^r(\Omega)} \df s\\
	&\le\varepsilon \norm{X}_{L^r(\Omega,L^\infty(0,t))} +\frac{1}{4\varepsilon}\int_0^t \norm{X}_{L^r(\Omega,L^\infty(0,s))} \df s.
	\end{align*}
\end{proof}

The next Proposition is the key step to show that we can apply Corollary $\ref{corollaryEstimatesToASconvergence}$ to the sequence of solutions $(u_n)_{n\in\N}$ of the equation $\eqref{galerkinEquation}$ in the defocusing case.

\begin{Prop}\label{EstimatesGalerkinSolution}
	Under Assumption $\ref{focusing}$ i), the following assertions hold.
	\begin{enumerate}
		\item[a)]  For all $q\in [1,\infty)$ there is a constant $C=C(q,\norm{u_0}_{\EA}, \alpha, F, \left(B_m\right)_{m\in\N},T)>0$ with
		\begin{align*}
		\sup_{n\in\N}\E \Big[\sup_{t\in[0,T]} \left[\norm{u_n(t)}_{H}^2+\energy(u_n(t))\right]^q\Big]\le C
		\end{align*}
		In particular, for all $r\in[1,\infty)$ there is $C_1=C_1(r,\norm{u_0}_{\EA}, \alpha, F, \left(B_m\right)_{m\in\N},T)>0$
		\begin{align*}
		\sup_{n\in\N}\E \Big[\sup_{t\in[0,T]} \norm{u_n(t)}_\EA^{r}\Big]\le C_1.
		\end{align*}
		
		\item[b)] The sequence $(u_n)_{n\in\N}$ satisfies the Aldous condition $[A]$ in $\EAdual.$
	\end{enumerate}
\end{Prop}

\begin{proof}
	\emph{ad a):}
	By Assumption $\ref{nonlinearAssumptions}$ ii) and iii), the restriction of the energy $\energy: H_n \to \R$ is twice continuously Fr\'{e}chet-differentiable with
	\begin{align*}
	\energy'[v]h_1=&\Real \duality{Av+F(v)}{ h_1}; \\
	\energy''[v] \left[h_1,h_2\right]=& \Real \skpH{\sqrtA h_1}{\sqrtA h_2}+\Real \duality{F'[v]h_2}{h_1}
	\end{align*}
	for $v, h_1, h_2\in H_n.$ We compute
	\begin{align*}
	\tr \Big(\energy''[u_n(s)]&\left(-\im S_n B \left(S_n u_n(s)\right),-\im S_n B \left(S_n u_n(s)\right)\right)\Big)\\
	=&\sumM \energy''[u_n(s)]\left(-\im S_n B_m S_n u_n(s),-\im S_n B_m S_n u_n(s)\right)\\
	=& \sumM  \Vert \sqrtA S_n B_m S_n u_n(s)\Vert_{H}^2
	+\sumM \Real \duality{F'[u_n(s)] \left(S_n B_m S_n u_n(s)\right)}{ S_n B_m S_n u_n(s)}
	\end{align*}	
	and therefore,   It\^o's formula and Proposition $\ref{MassEstimateGalerkinSolution}$ lead to the identity
	\begin{align}\label{ItoEnergyStart}
	\norm{u_n(t)}_{H}^2+\energy\left(u_n(t)\right)=&\norm{P_n u_0}_{H}^2+\energy\left(P_n u_0\right)
	\nonumber\\
	&+ \int_0^t\Real \duality{A u_n(s)+F(u_n(s))}{ -\im A u_n(s)-\im P_n F(u_n(s))}\df s\nonumber \\
	&+\int_0^t
	\Real \duality{A u_n(s)+F(u_n(s))}{ \mu_n(u_n(s))} \df s\nonumber \\
	&+\int_0^t \Real \duality{A u_n(s)+F(u_n(s))}{ -\im S_n B \left(S_n u_n(s)\right)\df W(s)}\nonumber \\
	&+\frac{1}{2}\sumM \int_0^t  \Vert \sqrtA S_n B_m S_n u_n(s)\Vert_{H}^2\df s\nonumber\\
	&+\frac{1}{2} \int_0^t \sumM \Real \duality{F'[u_n(s)] \left(S_n B_m S_n u_n(s)\right)}{ S_n B_m S_n u_n(s)} \df s
	\end{align}
	almost surely for all $t\in [0,T].$
	We can use $\eqref{PnInEAdual}$ for
	\begin{align*}
	\Real \duality{F(v)}{ -\im P_n F(v)}=\Real \left[\im \duality{F(v)}{  P_n F(v)}\right]=0;
	\end{align*}
	\begin{align*}
	\Real \left[\duality{A v}{ -\im P_n F(v)}+\duality{F(v)}{ -\im A v}\right]
	&=\Real \left[-\duality{A v}{ \im F(v)}+\overline{\duality{  A v}{\im F(v)}}\right]=0;
	\end{align*}
	\begin{align*}
	\Real \skpH{A v}{ -\im A v}=\Real \left[\im \norm{A v}_{H}^2\right]=0
	\end{align*}
	for all $v\in H_n$ to simplify $\eqref{ItoEnergyStart}$ and get				
	\begin{align}\label{ItoEnergyWithoutExponent}
	\norm{u_n(t)}_{H}^2+\energy\left(u_n(t)\right)=&\norm{P_n u_0}_{H}^2+\energy\left(P_n u_0\right)
	+\int_0^t \Real \duality{A u_n(s)+F(u_n(s))}{ \mu_n(u_n(s))} \df s\nonumber \\
	&+\int_0^t \Real \duality{A u_n(s)+F(u_n(s))}{ -\im S_n B \left(S_n u_n(s)\right)\df W(s)}\nonumber \\
	&+\frac{1}{2}\sumM \int_0^t  \Vert \sqrtA S_n B_m S_n u_n(s)\Vert_{H}^2\df s\nonumber\\
	&+\frac{1}{2}\int_0^t \sumM \Real \duality{F'[u_n(s)] \left(S_n B_m S_n u_n(s)\right)}{ S_n B_m S_n u_n(s)} \df s
	\end{align}	
	almost surely for all $t\in [0,T].$ 												
	Next, we fix $\delta>0,$ $q>1$ and apply \dela{It\^{o}'s}{the It\^o} formula to the process on the LHS of $\eqref{ItoEnergyWithoutExponent}$ and the function $\varPhi: (-\frac{\delta}{2},\infty)\to \R$ defined by $\varPhi(x):=\left(x+\delta\right)^q.$ The derivatives are given by
	\begin{align*}
	\varPhi'(x)=q \left(x+\delta\right)^{q-1},\qquad 	\varPhi''(x)=q (q-1)\left(x+\delta\right)^{q-2},\quad x\in \left(-\frac{\delta}{2},\infty\right).
	\end{align*}
	With the short notation
	\begin{align*}
	Y(s):=\delta+\norm{u_n(s)}_{H}^2+\energy\left(u_n(s)\right),\qquad s\in [0,T],
	\end{align*}
	we obtain
	\begin{align}\label{ItoEnergy}
	Y(t)^{q}=&\left[\delta+\norm{P_n u_0}_{H}^2+\energy\left(P_n u_0\right)\right]^q
	+q\int_0^t Y(s)^{q-1} \Real \duality{A u_n(s)+F(u_n(s))}{ \mu_n(u_n(s))} \df s\nonumber \\
	&+q\int_0^t Y(s)^{q-1} \Real \duality{A u_n(s)+F(u_n(s))}{ -\im S_n B \left(S_n u_n(s)\right)\df W(s)}\nonumber \\
	&+\frac{q}{2}\sumM \int_0^t  Y(s)^{q-1} \Vert \sqrtA S_n B_m S_n u_n(s)\Vert_{H}^2\df s\nonumber\\
	&+\frac{q}{2} \sumM \int_0^t  Y(s)^{q-1} \Real \duality{F'[u_n(s)] \left(S_n B_m S_n u_n(s)\right)}{ S_n B_m S_n u_n(s)} \df s\nonumber\\
	&+\frac{q}{2}(q-1)\sumM \int_0^t Y(s)^{q-2} \left[\Real \duality{A u_n(s)+F(u_n(s))}{-\im S_n B_m S_n u_n(s)}\right]^2 \df s
	\end{align}					
	almost surely for all $t\in [0,T].$
	In order to treat the stochastic integral, we  use  Propositions $\ref{PaleyLittlewoodLemma}$ and  $\ref{MassEstimateGalerkinSolution}$ to estimate for fixed $s\in [0,T]$
	\begin{align}\label{energyEstimateEins}
	\vert\skpH{A u_n(s)}{-\im   S_n B_m S_n u_n(s)}\vert
	&\le \Vert \sqrtA u_n(s)\Vert_{H} \Vert \sqrtA S_n B_m S_n u_n(s) \Vert_{H}\nonumber\\
	&\le \Vert \sqrtA u_n(s)\Vert_{H} \Vert   S_n B_m S_n u_n(s)\Vert_\EA\nonumber\\
	&\le \Vert \sqrtA u_n(s) \Vert_{H} \norm{S_n}_{\mathcal{L}(\EA)}^2\Vert B_m\Vert_{{\mathcal{L}(\EA)}} \Vert  u_n(s)\Vert_\EA\nonumber\\
	&\le \left( \norm{u_n(s)}_{H}^2+\Vert \sqrtA u_n(s) \Vert_{H}^2\right) \Vert B_m\Vert_{{\mathcal{L}(\EA)}}\nonumber\\
	&\lesssim Y(s) \Vert B_m\Vert_{{\mathcal{L}(\EA)}}
	\end{align}
	and $\eqref{nonlinearityEstimate},$ $\eqref{boundantiderivative}$ and Proposition $\ref{PaleyLittlewoodLemma}$ to estimate
	\begin{align}\label{energyEstimatezwei}
	\vert \duality{F(u_n(s))}{ -\im S_n B_m S_n u_n(s) }\vert
	&\le \Vert F(u_n(s))\Vert_\LalphaPlusEinsDual \Vert  S_n B_m S_n u_n(s) \Vert_\LalphaPlusEins\nonumber\\
	&\le   \Vert u_n(s)\Vert_\LalphaPlusEins^{\alpha+1}  \Vert S_n\Vert_{\mathcal{L}(L^{\alpha+1})}^2 \Vert B_m\Vert_{\mathcal{L}(L^{\alpha+1})}\nonumber\\
	&\lesssim \Fhat(u_n(s))   \Vert B_m\Vert_{\mathcal{L}(L^{\alpha+1})}
	\nonumber\\
	&\lesssim Y(s)   \Vert B_m\Vert_{\mathcal{L}(L^{\alpha+1})}.
	\end{align}
	The Burkholder-Gundy-Davis inequality, the estimates $\eqref{energyEstimateEins}$ and $\eqref{energyEstimatezwei},$ Assumption $\ref{stochasticAssumptions}$ and Lemma $\ref{LemmaYOmegaNachLzweiZeit}$ applied to the process $X=Y^q$ with $r=1$  yield for any $\varepsilon>0$
	\begin{align}\label{burkholderEstimate1}
	\E \Big[ &\sup_{s\in[0,t]} \left\vert \int_0^s Y(r)^{q-1}\Real \duality{A u_n(r)+F(u_n(r))}{-\im S_n B  \left(S_n u_n(r)\right) \df W(r)}\right\vert\Big]\nonumber\\
	&\lesssim\E \Big[  \left(\int_0^t \sumM \left\vert Y(r)^{q-1} \duality{A u_n(r)+F(u_n(r))}{-\im S_n B_m S_n u_n(r)}\right\vert^2 \df r\right)^{\frac{1}{2}}\Big]\nonumber\\
	&\lesssim \E \Bigg[  \Bigg(\int_0^t Y(r)^{2q}  \df r\Bigg)^{\frac{1}{2}}\Bigg]	
	\le \varepsilon \E\Big[ \sup_{s\in[0,{t}]} Y(s)^q\Big]+ \frac{1}{4 \varepsilon} \int_0^t \E \Big[\sup_{r\in[0,s]} Y(r)^q\Big] \df s\nonumber\\		
	\end{align}
	The integrands of the deterministic integrals can be estimated  by using  the bounds $\eqref{SnUniformlyBounded},$  Proposition $\ref{MassEstimateGalerkinSolution}$ for the linear and $\eqref{nonlinearityEstimate}$ as well as $\eqref{boundantiderivative}$ for the nonlinear part. We fix $s\in [0,T]$ and get
	\begin{align}\label{energyEstimateDrei}
	\Real \skpH{A u_n(s)}{ \left(S_n B_m S_n \right)^2 u_n(s)}
	&\le \Vert \sqrtA u_n(s)\Vert_{H} \Vert \sqrtA \left(S_n B_m S_n \right)^2 u_n(s)\Vert_{H}\nonumber\\
	&\le \Vert \sqrtA u_n(s)\Vert_{H} \Vert  \left(S_n B_m S_n\right)^2 u_n(s)\Vert_\EA\nonumber\\
	&\le \Vert \sqrtA u_n(s) \Vert_{H} \Vert S_n\Vert_{{\mathcal{L}(\EA)}}^4 \Vert B_m\Vert_{{\mathcal{L}(\EA)}}^2 \Vert  u_n(s)\Vert_\EA\nonumber\\
	&\le \left(\Vert u_n(s) \Vert_{H}^2+\Vert \sqrtA u_n(s) \Vert_{H}^2\right) \Vert B_m\Vert_{{\mathcal{L}(\EA)}}^2\nonumber\\
	&\lesssim Y(s) \Vert B_m\Vert_{{\mathcal{L}(\EA)}}^2; 
	\end{align}
	\begin{align}\label{energyEstimateVier}
	\Real \duality{F(u_n(s))}{ \left(S_n B_m S_n \right)^2 u_n(s)}
	&\le \Vert F(u_n(s))\Vert_\LalphaPlusEinsDual \Vert \left(S_n B_m S_n\right)^2 u_n(s) \Vert_\LalphaPlusEins\nonumber\\
	&\lesssim  \Vert u_n(s)\Vert_\LalphaPlusEins^{\alpha+1} \Vert S_n\Vert_{\mathcal{L}(L^{\alpha+1})}^4 \Vert B_m\Vert_{\mathcal{L}(L^{\alpha+1})}^2\nonumber\\
	&\lesssim \Fhat(u_n(s))   \Vert B_m\Vert_{\mathcal{L}(L^{\alpha+1})}^2
	\lesssim Y(s) \Vert B_m\Vert_{\mathcal{L}(L^{\alpha+1})}^2;
	\end{align}				
	\begin{align}\label{energyEstimateFunf}
	\Vert \sqrtA S_n B_m S_n u_n(s)\Vert_{H}^2
	&\le\Vert  S_n B_m S_n u_n(s)\Vert_\EA^2
	\le \Vert S_n\Vert_{{\mathcal{L}(\EA)}}^4 \Vert B_m\Vert_{{\mathcal{L}(\EA)}}^2 \Vert u_n(s)\Vert_\EA^2\nonumber\\
	&\le \Vert B_m\Vert_{{\mathcal{L}(\EA)}}^2 \left(\Vert  u_n(s) \Vert_{H}^2+\Vert \sqrtA u_n(s)\Vert_{H}^2\right)\nonumber\\
	&\lesssim \Vert B_m\Vert_{{\mathcal{L}(\EA)}}^2 Y(s)
	\end{align}
	for $m\in\N$ and $s\in[0,T].$
	By the bounds $\eqref{SnUniformlyBounded}$ of $S_n$ and the Assumptions $\eqref{deriveNonlinearBound}$ and $\eqref{boundantiderivative}$ on the nonlinearity
	\begin{align}\label{energyEstimateSechs}
	\Real \duality{F'[u_n(s)] \left(S_n B_m S_n u_n(s)\right)}{ S_n B_m S_n u_n(s)}
	&\lesssim \norm{F'[u_n(s)]}_{L^{\alpha+1}\to L^\frac{\alpha+1}{\alpha}}  \Vert S_n B_m S_n  u_n(s)\Vert_\LalphaPlusEins^2\nonumber\\
	&\lesssim  \Vert u_n(s)\Vert_\LalphaPlusEins^{\alpha+1}  \Vert S_n\Vert_{\mathcal{L}(L^{\alpha+1})}^4 \Vert B_m\Vert_{\mathcal{L}(L^{\alpha+1})}^2\nonumber\\
	&\lesssim \Fhat(u_n(s))   \Vert B_m\Vert_{\mathcal{L}(L^{\alpha+1})}^2
	\lesssim Y(s)   \Vert B_m\Vert_{\mathcal{L}(L^{\alpha+1})}^2
	\end{align}
	Substituting the inequalities $\eqref{burkholderEstimate1}$ to $\eqref{energyEstimateSechs},$   into the identity $\eqref{ItoEnergy},$ we get for each $t\in [0,T]$
	\begin{align}\label{GronwallInequality}
	\E \big[\sup_{s\in[0,t]}Y(s)^q\big]\lesssim_q& \left[\delta+\norm{P_n u_0}_{H}^2+\energy(P_n u_0)\right]^q		
	+\E \int_0^t  \sumM \norm{B_m}_{{\mathcal{L}(\EA)}}^2 Y(s)^q\df s\nonumber\\
	&+\E \int_0^t \sumM\norm{B_m}_{{\mathcal{L}(L^{\alpha+1})}}^2Y(s)^q \df s \nonumber\\
	&+\varepsilon \E  \Big[ \sup_{r\in[0,t]}Y(s)^q\Big] +\frac{1}{4 \varepsilon} \int_0^t \E\Big[\sup_{s\in[0,r]}  Y(s)^q\Big]  \df r\nonumber\\
	&+\E \sumM \int_0^t  \norm{B_m}_{{\mathcal{L}(\EA)}}^2 Y(s)^q \df s + \E \int_0^t\sumM\norm{B_m}_{{\mathcal{L}(L^{\alpha+1})}}^2Y(s)^q \df s\nonumber\\
	&+\E \int_{}^t Y(s)^{q} \sumM \max\{\norm{B_m}_{{\mathcal{L}(\EA)}}^2,\norm{B_m}_{\mathcal{L}(L^{\alpha+1})}^2\} \df s\nonumber\\
	\lesssim& \left[\delta+\norm{u_0}_{H}^2+\energy(P_n u_0)\right]^q		
	+\E \int_0^t   Y(s)^q\df s\nonumber\\
	&+\varepsilon \E  \Big[ \sup_{r\in[0,t]}Y(s)^q\Big] +\frac{1}{4 \varepsilon} \int_0^t \E\Big[\sup_{s\in[0,r]}  Y(s)^q\Big]  \df r\nonumber\\
	\lesssim_T& \left[\delta+\norm{u_0}_{H}^2+\energy(P_n u_0)\right]^q
	+\varepsilon \E  \Big[ \sup_{r\in[0,t]}Y(s)^q\Big] + \int_0^t \E \Big[\sup_{s\in[0,r]} Y(s)^q\Big]  \df r.
	\end{align}
	Choosing $\varepsilon>0$ small enough in inequality $\eqref{GronwallInequality}$, the Gronwall lemma yields
	\begin{align*}
	\E \big[\sup_{s\in[0,t]}Y(s)^q\big]\le C\left[\delta+\norm{u_0}_{H}^2+\energy(P_n u_0)\right]^q e^{C t},\qquad t\in [0,T],
	\end{align*}
	with a constant $C>0$, which is uniform in $n\in\N.$ Because of 
	\begin{align*}
	\energy(P_n u_0)\lesssim \norm{\sqrtA P_n u_0}_{H}^2+\norm{P_n u_0}_\LalphaPlusEins^{\alpha+1}\lesssim \norm{P_n u_0}_\EA^2+\norm{P_n u_0}_\EA^{\alpha+1}\lesssim 1,
	\end{align*}
	we obtain the assertion of Proposition \ref{EstimatesGalerkinSolution}, part a). \\
	
	\emph{ad b):} Now, we continue with the proof of the Aldous condition. We have 
	\begin{align*}
	u_n(t)- P_n u_0=&  -\im \int_0^t A u_n(s) \df s-\im \int_0^t P_n F(u_n(s)) \df s+\int_0^t \mu_n(u_n(s)) \df s\\
	\hspace{1cm}&- \im \int_0^t S_n B (S_n u_n(s)) \df W(s)\\
	=&:J_1(t)+J_2(t)+J_3(t)+J_4(t)
	\end{align*}
	in $H_n$ almost surely for all $t\in [0,T]$ and therefore
	\begin{align*}
	\Vert u_n((\tau_n+\theta)\land T)-u_n(\tau_n)\Vert_\EAdual\le\sum_{k=1}^4 \Vert J_k((\tau_n+\theta)\land T)-J_k(\tau_n)\Vert_{\EAdual}
	\end{align*}
	for each sequence $\left(\tau_n\right)_{n\in\N}$ of stopping times and $\theta>0.$ 
	Hence, we get
	\begin{align}\label{AldousStartingEstimate}
	\Prob \left\{\Vert u_n((\tau_n+\theta)\land T)-u_n(\tau_n)\Vert_\EAdual\ge \eta \right\}\le \sum_{k=1}^4 \Prob \left\{\Vert J_k((\tau_n+\theta)\land T)-J_k(\tau_n)\Vert_{\EAdual}\ge \frac{\eta}{4}\right\}
	\end{align}
	for a fixed $\eta>0$. We aim to apply Tschebyscheff's inequality and  estimate the expected value of each term in the sum. We use part a) for
	\begin{align*}
	\E\Vert J_1((\tau_n+\theta)\land T)-J_1(\tau_n)\Vert_{\EAdual}&\le \E \int_{\tau_n}^{(\tau_n+\theta)\land T} \Vert A u_n(s)\Vert_\EAdual \df s
	\le   \E\int_{\tau_n}^{(\tau_n+\theta)\land T} \Vert \sqrtA u_n(s)\Vert_{H} \df s\\
	&\lesssim \theta \E \big[ \sup_{s\in[0,T]} \norm{u_n(s)}_\EA\big]
	\le \theta \E \big[ \sup_{s\in[0,T]} \norm{u_n(s)}_\EA^2\big]^{\frac{1}{2}}\le \theta C_1;
	\end{align*}
	the embedding $ \LalphaPlusEinsDual\hookrightarrow \EAdual$ and the estimate $\eqref{nonlinearityEstimate}$ of the nonlinearity $F$ for
	\begin{align*}
	\E\Vert J_2((\tau_n+\theta)\land T)&-J_2(\tau_n)\Vert_{\EAdual}\le \E \int_{\tau_n}^{(\tau_n+\theta)\land T} \Vert P_n F(u_n(s))\Vert_\EAdual \df s\\
	&\le \E \int_{\tau_n}^{(\tau_n+\theta)\land T} \Vert  F(u_n(s))\Vert_\EAdual \df s\lesssim  \E\int_{\tau_n}^{(\tau_n+\theta)\land T} \Vert  F(u_n(s))\Vert_\LalphaPlusEinsDual \df s\\&\lesssim  \E\int_{\tau_n}^{(\tau_n+\theta)\land T} \Vert  u_n(s)\Vert_\LalphaPlusEins^\alpha \df s
	\lesssim \theta \E \big[\sup_{s\in[0,T]} \norm{u_n(s)}_\EA^\alpha\big]
	\le \theta C_2
	\end{align*}
	Propositions $\ref{PaleyLittlewoodLemma}$ and $\ref{MassEstimateGalerkinSolution}$ for
	\begin{align*}
	\E\Vert J_3((\tau_n+\theta)\land T)-J_3(\tau_n)\Vert_{\EAdual}
	&= \frac{1}{2}\E \left\Vert \int_{\tau_n}^{(\tau_n+\theta)\land T}\sumM \left(S_n B_m S_n\right)^2 u_n(s) \df s \right\Vert_\EAdual\\	
	&\le \frac{1}{2}\E \int_{\tau_n}^{(\tau_n+\theta)\land T} \sumM\Vert \left(S_n B_m S_n\right)^2 u_n(s)\Vert_\EAdual \df s\\
	&\lesssim \E  \int_{\tau_n}^{(\tau_n+\theta)\land T}\sumM \Vert \left(S_n B_m S_n\right)^2 u_n(s)\Vert_{H} \df s\\
	&\le  \E \int_{\tau_n}^{(\tau_n+\theta)\land T} \sumM \norm{B_m}_{{\mathcal{L}(H)}}^2\Vert  u_n(s)\Vert_{H} \df s\\
	&\lesssim \theta \E \big[ \sup_{s\in[0,T]} \Vert  u_n(s)\Vert_{H}\big]
	=C_3 \theta
	\end{align*}
	Finally, we use the It\^o isometry and again the Propositions  $\ref{PaleyLittlewoodLemma}$ and $\ref{MassEstimateGalerkinSolution}$ for
	\begin{align*}
	\E\Vert J_4((\tau_n+\theta)\land T)-J_4(\tau_n)\Vert_{\EAdual}^2&
	\le \E \left\Vert\int_{\tau_n}^{(\tau_n+\theta)\land T} S_n B \left(S_n u_n(s)\right) \df W(s)\right\Vert_{H}^2\\
	&= \E \left[\int_{\tau_n}^{(\tau_n+\theta)\land T}\Vert S_n B\left(S_n u_n(s)\right)\Vert_{\HS(Y,{H})}^2 \df s\right]\\
	&= \E \left[\int_{\tau_n}^{(\tau_n+\theta)\land T}\sumM\Vert S_n B_m S_n u_n(s)\Vert_{H}^2 \df s\right]\\
	&\le \E \left[ \int_{\tau_n}^{(\tau_n+\theta)\land T}\sumM \norm{B_m}_{{\mathcal{L}(H)}}^2\Vert  u_n(s)\Vert_{H}^2 \df s\right]\\
	&\lesssim \theta \E \big[\sup_{s\in[0,T]} \Vert  u_n(s)\Vert_{H}^2\big]
	=\theta C_4 
	\end{align*}		
	By the Tschebyscheff inequality, we obtain for a given $\eta>0$
	\begin{align}\label{AldousEins}
	\Prob \left\{\Vert J_k((\tau_n+\theta)\land T)-J_k(\tau_n)\Vert_{\EAdual}\ge \frac{\eta}{4}\right\}\le \frac{4}{\eta} \E\Vert J_k((\tau_n+\theta)\land T)-J_k(\tau_n)\Vert_{\EAdual}\le \frac{ 4C_k \theta}{\eta}
	\end{align}
	for $k\in \{1,2,3\}$ and
	\begin{align}\label{AldousZwei}
	\Prob \left\{\Vert J_4((\tau_n+\theta)\land T)-J_4(\tau_n)\Vert_{\EAdual}\ge \frac{\eta}{4}\right\}\le \frac{16}{\eta^2} \E\Vert J_4((\tau_n+\theta)\land T)-J_4(\tau_n)\Vert_{\EAdual}^2\le \frac{16 C_4 \theta}{\eta^2}.
	\end{align}
	Let us fix $\varepsilon>0.$ Due to  estimates $\eqref{AldousEins}$ and $\eqref{AldousZwei}$ we can choose $\delta_1,\dots,\delta_4>0$ such that
	\begin{align*}
	\Prob \left\{\Vert J_k((\tau_n+\theta)\land T)-J_k(\tau_n)\Vert_{\EAdual}\ge \frac{\eta}{4}\right\}\le \frac{\varepsilon}{4}
	\end{align*}
	for $0<\theta\le \delta_k$ and $k=1,\dots,4.$ With $\delta:= \min \left\{\delta_1,\dots,\delta_4\right\},$  using $\eqref{AldousStartingEstimate}$ we get
	\begin{align*}
	\Prob \left\{\Vert J_k((\tau_n+\theta)\land T)-J_k(\tau_n)\Vert_{\EAdual}\ge \eta\right\}\le \varepsilon
	\end{align*}
	for all $n\in\N$ and $0<\theta\le \delta$ and therefore, the Aldous condition $[A]$ holds in $E_A^*.$
\end{proof}

We continue with the a priori estimate for solutions of $\eqref{galerkinEquation}$ with a focusing nonlinearity. Note that this case is harder since the expression
\begin{align*}
\norm{v}_{H}^2+\energy(v):=\norm{v}_\EA^2+\Fhat(v), \qquad v\in H_n,
\end{align*}
does not dominate $\norm{v}_\EA^2,$ because $\Fhat$ is negative.

\begin{Prop}\label{EstimatesGalerkinSolutionFocusing}
	Under Assumption $\ref{focusing}$ i'), the following assertions hold:
	\begin{enumerate}
		\item[a)] For all $r\in[1,\infty),$ there is a constant $C=C(r,\norm{u_0}_\EA, \alpha, F,  \left(B_m\right)_{m\in\N},T)>0$ with
		\begin{align*}
		\sup_{n\in\N}\E \Big[\sup_{t\in[0,T]} \norm{u_n(t)}_\EA^{r}\Big]\le C
		\end{align*}
		\item[b)] The sequence $(u_n)_{n\in\N}$ satisfies the Aldous condition $[A]$ in $\EAdual.$
	\end{enumerate}
\end{Prop}	

\begin{proof}
	Let $\varepsilon>0.$ Assumption $\ref{focusing}$ i') and Young's inequality imply that  there are $\gamma>0$ and $C_\varepsilon>0$ such that
	\begin{align}\label{focusingNonlinearityControl}
	\norm{u}_\LalphaPlusEins^{\alpha+1} \lesssim \varepsilon \norm{u}_\EA^2+C_\varepsilon \norm{u}_H^\gamma,\qquad u\in\EA,
	\end{align}
	and therefore by Proposition $\ref{MassEstimateGalerkinSolution}$, we infer that
	\begin{align}\label{gagliardoFhat}
	-\Fhat(u_n(t))&\lesssim \norm{u_n(t)}_\LalphaPlusEins^{\alpha+1} \lesssim \varepsilon \norm{u_n(t)}_\EA^2+C_\varepsilon \norm{u_n(t)}_H^\gamma\nonumber\\
	&\lesssim \varepsilon \norm{\sqrtA u_n(t)}_{H}^2+\varepsilon \norm{u_0}_H^2+C_\varepsilon \norm{u_0}_H^\gamma,\qquad t\in [0,T].
	\end{align}
	By the same calculations as in the proof of Proposition $\ref{EstimatesGalerkinSolution}$ we get 
	\begin{align}\label{ItoFocusing}
	\frac{1}{2} \norm{\sqrtA u_n(s)}_{H}^2=&\energy(u_n(s))-\Fhat(u_n(s))\nonumber\\
	=&-\Fhat(u_n(s))+\energy\left(P_n u_0\right)
	+\int_0^s \Real \duality{A u_n(r)+F(u_n(r))}{ \mu_n(u_n(r))} \df r\nonumber\\
	&+\int_0^s \Real \duality{A u_n(r)+F(u_n(r))}{-\im S_n B\left(S_n u_n(r)\right)\df W(r)}\nonumber\\
	&+\frac{1}{2}\sumM \int_0^s  \Vert \sqrtA S_n B_m S_n u_n(r)\Vert_{H}^2\df r\nonumber\\
	&+\frac{1}{2}\int_0^s \sumM \Real \duality{F'[u_n(r)] \left(S_n B_m S_n  u_n(r)\right)}{ S_n B_m S_n u_n(r)} \df r
	\end{align}
	almost surely for all $s\in[0,T].$
	In the following, we fix $q\in [1,\infty)$ and $t\in(0,T]$ and want to apply the $L^q(\Omega,L^\infty(0,t))$-norm to the identity $\eqref{ItoFocusing}.$ We will use the notation	
	\begin{align}\label{ProcessX}
	X(s):=\left[\norm{u_0}_{H}^2+\norm{\sqrtA u_n(s)}_H^2+\norm{u_n(s)}_\LalphaPlusEins^{\alpha+1}\right], \qquad s\in [0,T],
	\end{align}								
	and estimate the stochastic integral by the Burkholder-Gundy-Davis inequality and the estimates $\eqref{energyEstimateEins}$ and $\eqref{energyEstimatezwei}$ as well as Lemma $\ref{LemmaYOmegaNachLzweiZeit}$
	\begin{align}\label{energyFocusingTerm1}
	&\left\Vert \int_0^\cdot \Real \duality{A u_n(r)+F(u_n(r))}{-\im S_n B  \left(S_n u_n(r)\right) \df W(r)}\right\Vert_{L^q(\Omega,L^\infty(0,t))}\nonumber\\
	&\hspace{2cm}\lesssim \left\Vert  \left(\sumM \vert \duality{A u_n(r)+F(u_n(r))}{-\im S_n B_m S_n u_n(r)}\vert^2\right)^\frac{1}{2}\right\Vert_{L^q(\Omega,L^2([0,t]))}\nonumber\\
	&\hspace{2cm}\lesssim \left\Vert  X\right\Vert_{L^q(\Omega,L^2([0,t]))}\nonumber\\
	&\hspace{2cm}\le \varepsilon \norm{X}_{L^q(\Omega,L^\infty(0,t))} +\frac{1}{4\varepsilon}\int_0^t \norm{X}_{L^q(\Omega,L^\infty(0,s))} \df s
	\end{align}	
	By $\eqref{gagliardoFhat},$ we get
	\begin{align}\label{energyFocusingTerm2}
	\norm{-\Fhat(u_n)}_{L^q(\Omega,L^\infty(0,t))}
	\lesssim \varepsilon \left\Vert \norm{\sqrtA u_n}_{H}^2\right\Vert_{L^q(\Omega,L^\infty(0,t))}+\varepsilon \norm{u_0}_H^2+C_\varepsilon \norm{u_0}_H^\gamma
	\end{align}	
	For the following estimates, we will use $\eqref{energyEstimateDrei}$-$\eqref{energyEstimateSechs}$ and \dela{Minkowski's}{the Minkowski} inequality and obtain		
	\begin{align}\label{energyFocusingTerm3}
	\left\Vert\int_0^\cdot\Real \duality{A u_n(s)+F(u_n(s))}{ \mu_n(u_n(s))} \df s \right\Vert_{L^q(\Omega,L^\infty(0,t))}
	&\lesssim \left\Vert \int_0^t X(s) \df s\right\Vert_{L^q(\Omega)}
	\lesssim \int_0^t \norm{X(s)}_{L^q(\Omega)} \df s;
	\end{align}			
	\begin{align}\label{energyFocusingTerm4}
	\left\Vert\sumM \int_0^\cdot  \Vert \sqrtA S_n B_m S_n u_n(s)\Vert_{H}^2\df s\right\Vert_{L^q(\Omega,L^\infty(0,t))}
	\lesssim \left\Vert \int_0^t X(s) \df s\right\Vert_{L^q(\Omega)}\lesssim \int_0^t \norm{X(s)}_{L^q(\Omega)} \df s;
	\end{align}
	\begin{align}\label{energyFocusingTerm5}
	&\left\Vert\int_0^\cdot \sumM \Real \duality{F'[u_n(s)] \left(S_n B_m S_n  u_n(s)\right)}{ S_n B_m S_n u_n(s)} \df s \right\Vert_{L^q(\Omega,L^\infty(0,t))}\nonumber\\
	&\hspace{7cm}\lesssim \left\Vert \int_0^t X(s) \df s\right\Vert_{L^q(\Omega)}\lesssim \int_0^t \norm{X(s)}_{L^q(\Omega)} \df s.
	\end{align}
	By $\eqref{ItoFocusing}$ and the estimates $\eqref{energyFocusingTerm1}$-$\eqref{energyFocusingTerm5},$ we get			
	\begin{align}\label{sqrtAFocusing}
	\left\Vert  \norm{\sqrtA u_n}_{H}^{2}\right\Vert_{L^q(\Omega,L^\infty(0,t))} \lesssim&\,\,\, \varepsilon\left\Vert \norm{\sqrtA u_n(t)}_{H}^2\right\Vert_{L^q(\Omega,L^\infty(0,t))}+\varepsilon \norm{u_0}_H^2+C_\varepsilon \norm{u_0}_H^\gamma+\norm{ u_0}_\EA
	\nonumber\\&+\int_0^t \norm{X(s)}_{L^q(\Omega)} \df s+\varepsilon \norm{X}_{L^q(\Omega,L^\infty(0,t))} +\frac{1}{4\varepsilon}\int_0^t \norm{X}_{L^q(\Omega,L^\infty(0,s))} \df s
	\nonumber\\&+\int_0^t \norm{X(s)}_{L^q(\Omega)} \df s
	\end{align}
	In order to estimate the terms with $X$ by the LHS of \eqref{sqrtAFocusing}, we exploit $\eqref{focusingNonlinearityControl}$ to get		
	\begin{align*}
	\norm{X}_{L^q(\Omega,L^\infty(0,t))}&\le \norm{u_0}_{H}^2+\E \big[\sup_{s\in [0,t]}  \norm{\sqrtA u_n(s)}_{H}^{2q}\big]^{\frac{1}{q}}+\E \big[\sup_{s\in [0,t]}  \norm{ u_n(s)}_\LalphaPlusEins^{(\alpha+1)q}\big]^{\frac{1}{q}}\\
	&\lesssim \norm{u_0}_{H}^2+\E \big[\sup_{s\in [0,t]}  \norm{\sqrtA u_n(s)}_{H}^{2q}\big]^{\frac{1}{q}}\nonumber\\
	&+\varepsilon \E \big[\sup_{s\in [0,t]}   \norm{\sqrtA u_n(s)}_{H}^{2q}\big]^{\frac{1}{q}}+\varepsilon \norm{u_0}_H^2+C_\varepsilon \norm{u_0}_H^\gamma\\
	&\lesssim \left\Vert\sup_{s\in [0,t]}  \norm{\sqrtA u_n(s)}_{H}^{2}\right\Vert_{L^q(\Omega)}+ \norm{u_0}_H^2+ \norm{u_0}_H^\gamma.
	\end{align*}			
	Hence, by \eqref{ProcessX}, we obtain
	\begin{align*}
	\left\Vert  \norm{\sqrtA u_n}_{H}^{2}\right\Vert_{L^q(\Omega,L^\infty(0,t))} \lesssim&\varepsilon \left\Vert \norm{\sqrtA u_n(t)}_{H}^2\right\Vert_{L^q(\Omega,L^\infty(0,t))}+\varepsilon \norm{u_0}_H^2+C_\varepsilon \norm{u_0}_H^\gamma+\norm{ u_0}_\EA
	\\&+\int_0^t \left\Vert  \norm{\sqrtA u_n}_{H}^{2}\right\Vert_{L^q(\Omega,L^\infty(0,s))} \df s
	+ t \norm{u_0}_H^2+ t \norm{u_0}_H^\gamma \nonumber\\
	&+\varepsilon \left\Vert  \norm{\sqrtA u_n(s)}_{H}^{2}\right\Vert_{L^q(\Omega,L^\infty(0,t))}+ \varepsilon\norm{u_0}_H^2+\varepsilon \norm{u_0}_H^\gamma.
	\end{align*}
	Choosing $\varepsilon>0$ small enough, we get
	\begin{align*}
	\left\Vert  \norm{\sqrtA u_n}_{H}^{2}\right\Vert_{L^q(\Omega,L^\infty(0,t))} \le& C_1(\norm{u_0}_\EA,T,q)+\int_0^t C_2(q)\left\Vert  \norm{\sqrtA u_n}_{H}^{2}\right\Vert_{L^q(\Omega,L^\infty(0,s))} \df s,\qquad t\in [0,T],
	\end{align*}							
	and thus, \dela{Gronwall's}{the Gronwall} Lemma yields
	\begin{align*}
	\left\Vert \norm{\sqrtA u_n(s)}_{H}^{2}\right\Vert_{L^q(\Omega,L^\infty(0,t))}\le C_1(\norm{u_0}_\EA,T,q) e^{C_2(q)t},\qquad t\in[0,T].
	\end{align*}
	This implies that there is $C>0$ with
	\begin{align*}
	\sup_{n\in\N}\E \Big[\sup_{t\in[0,T]} \norm{u_n(t)}_\EA^{2q}\Big]\le C.
	\end{align*}
	since the $H$-norm is conserved by Proposition $\ref{MassEstimateGalerkinSolution}.$ Therefore, we obtain the assertion for $r\ge 2.$ Finally, the case $r\in[1,2)$ is an application of H\"older's inequality.\\

	\emph{ad b):} Analogous to the proof of Proposition $\ref{EstimatesGalerkinSolution}$ b).
\end{proof}

\section{Construction of a martingale solution}
The aim  of this section is the construction of a solution of equation $\eqref{ProblemStratonovich}$ by a suitable limiting process in the Galerkin equation $\eqref{galerkinEquation}$ using the results from the previous  sections. Let us recall that
\begin{align*}
Z_T:=\StetigEAdual\cap \LalphaPlusEinsAlphaPlusEins \cap \weaklyContinousEA.
\end{align*}

\begin{Prop} \label{PropAlmostSureConvergence}
	Let $\left(u_n\right)_{n\in\N}$ be the sequence of solutions to the Galerkin equation $\eqref{galerkinEquation}.$
	\begin{enumerate}
		\item[a)] There are a subsequence $\left(u_{n_k}\right)_{k\in\N}$, a probability space $\left(\tilde{\Omega},\tilde{\F},\tilde{\Prob}\right)$ and random variables $v_k, v:\tilde{\Omega} \rightarrow Z_T$ with $\tilde{\Prob}^{v_k}=\Prob^{u_{n_k}}$ such that $v_k\to v$ $\tilde{\Prob}$-a.s.  in $Z_T$ for $k\to \infty.$
		\item[b)] We have $v_k \in  C\left([0,T],H_k\right)$ $\tilde{\Prob}$-a.s. and for all $r\in [1,\infty),$ there is
		$C>0$
		with
		\begin{align*}
		\sup_{k\in\N} \Etilde \left[ \norm{v_k}_\LinftyEA^r\right]\le C
		\end{align*}
		\item[c)] For all $r\in [1,\infty),$ we have
		\begin{align*}
		\Etilde \left[ \norm{v}_\LinftyEA^r\right]\le C
		\end{align*}
		with the same constant $C>0$ as in $b).$
	\end{enumerate}
\end{Prop}
For the precise dependence of the constants, we refer to the Propositions \ref{EstimatesGalerkinSolution} and \ref{EstimatesGalerkinSolutionFocusing}.
\begin{proof}
	\emph{ad a):}  The estimates to apply Corollary $\ref{corollaryEstimatesToASconvergence}$ are provided by  Propositions \ref{EstimatesGalerkinSolution} and \ref{EstimatesGalerkinSolutionFocusing}.
	
	\emph{ad b):} Since we have $u_{n_k} \in C\left([0,T],H_k\right)$ $\Prob$-a.s. and $C\left([0,T],H_k\right)$ is closed in  $\StetigEAdual$ and therefore a Borel set
	, we conclude $v_k \in  C\left([0,T],H_k\right)$ $\tilde{\Prob}$-a.s. by the identity of the laws. Furthermore, the map $C\left([0,T],H_k\right)\ni u \mapsto \norm{u}_{\LinftyEA}^r\in [0,\infty)$ is continuous and therefore measurable, \dela{such}{so} that we can conclude that
	\begin{align*}
	\Etilde \left[\norm{v_k}_\LinftyEA^r\right]
	&=\int_{C([0,T],H_k)}\norm{u}_\LinftyEA^r \df \tildeProb^{v_k}(u)
	=\int_{C([0,T],H_k)}\norm{u}_\LinftyEA^r \df \Prob^{u_{n_k}}(u)\\
	&=\E \left[\norm{u_{n_k}}_\LinftyEA^r \right].
	\end{align*}
	Use the Propositions $\ref{EstimatesGalerkinSolution}$ in the defocusing respectively $\ref{EstimatesGalerkinSolutionFocusing}$ in the focusing case to get the assertion.\\

	\emph{ad c):}
	We have $v_n \to v$ almost surely in $\LalphaPlusEinsAlphaPlusEins$ by part a). From part b) and the embedding $\LinftyEA\hookrightarrow \LalphaPlusEinsAlphaPlusEins$, we obtain that the sequence $\left(v_n\right)_{n\in\N}$ is bounded in $L^{\alpha+1}(\tilde{\Omega}\times [0,T]\times M).$ By Vitali's Theorem (see \cite{Elstrodt}, Theorem VI, 5.6), we  conclude
	\begin{align*}
	v_n \to v \hspace{0,3cm} \text{in} \hspace{0,3 cm} L^2(\tilde{\Omega},\LalphaPlusEinsAlphaPlusEins)
	\end{align*}
	for $n\to \infty.$
	On the other hand, part b) yields the existence of $\tilde{v}\in L^r(\tilde{\Omega},\LinftyEA)$ for all $r\in[1,\infty)$ with norm less than the constant $C=C(\norm{u_0}_\EA,T,r)>0$  and a subsequence $\left(v_{n_k}\right)_{k\in\N},$ such that $v_{n_k} \rightharpoonup^* \tilde{v}$ for $k\to \infty.$
	Especially, $v_{n_k} \rightharpoonup^* \tilde{v}$ for $k\to \infty$ in $L^2(\tilde{\Omega},\LalphaPlusEinsAlphaPlusEins)$ and hence,
	\begin{align*}
	v=\tilde{v}\in L^r(\tilde{\Omega},\LinftyEA)
	\end{align*}
	
\end{proof}

The next Lemma shows, how convergence in $Z_T$ can be used for the convergence of the terms appearing in the Galerkin equation.
\begin{Lemma}\label{LemmaConvergences}
	Let $z_n\in C([0,T],H_n)$ for $n\in\N$ and $z\in Z_T.$ Assume $z_n \to z$ for $n\to \infty$ in $Z_T.$ Then, for $t\in[0,T]$ and $\psi \in \EA$ as $n\to \infty$ 			
	\begin{align*}
	\skpH{z_n(t)}{\psi} \to \duality{z(t)}{\psi},
	\end{align*}
	\begin{align*}
	\int_0^t \skpH{A z_n(s)}{\psi} \df s \to \int_0^t \duality{A z(s)}{ \psi} \df s,
	\end{align*}
	\begin{align*}
	\int_0^t \skpH{\mu_n\left(z_n(s)\right)}{\psi} \df s \to \int_0^t \duality{\mu\left(z(s)\right)}{ \psi} \df s,					
	\end{align*}	
	\begin{align*}
	\int_0^t \skpH{P_n F(z_n(s))}{\psi} \df s \to \int_0^t \duality{F( z(s))}{\psi} \df s.
	\end{align*}				
\end{Lemma}
\begin{proof}
	\emph{Step 1:} We fix $\psi \in \EA$ and $t\in[0,T].$
	Recall, that the assumption implies $z_n \to z$ for $n\to \infty$ in $\StetigEAdual.$ This can be used to deduce
	\begin{align*}
	\left\vert \skpH{z_n(t)}{\psi} - \duality{z(t)}{\psi}\right\vert \le \Vert z_n - z\Vert_\StetigEAdual \norm{\psi}_\EA \to 0.
	\end{align*}
	By $z_n \to z$  in $\weaklyContinousEA$ we get
	$\sup_{s\in [0,T]} \vert \duality{z_n(s)-z(s)}{\varphi}\vert \to 0$
	for $n\to \infty$ and all $\varphi \in \EAdual.$ We plug in $\varphi=A \psi$  and use
	$\duality{A z_n(s)}{\psi}=\duality{z_n(s)}{A \psi}$
	for $n\in\N$ and $s\in[0,t]$ to get
	\begin{align*}
	\int_0^t \left\vert\skpH{A z_n(s)}{\psi}- \duality{z(s)}{A \psi}\right\vert \df s&=				\int_0^t \left\vert \duality{z_n(s)-z(s)}{A \psi}\right\vert \df s\\
	&\le T \sup_{s\in [0,T]} \vert \duality{z_n(s)-z(s)}{A \psi}\vert \to 0,\quad n\to \infty.
	\end{align*}

	\emph{Step 2:}			
	First, we fix $m\in \N.$ Using that the operators $B_m$  and  $S_n$ are selfadjoint, we get
	\begin{align*}
	\int_0^t &\left\vert \skpHBig{(S_n B_m S_n)^2 z_n(s)}{\psi} - \duality{B_m^2 z(s)}{ \psi}\right\vert \df s	\\
	\le& \int_0^t \left\vert \skpHBig{(S_n-I ) B_m S_n^2 B_m S_n z_n(s)}{\psi} \right\vert \df s	+\int_0^t \left\vert \skpHBig{ B_m (S_n^2-I ) B_m S_n  z_n(s)}{\psi} \right\vert \df s	\\
	&\hspace{1cm}+ \int_0^t \left\vert \skpHBig{ B_m^2 (S_n-I) z_n(s)}{\psi} \right\vert \df s	+\int_0^t \left\vert  \duality{B_m^2 \left(z_n(s)-z(s)\right)}{ \psi}\right\vert \df s\\
	\le&T \norm{z_n}_\StetigEAdual \norm{B_m}_{{\mathcal{L}(\EA)}}^2 \norm{S_n}_{{\mathcal{L}(\EA)}}^3 \norm{(S_n-I)\psi}_\EA	\\
	&\hspace{1cm}+T \norm{z_n}_\StetigEAdual \norm{S_n}_{{\mathcal{L}(\EA)}} \norm{B_m}_{{\mathcal{L}(\EA)}} \norm{S_n+I}_{{\mathcal{L}(\EA)}}\norm{(S_n-I) \left(B_m \psi\right)}_\EA	\\
	&\hspace{1cm}+T \norm{z_n}_\StetigEAdual  \norm{(S_n-I) \left(B_m^2 \psi\right)}_\EA	\\
	&\hspace{1cm}+T \norm{z_n-z}_\StetigEAdual \norm{B_m^2 }_{{\mathcal{L}(\EA)}} \norm{\psi}_\EA\longrightarrow 0, \quad n\to \infty,
	\end{align*}
	since $S_n \varphi \to \varphi$ in $\EA$ for $\varphi \in \EA$ by Proposition $\ref{PaleyLittlewoodLemma}$  and $z_n \to z$ in $\StetigEAdual.$
	By the estimate
	\begin{align*}
	\int_0^t &\left\vert \skpHBig{(S_n B_m S_n)^2 z_n(s)}{\psi} - \duality{B_m^2 z(s)}{ \psi}\right\vert \df s\\
	&\le T \norm{\psi}_\EA \left[\norm{(S_n B_m S_n)^2}_{{\mathcal{L}(\EA)}} \norm{z_n}_\StetigEAdual+\norm{B_m^2}_{{\mathcal{L}(\EA)}} \norm{z}_\StetigEAdual\right]\\
	&\lesssim_{T,\psi} \norm{B_m}_{{\mathcal{L}(\EA)}}^2 \in l^1(\N)
	\end{align*}
	and Lebesgue's convergence Theorem, we obtain
	\begin{align*}
	\sum_{m=1}^\infty \int_0^t \left\vert \skpHBig{(S_n B_m S_n)^2 z_n(s)}{\psi} - \duality{B_m^2 z(s)}{ \psi}\right\vert \df s\longrightarrow 0, \qquad n\to \infty,
	\end{align*}
	and therefore
	\begin{align*}
	\int_0^t \skpH{\mu_n\left(z_n(s)\right)}{\psi} \df s \to \int_0^t \duality{\mu\left(z(s)\right)}{ \psi} \df s,\qquad n\to \infty.	
	\end{align*}	
	
	\emph{Step 3.}			
	Before we prove the last assertion, we recall $z_n \to z$ in $\LalphaPlusEinsAlphaPlusEins$ for $n\to \infty.$  We estimate
	\begin{align}\label{NonlinearConvergenceStart}
	\int_0^t& \left\vert \skpH{P_n F(z_n(s))}{\psi} - \duality{F(z(s))}{\psi} \right\vert \df s\nonumber\\
	\le& \int_0^t \left\vert \duality{F(z_n(s))}{(P_n-I)\psi} \right\vert \df s+\int_0^t \left\vert  \duality{F(z_n(s))-F(z(s))}{\psi} \right\vert \df s
	\end{align}
	where we used $\eqref{PnInEAdual}.$ For the first term in $\eqref{NonlinearConvergenceStart}$, we look at
	\begin{align*}
	\int_0^t \left\vert \duality{F(z_n(s))}{(P_n-I)\psi} \right\vert \df s
	&\le \norm{F (z_n)}_{L^1(0,T;\EAdual)} \norm{(P_n-I)\psi}_\EA\\
	&\lesssim \norm{F(z_n)}_{L^1(0,T;\LalphaPlusEinsDual)} \norm{(P_n-I)\psi}_\EA\\
	&\lesssim \norm{ z_n}_{L^\alpha(0,T;\LalphaPlusEins)}^\alpha \norm{(P_n-I)\psi}_\EA\\
	&\lesssim \norm{ z_n}_{L^{\alpha+1}(0,T;\LalphaPlusEins)}^\alpha \norm{(P_n-I)\psi}_\EA\longrightarrow 0, \quad n\to \infty.
	\end{align*}			
	By Assumption $\eqref{nonlinearAssumptions}$ (see $\eqref{deriveNonlinearBound}$), we get
	\begin{align*}
	\norm{  F(z_n(s))-F(z(s))}_{\LalphaPlusEinsDual}
	&\lesssim \left(\norm{z_n(s)}_\LalphaPlusEins+\norm{z(s)}_\LalphaPlusEins\right)^{\alpha-1} \norm{z_n(s)-z(s)}_\LalphaPlusEins		
	\end{align*}
	for $s\in [0,T].$
	Now, we  apply H\"{o}lder's inequality in time with $\frac{1}{\alpha+1}+\frac{1}{\alpha+1}+\frac{\alpha-1}{\alpha+1}=1$
	\begin{align*}
	\norm{F(z_n)-F(z)}_{L^1(0,T; \LalphaPlusEinsDual)}
	&\le T^\frac{1}{\alpha+1}\left(\norm{z_n}_\LalphaPlusEinsAlphaPlusEins+ \norm{z}_\LalphaPlusEinsAlphaPlusEins\right)^{\alpha-1}\\&\hspace{1cm} \norm{z_n-z}_\LalphaPlusEinsAlphaPlusEins\rightarrow 0,\qquad n\to \infty.
	\end{align*}
	This leads to the last claim.
\end{proof}		

By the application of the Skorohod-Jakubowski Theorem, we have replaced the Galerkin solutions $u_n$ by the processes $v_n$ on $\tilde{\Omega}.$ Now, we want to transfer the properties given by the Galerkin equation $\eqref{galerkinEquation}.$
Therefore, we define the process $N_n: \tilde{\Omega} \times [0,T] \rightarrow H_n$  by
\begin{align*}
N_n(t)=-v_n(t)+ P_n u_0&+ \int_0^t \left[-\im A v_n(s)-\im P_n F(v_n(s))+\mu_n(v_n(s))\right] \df s
\end{align*}
for $n\in \N$ and $t\in[0,T]$ and in the following lemma, we prove its martingale property. Note that in this section, we consider ${H}$ as a real Hilbert space equipped with the real scalar product
$\Real \skpH{u}{v}$ for $u,v\in{H}$
in order to be consistent with the martingale theory from \cite{daPrato} we use.

\begin{Lemma}\label{NnIsMartingale}
	For each $n\in \N,$ the process $N_n$ is  an ${H}$-valued continuous square integrable martingale w.r.t the filtration $\tilde{\F}_{n,t}:=\sigma \left(v_n(s): s\le t\right).$ The quadratic variation of $N_n$ is given by
	\begin{align*}
	\quadVar{N_n}_t\psi= \sumM \int_0^t \im S_n B_m S_n v_n(s) \skpHReal{S_n B_m S_n v_n(s)}{\psi} \df s
	\end{align*}
	for all $\psi \in {H}.$
\end{Lemma}

\begin{proof}
	Fix $n\in\N.$ We define  $M_n: \Omega \times [0,T] \rightarrow H_n$ by
	\begin{align*}
	M_n(t):=-u_n(t)+ P_n u_0&+ \int_0^t \left[-\im A u_n(s)-\im P_n F(u_n(s))+\mu_n(u_n(s))\right] \df s
	\end{align*}
	for $t\in[0,T].$ Since $u_n$ is a solution of the Galerkin equation $\eqref{galerkinEquation}$, we obtain  the representation
	\begin{align*}
	M_n(t)=  \im \int_0^t S_n B_m (S_n u_n(s)) \df W(s)
	\end{align*}
	$\Prob$-a.s. for all $t\in[0,T].$ The estimate
	\begin{align*}
	\E \left[\sumM \int_0^T \norm{S_n B_m S_n u_n(s)}_{H}^2 \df s\right]&\le \sumM \norm{B_m}_{{\mathcal{L}({H})}}^2 \E \left[\int_0^T \norm{u_n(s)}_{H}^2 \df s\right]\\
	&\le T \sumM \norm{B_m}_{{\mathcal{L}({H})}}^2 \norm{u_0}_{H}^2
	<\infty
	\end{align*}
	yields, that $M_n$ is a square integrable continuous martingale w.r.t. the filtration $\left(\F_t\right)_{t\in[0,T]}.$
	From the definition of $M_n$ we get, that for each $t\in [0,T],$ $M_n(t)$ is measurable w.r.t. the smaller $\sigma$-field
	$\F_{n,t}:=\sigma \left(u_n(s): s\le t\right).$\\
	The adjoint of the operator $\varPhi_n(s):=\im S_n B(S_n u_n(s)):Y\to {H}$ for $s\in[0,T]$ is given by \\$\varPhi^*(s)\psi= \sumM \skpHReal{\im S_n B_m S_n u_n(s)}{\psi} f_m$ for $\psi\in {H}.$  Therefore
	\begin{align*}
	\varPhi(s)\varPhi^*(s)\psi= \sumM \skpHReal{\im S_n B_m S_n u_n(s)}{\psi} \im S_n B_m S_n u_n(s)
	\end{align*}
	for $\psi\in {H}$ and $s\in [0,T].$
	Hence, $M_n$ is a $\left(\F_{n,t}\right)$-martingale with  quadratic variation
	\begin{align*}
	\quadVar{M_n}_t\psi= \sumM \int_0^t \im S_n B_m S_n u_n(s) \skpHReal{\im S_n B_m S_n u_n(s)}{\psi} \df s
	\end{align*}
	for $\psi \in{H}$ (see \cite{daPrato}, Theorem 4.27).
	This property can be rephrased as
	\begin{align*}
	\E \left[ \skpHReal{M_n(t)-M_n(s)}{\psi} h(u_n|_{[0,s]})\right]=0
	\end{align*}
	and
	\begin{align*}
	\E \Bigg[& \Bigg(\skpHReal{M_n(t)}{\psi}\skpHReal{M_n(t)}{\varphi}-\skpHReal{M_n(s)}{\psi}\skpHReal{M_n(s)}{\varphi}\\
	&\hspace{1 cm}-\sumM \int_0^t  \skpHReal{\im S_n B_m S_n u_n(s)}{\psi} \skpHReal{\im S_n B_m S_n u_n(s)}{\varphi} \df s\Bigg) h(u_n|_{[0,s]})\Bigg]=0
	\end{align*}
	for all $\psi, \varphi \in {H}$ and bounded, continuous functions $h$ on $C([0,T],{H}).$
	
	We use the identity of the laws of $u_n$ and $v_n$ on $C([0,T],H_n)$ to obtain
	\begin{align*}
	\Etilde \left[ \skpHReal{N_n(t)-N_n(s)}{\psi} h(v_n|_{[0,s]})\right]=0
	\end{align*}
	and
	\begin{align*}
	\Etilde \Bigg[ \Bigg(&\skpHReal{N_n(t)}{\psi}\skpHReal{N_n(t)}{\varphi}-\skpHReal{N_n(s)}{\psi}\skpHReal{N_n(s)}{\varphi}\\
	&\hspace{1 cm}-\sumM \int_0^t  \skpHReal{\im S_n B_m S_n v_n(s)}{\psi} \skpHReal{\im S_n B_m S_n v_n(s)}{\varphi} \df s\Bigg) h(v_n|_{[0,s]})\Bigg]=0
	\end{align*}		
	for all $\psi, \varphi \in {H}$ and bounded, continuous functions $h$ on $C([0,T],H_n).$ Hence, $N_n$ is  a continuous square integrable martingale w.r.t $\tilde{\F}_{n,t}:=\sigma \left(v_n(s): s\le t\right)$ and the quadratic variation is given as claimed in the lemma.
\end{proof}

We define a process $N$ on $\tilde{\Omega} \times [0,T]$ by
\begin{align*}
N(t):=-v(t)+ u_0+ \int_0^t \left[-\im A v(s)-\im  F(v(s))+\mu(v(s))\right] \df s, \quad t\in[0,T].
\end{align*}
By Proposition $\ref{PropAlmostSureConvergence}$, we infer that $v\in \StetigEAdual$ almost surely and
\begin{align*}
\norm{F(v)}_{L^\infty(0,T;\EAdual)}
\lesssim \norm{F(v)}_{L^\infty(0,T;\LalphaPlusEinsDual)}
= \norm{v}_{L^\infty(0,T;\LalphaPlusEins)}^\alpha<\infty \hspace{0,3cm} \text{a.s.}
\end{align*}
\begin{align*}
\norm{A  v}_{L^\infty(0,T;\EAdual)}\le \norm{  v}_{L^\infty(0,T;\EA)}<\infty \hspace{0,3cm} \text{a.s.}
\end{align*}
Because of $\mu\in \mathcal{L}(\EAdual),$ we infer that $\mu(v)\in \StetigEAdual$ almost surely. Hence,  $N$ has $\EAdual$-valued continous paths.

Let $\iota: \EA \hookrightarrow {H}$ be the usual embedding, $\iota^*: {H} \rightarrow \EA$ its Hilbert-space-adjoint,
i.e. $\skpH{\iota u}{v}=\skp{u}{\iota^* v}_\EA$ for $u\in\EA$ and $v\in{H}.$
Further, we set $L:=\left(\iota^*\right)': \EAdual \rightarrow {H}$  as the dual operator of $\iota^*$ with respect to the Gelfand triple $\EA\hookrightarrow H\eqsim H^*\hookrightarrow\EAdual.$

In the next Lemma, we use the martingale property of $N_n$ for $n\in\N$ and a limiting process based on Proposition $\ref{PropAlmostSureConvergence}$ and Lemma $\ref{LemmaConvergences}.$	to conclude that $L N$ is also an $H$-valued martingale.		

\begin{Lemma}\label{NisMartingale}
	The process $L N$ is  an ${H}$-valued continuous square integrable martingale with respect to the filtration $\tilde{\Filtration}=\left(\tilde{\F}_t\right)_{t\in[0,T]},$ where $\tilde{\F}_{t}:=\sigma \left(v(s): s\le t\right).$ The quadratic variation is given by
	\begin{align*}
	\quadVar{L N}_t\zeta= \sumM \int_0^t\im  L B_m v(s) \skpHReal{\im L B_m v(s)}{\zeta} \df s
	\end{align*}
	for all $\zeta \in {H}.$	
\end{Lemma}

\begin{proof}
	\emph{Step 1:} Let $t\in [0,T].$ We will first show that
	$\Etilde \left[\norm{N(t)}_\EAdual^2\right]<\infty.$
	By Lemma $\ref{LemmaConvergences}$, we have $N_n(t) \to N(t)$ almost surely in $\EAdual$ for $n\to \infty.$
	By the Davis inequality for continuous martingales (see \cite{Pardoux76}), Lemma $\ref{NnIsMartingale}$ and Proposition $\ref{PropAlmostSureConvergence}$ , we conclude
	\begin{align}\label{VitaliBoundAlpha}
	\Etilde \left[\sup_{t\in [0,T]}\norm{N_n(t)}_{H}^{\alpha+1}\right]
	&\lesssim \Etilde \left[\left(\sumM \int_0^T \norm{S_n B_m S_n v_n(s)}_{H}^2 \df s\right)^{\frac{\alpha+1}{2}}\right]\nonumber\\
	&\le \left(\sumM \norm{B_m }_{{\mathcal{L}({H})}}^2\right)^{\frac{\alpha+1}{2}} \Etilde \left[ \left(\int_0^T \norm{ v_n(s)}_{H}^2 \df s\right)^{\frac{\alpha+1}{2}}\right]\nonumber\\
	&\lesssim   \Etilde \left[ \int_0^T \norm{ v_n(s)}_{H}^{\alpha+1} \df s\right]\lesssim   \Etilde \left[ \int_0^T \norm{ v_n(s)}_\LalphaPlusEins^{\alpha+1} \df s\right]\nonumber\\
	&\le T \sup_{n\in\N}\Etilde \left[ \norm{v_n}_{L^\infty(0,T;\LalphaPlusEins)}^{\alpha+1}\right] \le T C.
	\end{align}
	Since $\alpha+1>2$, we deduce $N(t)\in L^2(\tilde{\Omega},\EAdual)$ by the Vitali Theorem and $N_n(t)\to N(t)$ in $L^2(\tilde{\Omega},\EAdual)$ for $n\to \infty.$\\

	\emph{Step 2:}
	Let $\psi, \varphi\in \EA$ and $h$ be a bounded continuous function on $\StetigEAdual.$ \\
	For $0\le s\le t\le T,$  we define the random variables
	\begin{align*}
	f_n(t,s):=&\skpHReal{N_n(t)-N_n(s)}{\psi} h(v_n|_{[0,s]}),\qquad
	f(t,s):=\dualityReal{N(t)-N(s)}{\psi} h(v|_{[0,s]}).
	\end{align*}
	The $\tildeProb$-a.s.-convergence $v_n \to v$ in $Z_T$ for $n\to \infty$ yields by Lemma $\ref{LemmaConvergences}$
	$f_n(t,s)\to f(t,s)$
	$\tildeProb$-a.s. for all $0\le s\le t\le T.$ We use $\left(a+b\right)^p\le 2^{p-1} \left(a^p+b^p\right)$ for $a,b\ge 0$ and $p\ge 1$ and the estimate $\eqref{VitaliBoundAlpha}$ for
	\begin{align*}
	\Etilde \vert f_n(t,s)\vert^{\alpha+1}&\le 2^\alpha \norm{h}_\infty^{\alpha+1} \norm{\psi}_{H}^{\alpha+1} \Etilde \left[\norm{N_n(t)}_{H}^{\alpha+1}+\norm{N_n(s)}_{H}^{\alpha+1}\right]\\
	&\le 2^\alpha \norm{h}_\infty^{\alpha+1} \norm{\psi}_{H}^{\alpha+1} 2 T C
	\end{align*}
	In view of the Vitali Theorem, 			we get
	\begin{align*}
	0=\lim_{n\to \infty}\Etilde f_n(t,s)= \Etilde f(t,s), \qquad 0\le s\le t\le T.
	\end{align*}\\
	
	\emph{Step 3:} For $0\le s\le t\le T,$ we define			
	\begin{align*}
	g_{1,n}(t,s):=\Big(\skpHReal{N_n(t)}{\psi}&\skpHReal{N_n(t)}{\varphi}-\skpHReal{N_n(s)}{\psi}\skpHReal{N_n(s)}{\varphi}\Big) h(v_n|_{[0,s]})
	\end{align*}
	and
	\begin{align*}
	g_1(t,s):=\Big(\dualityReal{N(t)}{\psi}&\dualityReal{N(t)}{\varphi}-\dualityReal{N(s)}{\psi}\dualityReal{N(s)}{\varphi}\Big) h(v|_{[0,s]}).
	\end{align*}
	By Lemma $\ref{LemmaConvergences},$ we obtain
	$g_{1,n}(t,s)\to g_1(t,s)$
	$\tildeProb$-a.s. for all $0\le s\le t\le T.$ In order to get uniform integrability, we set $r:={\frac{\alpha+1}{2}}>1$ and estimate
	\begin{align*}
	\Etilde \vert g_{1,n}(t,s)\vert^r\le& 2^{r} \norm{h}_\infty^{r}  \Etilde \left[\vert \skpHReal{N_n(t)}{\psi}\skpHReal{N_n(t)}{\varphi}\vert^r+\vert \skpHReal{N_n(s)}{\psi}\skpHReal{N_n(s)}{\varphi}\vert^r\right]\\
	\le& 2^r \norm{h}_\infty^r  \norm{\psi}_{H}^r \norm{\varphi}_{H}^r \Etilde \left[\norm{N_n(t)}_{H}^{\alpha+1}+\norm{N_n(s)}_{H}^{\alpha+1}\right]
	\le 2^r \norm{h}_\infty^r  \norm{\psi}_{H}^r \norm{\varphi}_{H}^r 2 T C,
	\end{align*}
	where we used $\eqref{VitaliBoundAlpha}$ again.			
	As above, Vitali's Theorem yields
	\begin{align*}
	0=\lim_{n\to \infty}\Etilde g_{1,n}(t,s)= \Etilde g_1(t,s), \qquad 0\le s\le t\le T
	\end{align*} \\			
	
	\emph{Step 4:} For $0\le s\le t\le T,$  we define	
	\begin{align*}
	g_{2,n}(t,s)&:=h(v_n|_{[0,s]}) \sumM \int_s^t  \skpHReal{S_n B_m S_n v_n(\tau)}{\psi} \skpHReal{S_n B_m S_n v_n(\tau)}{\varphi}\df \tau \\
	g_{2}(t,s)&:=h(v|_{[0,s]}) \sumM \int_s^t  \dualityReal{B_m v(\tau)}{\psi} \dualityReal{B_m v(\tau)}{\varphi}\df \tau .
	\end{align*}
	Because of $h(v_n |_{[0,s]}) \to h(v|_{[0,s]})$ $\tildeProb$-a.s. and the continuity of the inner product $L^2([s,t]\times \N),$ the convergence
	\begin{align*}
	\skpHReal{S_n B_m S_n v_n}{\psi} \to \dualityReal{B_m v}{\psi}
	\end{align*}
	$\tildeProb$-a.s. in $L^2([s,t]\times \N)$ already implies
	$g_{2,n}(t,s)\to g_2(t,s)$ $\tildeProb$-a.s. Therefore, we consider 			
	\begin{align*}
	\Vert&\skpHReal{S_n B_m S_n v_n}{\psi} - \dualityReal{B_m v}{\psi}\Vert_\LzweiTimeSum\\
	&\le \norm{\skpHReal{B_m S_n v_n}{\left(S_n-I\right)\psi} }_\LzweiTimeSum+
	\norm{\skpHReal{ v_n}{\left(S_n-I\right)B_m \psi} }_\LzweiTimeSum \\
	&\hspace{1cm}+\norm{ \dualityReal{B_m \left(v_n-v\right)}{\psi}}_\LzweiTimeSum \\
	&\le \norm{B_m S_n v_n}_\LzweiTimeSumHminusEins \norm{\left(S_n-I\right)\psi}_\EA+\norm{\skpHReal{ v_n}{\left(S_n-I\right)B_m \psi} }_\LzweiTimeSum\\
	&\hspace{1cm}+\norm{\psi}_\EA		\norm{B_m (v_n-v)}_\LzweiTimeSumHminusEins\\
	&\le \left(\sumM \norm{B_m}_{{\mathcal{L}(\EA)}}^2\right)^{\frac{1}{2}}	T^\frac{1}{2} \norm{v_n}_\StetigEAdual \norm{\left(P_n-I\right)\psi}_\EA +\norm{\skpHReal{ v_n}{\left(S_n-I\right)B_m \psi} }_\LzweiTimeSum\\
	&\hspace{1cm}+\left(\sumM \norm{B_m}_{{\mathcal{L}(\EA)}}^2\right)^{\frac{1}{2}}	T^\frac{1}{2} \norm{v_n-v}_\StetigEAdual \norm{\psi}_\EA.
	\end{align*}
	The first and the third term tend to 0 as $n\to \infty$ by Proposition $\ref{PropAlmostSureConvergence}$ and for the second one, this follows by the estimate
	\begin{align*}
	\vert\skpHReal{ v_n(s)}{\left(S_n-I\right)B_m \psi}\vert^2 \le 4 \norm{v_n(s)}_\EAdual^2 \norm{B_m}_{{\mathcal{L}(\EA)}}^2 \norm{\psi}_\EA^2 \in L^1([s,t]\times \N)
	\end{align*}
	and Lebesgue's convergence Theorem.
	Hence, we  conclude
	\begin{align*}
	\Vert\skpHReal{S_n B_m S_n v_n}{\psi} - \dualityReal{B_m v}{\psi}\Vert_\LzweiTimeSum \to 0
	\end{align*}
	$\tildeProb$-a.s. as $n\to \infty.$ Furthermore, we estimate
	\begin{align*}
	\sumM \int_s^t \vert \skpHReal{S_n B_m S_n v_n(\tau)}{\psi} \vert^2 \df \tau \le&  \int_0^T\norm{v_n(\tau)}_\EAdual^2 \df \tau\norm{\psi}_\EA^2  \sumM \norm{B_m}_{{\mathcal{L}(\EA)}}^2
	\end{align*}
	and continue with $r:=\frac{\alpha+1}{2}>1$ and
	\begin{align*}
	\Etilde \vert g_{2,n}(t,s)\vert^r
	&\le \Etilde \Big[\norm{\dualityReal{S_n B_m S_n v_n}{\psi}}_{L^2([s,t]\times \N)} ^r\norm{\dualityReal{S_n B_m S_n v_n}{\varphi}}_{L^2([s,t]\times \N)}^r \vert h(v_n|_{[0,s]})\vert^r\Big]\\
	&\le \Etilde \left[\left(\int_0^T\norm{v_n(\tau)}_\EAdual^2 \df \tau\right)^r\right] \norm{\psi}_\EA^r  \norm{\varphi}_\EA^r \left(\sumM \norm{B_m}_{{\mathcal{L}(\EA)}}^2\right)^r \norm{h}_\infty^r\\
	&\lesssim \Etilde \left[\int_0^T\norm{v_n(\tau)}_\EAdual^{\alpha+1} \df \tau\right]
	\lesssim \sup_{n\in\N} \Etilde \left[ \norm{v_n}_\LalphaPlusEinsAlphaPlusEins^{\alpha+1}\right]\le C T .
	\end{align*}
	Using Vitali's Theorem, we obtain
	\begin{align*}
	\lim_{n\to \infty} \Etilde \left[g_{2,n}(t,s)\right]=\Etilde \left[g_{2}(t,s)\right],\qquad 0\le s\le t\le T.
	\end{align*} \\
	
	\emph{Step 5:} From step 2, we have
	\begin{align}\label{NmartingaleHminusDrei1}
	\Etilde \left[\dualityReal{N(t)-N(s)}{\psi} h(u|_{[0,s]})\right]=0
	\end{align}
	and step 3, step 4 and Lemma $\ref{NnIsMartingale}$ yield
	\begin{align}\label{NmartingaleHminusDrei2}
	\Etilde \Bigg[ \Bigg(\dualityReal{N(t)}{\psi}&\dualityReal{N(t)}{\varphi}-\dualityReal{N(s)}{\psi}\dualityReal{N(s)}{\varphi}\nonumber\\
	&+\sumM \int_s^t  \dualityReal{B_m v(\tau)}{\psi} \dualityReal{B_m v(\tau)}{\varphi} \df \tau\Bigg) h(v|_{[0,s]})\Bigg]=0.
	\end{align}

	Now, let $\eta,\zeta \in{H}.$ Then $\iota^* \eta,\iota^* \zeta \in \EA$ and for all $z\in \EAdual,$ we have $\skpHReal{L z}{\eta}=\dualityReal{z}{\iota^* \eta}.$ By the first step, $L N$ is a continuous, sqare integrable process in ${H}$ and the identities $\eqref{NmartingaleHminusDrei1}$ and $\eqref{NmartingaleHminusDrei2}$ imply
	\begin{align*}
	\Etilde \left[\skpHReal{L N(t)-L N(s)}{\eta} h(u|_{[0,s]})\right]=0
	\end{align*}
	and
	\begin{align*}
	\Etilde \Bigg[ \Bigg(\skpHReal{L N(t)}{\eta}&\skpHReal{L N(t)}{\zeta}-\skpHReal{L N(s)}{\eta}\skpHReal{L N(s)}{\zeta}\\
	&+\sumM \int_s^t  \skpHReal{L  B_m v(\tau)}{\eta} \skpHReal{L B_m v(\tau)}{\zeta} \df \tau\Bigg) h(v|_{[0,s]})\Bigg]=0.
	\end{align*}
	Hence, $L N$ is a continuous, square integrable martingale in ${H}$ with respect to the $\tilde{\F}_{n,t}:=\sigma \left(v(s): s\le t\right)$ and quadratic variation						
	\begin{align*}
	\quadVar{L N}_t\zeta= \sumM \int_0^t \im L B_m v(s) \skpHReal{\im L B_m v(s)}{\zeta} \df s
	\end{align*}
	for all $\zeta \in {H}.$	
\end{proof}

Finally, we can prove our main result Theorem $\ref{MainTheorem}$ using the Martingale Representation Theorem from \cite{daPrato}, Theorem 8.2.	

\begin{proof}[Proof of Theorem 1.1]
	We choose  $H=L^2(M),$ $Q=I$ and $\varPhi(s):= \im L B\left( v(s)\right)$ for all  $s\in[0,T].$ The adjoint $\varPhi(s)^*$ is given by
	$\varPhi(s)^*\zeta:= \sumM \skpHReal{\im L B_m v(s)}{\zeta} f_m$
	and  hence,
	\begin{align*}
	\left(\varPhi(s) Q^{\frac{1}{2}}\right)\left(\varPhi(s) Q^{\frac{1}{2}}\right)^*\zeta =\varPhi(s)\varPhi(s)^*\zeta =\sumM \skpHReal{\im L B_m v(s)} {\zeta}\im L B_m v(s)
	\end{align*}
	for $\zeta \in {H}.$
	Clearly, $v$ is continuous in $\EAdual$ and adapted to the filtration $\tilde{\Filtration}$ given by $\tilde{\mathcal{F}}_t=\sigma \left(v(s): 0\le s\le t\right)$ for $s\in [0,T].$ Hence, $\varPhi$ is continuous in ${H}$ and adapted to $\tilde{\Filtration}$ and therefore progressively measurable.\\ 	
	By an application of Theorem 8.2 in \cite{daPrato} to the process $L N$ from Lemma $\ref{NisMartingale},$ we obtain a cylindrical Wiener process $\tilde{W}$ on $Y$ defined on a probability space
	\begin{align*}
	\left(\Omega',\F',\Prob'\right)=\left(\tilde{\Omega} \times \tilde{\tilde{\Omega}}, \tilde{\F}\otimes \tilde{\tilde{\F}},  \tilde{\Prob}\otimes\tilde{\tilde{\Prob}}\right)
	\end{align*}
	with
	\begin{align*}
	L N(t)=\int_0^t \varPhi(s) \df \tilde{W}(s)=\int_0^t \im L B\left( v(s)\right) \df \tilde{W}(s)
	\end{align*}
	for $t\in [0,T].$ The estimate
	\begin{align*}
	\norm{B v}_{L^2([0,T]\times \Omega,\HS(Y,\EAdual))}^2=&\E \int_0^T \sumM \norm{B_m v(s)}_\EAdual^2 \df s
	\lesssim\E \int_0^T \sumM \norm{B_m v(s)}_\EA^2 \df s\\
	\le& \E \int_0^T \left(\sumM \norm{B_m}_{{\mathcal{L}(\EA)}}^2\right) \norm{v(s)}_\EA^2 \df s
	\lesssim \E \int_0^T  \norm{v(s)}_\EA^2 \df s\\
	\le& T \norm{v}_{L^2(\Omega,\LinftyEA)}^2\le T C		
	\end{align*}
	yields that the stochastic integral $\int_0^\cdot  B\left( v(s)\right) \df \tilde{W}(s)$ is a continuous martingale in $\EAdual$ and using the continuity of the operator $L$, we get
	\begin{align*}
	\int_0^t \im L B\left( v(s)\right) \df \tilde{W}(s)=L \left(\int_0^t \im  B\left( v(s)\right) \df \tilde{W}(s)\right)
	\end{align*}
	for all $t\in [0,T].$
	The definition of $N$ and the injectivity of $L$ yield the equality
	\begin{align}\label{vIsSolution}		
	\int_0^t \im B  v(s) \df \tilde{W}(s)=-v(t)+ u_0+ \int_0^t \left[-\im A v(s)-\im  F(v(s))+\mu(v(s))\right] \df s
	\end{align}
	in $\EAdual$ for $t\in [0,T].$
	The weak continuity of the paths of $v$ in $\EA$ and the estimates for property $\eqref{propertySolution}$ have already been shown in Proposition $\ref{PropAlmostSureConvergence}.$ Hence, the system $\left(\tilde{\Omega},\tilde{\F},\tilde{\Prob},\tilde{W},\tilde{\Filtration},v\right)$ is a martingale solution of equation $\eqref{ProblemStratonovich}.$
\end{proof} 

It remains to prove the mass conservation from Theorem \ref{MainTheorem}. In Proposition $\ref{MassEstimateGalerkinSolution},$ we proved a similar result for the approximating equation. Since this property is not invariant under the limiting procedure from above, we have to repeat the calculation in infinite dimensions and justify it by a regularization procedure.

\begin{Prop}\label{MassConservationMartingaleSolution}
	Let $\left(\tilde{\Omega},\tilde{\F},\tilde{\Prob},\tilde{W},\tilde{\Filtration},u\right)$ be a martingale solution of $\eqref{ProblemStratonovich}.$
	Then, we have 
	$\norm{u(t)}_{L^2}=\norm{u_0}_{L^2}$
	almost surely for all $t\in[0,T].$
\end{Prop}

\begin{proof}
	\emph{Step 1.}
	Given $\lambda>0,$ we define $\Yosida:=\lambda\left(\lambda+A\right)^{-1}.$ Using the series representation, one can verify
	\begin{align}\label{YosidaProperties}
	&\Yosida f \to f \quad \text{in}\quad {X},\quad \lambda \to \infty, \quad f\in{X}\nonumber\\
	&\hspace{1,5cm}\norm{\Yosida}_{{\mathcal{L}(X)}}\le 1
	\end{align}
	for $X\in \left\{{H}, \EA, \EAdual\right\}.$ Moreover, $\Yosida(\EAdual)=\EA$ and hence,
	the equation
	\begin{align}\label{ItoFormSolutionDifferenceYosida}
	\Yosida u(t)=  \Yosida u_0+\int_0^t \left[-\im  \Yosida A u(s)-\im \Yosida F(u(s))+\Yosida \mu(u(s))\right] \df s- \im \int_0^t \Yosida B u(s) \df \tilde{W}(s)
	\end{align}	
	holds almost surely in $\EA$ for all $t\in[0,T].$ The function $\mass: {{H}} \to \R$ defined by\\ $\mass(v):=\norm{v}_{{H}}^2$ is twice continuously Fr\'{e}chet-differentiable with
	\begin{align*}
	\mass'[v]h_1&= 2 \Real \skpH{v}{ h_1}, \qquad
	\mass''[v] \left[h_1,h_2\right]= 2 \Real \skpH{ h_1}{h_2}
	\end{align*}
	for $v, h_1, h_2\in {{H}}.$ Therefore, we get
	\begin{align}\label{LzweiDifferenceYosida}
	\norm{\Yosida u(t)}_{{H}}^2=&\norm{\Yosida u_0}_{H}^2+2 \int_0^t \Real \skpH{\Yosida u(s)}{-\im \Yosida A u(s)-\im\Yosida F(u(s))+\Yosida\mu(u(s))} \df s\nonumber\\
	&- 2 \int_0^t \Real \skpH{\Yosida u(s)}{\im\Yosida B u(s) \df \tilde{W}(s)}
	+\sumM \int_0^t
	\Vert \Yosida B_m u(s)\Vert_{{H}}^2\df s
	\end{align}
	almost surely for all $t\in[0,T].$\\
	
	\emph{Step 2.} In the following, we deal with the behaviour of the terms  in $\eqref{LzweiDifferenceYosida}$ for  $\lambda\to\infty.$
	Since $\Yosida$ and $A$ commute, we get
	\begin{align}\label{CancellationYosida}
	\Real\skpH{\Yosida u(s)}{-\im \Yosida A u(s)}=\Real\skpH{\Yosida u(s)}{-\im A \Yosida  u(s)}=0,\quad s\in [0,T],\quad \lambda>0.
	\end{align}
	For $s\in [0,T],$ we have
	\begin{align}\label{KonvergenzRestterme}
	&\Real \skpH{\Yosida u(s)}{-\im\Yosida F(u(s))}\to
	\Real \duality{ u(s)}{-\im F(u(s))}=0\nonumber\\
	&\Real \skpH{\Yosida u(s)}{\Yosida \mu(u(s))}\to
	\Real \skpH{ u(s)}{\mu(u(s))},\qquad \lambda\to\infty.
	\end{align}
	by $\eqref{YosidaProperties}.$
	In order to apply the dominated convergence Theorem by Lebesgue, we estimate
	\begin{align*}
	\vert \Real &\skpH{\Yosida u(s)}{-\im\Yosida  F(u(s))+\Yosida\mu(u(s))}\vert\nonumber\\
	&\hspace{2cm}\le \norm{u(s)}_{\EA} \bigNorm{-\im F(u(s))+\mu(u(s))}_{\EAdual}\\
	&\hspace{2cm}\lesssim \norm{u(s)}_{\EA} \left(\norm{F(u(s))}_{\LalphaPlusEinsDual}+\sumM \norm{B_m}_{\mathcal{L}({H})}^2\norm{u(s)}_{{H}}\right)
	\\
	&\hspace{2cm}\lesssim \norm{u(s)}_{\EA} \left(\norm{u(s)}_{\LalphaPlusEins}^\alpha+\norm{u(s)}_{{H}}\right)
	\\&\hspace{2cm}\lesssim \norm{u(s)}_\EA^{\alpha+1}+\norm{u(s)}_\EA^2
	\end{align*}
	using $\eqref{YosidaProperties}$ and the Sobolev embeddings $\LalphaPlusEinsDual\hookrightarrow \EAdual$ and $\EA \hookrightarrow \LalphaPlusEins.$  \\
	Since $u\in C_w([0,T],\EA)$ almost surely and $C_w([0,T],\EA)\subset \LinftyEA$, we obtain
	\begin{align}
	\int_0^t \Real \skpH{\Yosida u(s)}{-\im\Yosida F(u_1(s))+\Yosida\mu(u(s))} \df s\nonumber\to \int_0^t \Real \skpH{u(s)}{\mu(u(s))}\df s
	,\qquad \lambda\to \infty,
	\end{align}
	almost surely for all $t\in[0,T].$	
	Moreover, the pointwise convergence
	\begin{align*}
	\Vert \Yosida B_m u(s)\Vert_{{H}} \to \Vert  B_m u(s)\Vert_{{H}}, \qquad m\in\N, \quad \text{f.a.a. }s\in[0,T]
	\end{align*}
	and the estimate
	\begin{align*}
	\Vert \Yosida B_m u(s)\Vert_{{H}}^2\le  \Vert B_m\Vert_{{\mathcal{L}({H})}}^2 \Vert u(s)\Vert_{{H}}^2 \in L^1([0,T]\times \N)
	\end{align*}
	lead to, by Lebesgue DCT, 
	\begin{align}\label{convergenceItoTerm}
	\sumM \int_0^t
	\Vert \Yosida B_m u(s)\Vert_{{H}}^2\df s \to \sumM \int_0^t
	\Vert  B_m u(s)\Vert_{{H}}^2\df s,\quad \lambda\to \infty
	\end{align}
	almost surely for all $t\in[0,T]$.
	For the stochastic term, we fix $K\in\N$ and define a stopping time  $\tau_K$ by 
	\begin{align*}
	\tau_K:=\inf \left\{t\in [0,T]: \norm{u(t)}_{{H}}>K\right\}.
	\end{align*}
	Then, we infer that
	\begin{align*}
	\Real \skpH{\Yosida u(s)}{\im\Yosida B_m u(s)}\to \Real \skpH{ u(s)}{\im B u(s)}=0 \quad \text{a.s.},\quad m\in\N, s\in[0,T]
	\end{align*}
	and
	\begin{align*}
	\mathbf{1}_{[0,\tau_K]}(s) \vert  \Real \skpH{\Yosida u(s)}{\im\Yosida B_m u(s)}\vert^2
	&\le  \mathbf{1}_{[0,\tau_K]}(s) \norm{u(s)}_{{H}}^4 \norm{B_m}_{{\mathcal{L}({H})}}^2\\
	&\le K^4 \norm{B_m}_{{\mathcal{L}({H})}}^2\in L^1(\tilde{\Omega}\times[0,T]\times\N)
	\end{align*}
	to get
	\begin{align*}
	\Etilde\sumM\int_0^{\tau_K}\left[\Real \skpH{\Yosida u(s)}{\im\Yosida B_m u(s)}\right]^2 \df s\to 0,\quad \lambda\to \infty,
	\end{align*}
	by Lebesgue. The It\^o isometry and the Doob inequality yield
	\begin{align*}
	\Etilde\left[\sup_{t\in[0,\tau_K]}\left\vert\int_0^t \Real \skpH{\Yosida u(s)}{\im\Yosida B u(s) \df W(s)} \right\vert^2 \right] \to 0,\quad \lambda \to \infty,
	\end{align*}
	After passing to a subsequence, we get
	\begin{align}\label{convergenceStochasticTerm}
	\int_0^t \Real \skpH{\Yosida u(s)}{\im\Yosida B u(s) \df W(s)}\to 0,\quad \lambda\to \infty,
	\end{align}
	almost surely in $\left\{t\le \tau_K\right\}.$ By
	\begin{align*}
	\bigcup_{K\in\N} \left\{t\le \tau_K\right\}=[0,T]\qquad \text{a.s.},
	\end{align*}
	we conclude that $\eqref{convergenceStochasticTerm}$ holds almost surely on $[0,T].$\\
	
	\emph{Step 3.}
	Using $\eqref{CancellationYosida},$  $\eqref{convergenceItoTerm}$ and $\eqref{convergenceStochasticTerm}$ in $\eqref{LzweiDifferenceYosida},$ we obtain
	\begin{align*}
	\norm{ u(t)}_{{H}}^2=&\norm{ u_0}_{{H}}^2+2 \int_0^t \Real \skpH{ u(s)}{\mu(u(s))} \df s
	+\sumM \int_0^t
	\Vert B_m u(s)\Vert_{{H}}^2\df s
	\end{align*}
	almost surely for all $t\in[0,T].$
	By the selfadjointness of $B_m,$ $m\in\N,$ we simplify
	\begin{align*}
	2 \Real \skpH{u(s)}{\mu(u(s))}=-\sumM \Real \skpH{u(s)}{B_m^2 u(s)}
	=-\sumM \norm{B_m u(s)}_{{H}}^2.
	\end{align*}
	Therefore, we have
	$\norm{ u(t)}_{{H}}^2=\norm{ u_0}_{{H}}^2$
	almost surely for all $t\in[0,T].$
\end{proof}

\section{Regularity and Uniqueness of solutions on 2d manifolds}\label{UniquenessSection}

In this section, we want to study pathwise uniqueness of solutions to \eqref{ProblemStratonovich} and we consider the case of a 2-dimensional riemannian manifold without boundary $M.$ We drop the assumption that $M$ is compact and replace it by
\begin{align}\label{ManifoldAssumption}
\text{$M$ is complete, has a positive injectivity radius and a bounded geometry.}
\end{align}
We refer to \cite{Triebel}, chapter 7,  for the definitions of the notions above and background references on differential geometry.
We equip $M$ with the canonical volume $\mu$ and suppose that $M$ satisfies the doubling property: For all $x\in \tilde{M}$ and $r>0,$ we have
$\mu(B(x,r))<\infty$ and
\begin{align}\label{doublingUniqueness}
\mu(B(x,2r))\lesssim \mu(B(x,r)).
\end{align}	
We emphasize that  \eqref{ManifoldAssumption} is satisfied by compact manifolds. Examples for manifolds with the property \eqref{doublingUniqueness} are given by compact manifolds and manifolds with non-negative Ricci-curvature, see \cite{CoulhonRuss}.

Let $A=-\Delta_g$ be the Laplace-Beltrami operator  $F=F_\alpha^\pm$ be the model nonlinearity from section \ref{ExampleSection}.
The proof is based on an additional regularity of the solution, which we obtain by applying the deterministic and the  stochastic Strichartz estimates from \cite{Bernicot} and \cite{BrzezniakStrichartz}.

In two dimensions, the mapping properties of the nonlinearity improve, as we will see in the first Lemma.

\begin{Lemma}\label{nonlinearEstimateUniqueness}
	Let $d=2,$  $\alpha>1,$ $s\in ( \smin,1]$ and $\tilde{s}\in (0,1-\alpha+s\alpha]\cap (0,1).$	
	Then, we have $F_\alpha^\pm: H^s(M)\to H^{\tilde{s}}(M)$  and
	\begin{align*}
	\norm{F_\alpha^\pm(u)}_{H^{\tilde{s}}}\lesssim \norm{u}_{H^s}^{\alpha},\qquad u\in H^s(M).
	\end{align*}
\end{Lemma}

\begin{proof}
	\emph{Step 1.} First, we consider the case $s=1.$
	Take $q\in [2,\infty)$ and $r\in(2,\infty)$ with
	\begin{align}\label{exponentsScaling}
	q\ge\frac{2(\alpha-1)}{1-\tilde{s}},\qquad \frac{1}{r}=\frac{1}{2}+\frac{\alpha-1}{q}.
	\end{align}
	Due to $d=2,$ we have $H^1(M)\hookrightarrow L^q(M)$ and by	\cite{Bolleyer}, Lemma, III. 1.4., we get
	\begin{align*}
	\norm{F_\alpha^\pm(u)}_{H^{1,r}}\lesssim \norm{u}_{H^1}^{\alpha},\qquad u\in H^1(M).
	\end{align*}		
	The condition $\eqref{exponentsScaling}$ yields
	\begin{align*}
	\tilde{s}-1\le-\frac{2(\alpha-1)}{q}=1-\frac{2}{r}
	\end{align*}
	and therefore, the assertion follows by applying the Sobolev embedding $H^{1,r}(M)\hookrightarrow H^{\tilde{s}}(M).$ \\
	
	\emph{Step 2.} Next, we consider $s\in (\smin,1).$
	Let $r=\frac{2}{(1-s)\alpha+s}\in (1,2)$ and $q=\frac{2}{1-s}\in (2\alpha,\infty).$
	Then, we have  $\frac{1}{r}=\frac{1}{2}+\frac{\alpha-1}{q}.$ Thus, we can apply \cite{ChristWeinstein}, Proposition 3.1, and obtain
	\begin{align}\label{fractionalChainRule}
	\norm{\vert \nabla\vert^s F_\alpha^\pm(u)}_{L^r}\lesssim \norm{u}_{L^q}^{\alpha-1} \norm{\vert \nabla\vert^s u}_{L^2}.
	\end{align}
	Furthermore, we have
	\begin{align*}
	s-1=-\frac{2}{q},\qquad s-1=\frac{s}{\alpha}-\frac{2}{r\alpha}\ge -\frac{2}{r\alpha},
	\end{align*}  which implies
	\begin{align*}
	H^s(\R^2)\hookrightarrow L^q(\R^2),\qquad H^s(\R^2)\hookrightarrow L^{r\alpha}(\R^2).
	\end{align*}	
	Together with \eqref{fractionalChainRule} and $\norm{F_\alpha^\pm(u)}_{L^r}=\norm{u}_{L^{r\alpha}}^\alpha$ for $u\in L^{r\alpha}(\R^2),$ this implies
	\begin{align}\label{NonlinearityAtRegularityS}
	\norm{F_\alpha^\pm(u)}_{H^{s,r}(\R^2)}\lesssim \norm{u}_{H^s(\R^2)}^{\alpha},\qquad u\in H^s(\R^2).
	\end{align}
	Since we have the Sobolev embedding $H^{s,r}(\R^2)\hookrightarrow H^{\tilde{s}}(\R^2)$ as a consequence of $0<\tilde{s}\le 1-\alpha+s\alpha\le s,$ we obtain
	\begin{align*}
	\norm{F_\alpha^\pm(u)}_{H^{\tilde{s}}(\R^2)}\lesssim \norm{u}_{H^s(\R^2)}^{\alpha},\qquad u\in H^s(M).
	\end{align*}
	This completes the proof in the case $M=\R^2.$ 	For a general manifold $M$, the estimate follows by the definition of fractional Sobolev spaces via charts, see Appendix B.  	
\end{proof}

In the following Proposition, we  reformulate problem $\eqref{ProblemStratonovich}$ in a mild form and use this to show additional regularity properties of solutions of $\eqref{ProblemStratonovich}$. Let us therefore recall the notation
\begin{align*}
\mu=-\frac{1}{2}\sumM B_m^2.
\end{align*}

\begin{Prop}\label{extraqIntegrability}
	Assume $d=2$ and choose $2<p,q<\infty$ with
	\begin{align*}
	\frac{2}{p}+\frac{2}{q}=1.
	\end{align*}
	Let $\varepsilon\in (0,1),$ $\alpha>1,$ $s\in [1+\frac{1+\varepsilon}{q\alpha}-\frac{1}{\alpha},1],$ $r>1$ and $\beta:= \max\{\alpha,2\}.$ Let $\left(\tilde{\Omega},\tilde{\F},\tilde{\Prob},\tilde{W},\tilde{\Filtration},u\right)$ be a solution to $\eqref{ProblemStratonovich}$ with $F=F_\alpha^\pm$ and assume
	\begin{align}\label{originalSmoothnessSolutions}
	u\in L^{r\alpha}(\tilde{\Omega},L^\beta(0,T;H^s(M))).
	\end{align}
	Then, for each
	$\tilde{s}\in [\frac{1+\varepsilon}{q},1-\alpha+s\alpha]\cap (0,1),$  we have
	\begin{align}\label{extraIntegrabilityInclusion}
	u\in L^r(\tilde{\Omega}, C([0,T],H^{\tilde{s}}(M))\cap L^q(0,T;H^{\tilde{s}-\frac{1+\varepsilon}{q},p}(M)))
	\end{align}
	and almost surely in $H^{\tilde{s}}(M)$ for all $t\in [0,T]$
	\begin{align}\label{mildEquation}
	\im u(t)=\im e^{-\im tA}u_0+\int_0^t e^{-\im (t-\tau)A}F_\alpha^\pm(u(\tau))\df \tau
	+\int_0^t e^{-\im (t-\tau)A}\mu(u(\tau))\df \tau
	+\int_0^t e^{-\im (t-\tau)A}B(u(\tau))\df W(\tau).
	\end{align}
	
\end{Prop}

\begin{Remark}
	Of course, \eqref{extraIntegrabilityInclusion} also holds for  $\varepsilon\ge 1,$ but then $u\in L^r(\tilde{\Omega},L^q(0,T;H^{\tilde{s}-\frac{1+\varepsilon}{q},p}(M)))$ would be trivial by the Sobolev embedding $H^{\tilde{s}}(M)\hookrightarrow H^{\tilde{s}-\frac{1+\varepsilon}{q},p}(M).$ Being able to choose $\varepsilon\in (0,1)$ means a gain of regularity which will be used below via $H^{\tilde{s}-\frac{1+\varepsilon}{q},p}(M)\hookrightarrow \LInfty$ for an appropriate choice of the parameters.
\end{Remark}

\begin{proof}[Proof of Proposition \ref{extraqIntegrability}]
	\emph{Step 1.}	
	First, we will show that it is possible to rewrite the equation $\eqref{ItoFormSolution}$ from the definition of solutions for $\eqref{ProblemStratonovich}$  in the mild form $\eqref{mildEquation}.$ \\
	We note that for each $s_0<0$ the semigroup $\left(e^{-\im t A}\right)_{t\ge 0}$ on $L^2(M)$ extends to a semigroup $\left(T_{s_0}(t)\right)_{t\ge 0}$ with the generator $A_{s_0}$ that extends $A$ to $\D(A_{s_0})=H^{s_0+2}(M).$ To keep the notation simple, we also call this semigroup $\left(e^{-\im t A}\right)_{t\ge 0}.$\\
	We apply \dela{It\^{o}'s}{the It\^o} formula to $\varPhi \in C^{1,2}([0,t]\times H^{s-2}(M), H^{s-4}(M))$ defined by
	\begin{align*}
	\varPhi(\tau,x):=e^{-\im (t-\tau)A}x,\qquad \tau\in [0,t],\quad  x\in H^{s-2}(M)
	\end{align*}
	and obtain
	\begin{align*}
	\im u(t)=\im e^{-\im tA}u_0+\int_0^t e^{-\im (t-\tau)A}F_\alpha^\pm(u(\tau))\df \tau
	+\int_0^t e^{-\im (t-\tau)A}\mu(u(\tau))\df \tau
	+\int_0^t e^{-\im (t-\tau)A}B(u(\tau))\df W(\tau)
	\end{align*}
	almost surely in $H^{s-4}(M)$ for all $t\in [0,T].$ \\
	
	\emph{Step 2.}
	Using the Strichartz estimates from Lemma $\ref{StrichartzLemma}$ we deal with the free term and each convolution term on the right hand site to get $\eqref{extraIntegrabilityInclusion}$ and the identity $\eqref{mildEquation}$ in $H^{\tilde{s}}(M).$
	For this purpose, we define
	\begin{align*}
	Y_T:=L^q(0,T; H^{\tilde{s}-\frac{1+\varepsilon}{q},p}(M))\cap L^\infty(0,T;H^{\tilde{s}}(M)).
	\end{align*}
	By $\eqref{homogenousStrichartzEstimate-1}$ we obtain
	\begin{align*}
	\norm{e^{-\im tA}u_0}_{L^r(\tilde{\Omega},Y_T)}\lesssim \norm{u_0}_{H^{\tilde{s}}}\lesssim \norm{u_0}_{H^s}<\infty
	\end{align*}
	and by $\eqref{inhomogenousStrichartzEstimate}$ and Lemma  $\ref{nonlinearEstimateUniqueness}$, we get
	\begin{align*}
	\bigNorm{\int_0^t e^{-\im (t-\tau)A}F_\alpha^\pm(u(\tau))\df \tau}_{Y_T}\lesssim \norm{F_\alpha^\pm(u)}_{L^1(0,T; H^{\tilde{s}})}\lesssim \norm{u}_{L^\alpha(0,T; H^s)}^\alpha.
	\end{align*}
	Integration over $\tilde{\Omega}$ and $\eqref{originalSmoothnessSolutions}$ yields
	\begin{align*}
	\bigNorm{\int_0^t e^{-\im (t-\tau)A}F_\alpha^\pm(u(\tau))\df \tau}_{L^r(\tilde{\Omega},Y_T)}\lesssim \norm{u}_{L^{r\alpha}(\tilde{\Omega},L^\alpha(0,T; H^s))}^\alpha<\infty.
	\end{align*}
	To estimate the other convolutions, we need that  $\mu$ is bounded in $H^{\tilde{s}}(M)$ and $B$ is bounded from $H^{\tilde{s}}(M)$ to $\HS(Y,H^{\tilde{s}}(M)).$ This can be deduced from the following estimate, which follows from
	complex interpolation (see \cite{Lunardi}, Theorem 2.1.6), H\"{o}lder's inequality and Assumption $\ref{stochasticAssumptions}$:
	\begin{align}\label{HsEstimateNoise}
	\sumM \norm{B_m}_{{\mathcal{L}(H^{\tilde{s}})}}^2&\le\sumM \norm{B_m}_{{\mathcal{L}(H^1)}}^{2{\tilde{s}}} \norm{B_m}_{{\mathcal{L}({H})}}^{2(1-{\tilde{s}})}\nonumber\\
	&\le \left(\sumM \norm{B_m}_{{\mathcal{L}(H^1)}}^{2}\right)^{{\tilde{s}}} \left(\sumM\norm{B_m}_{{\mathcal{L}({H})}}^{2}\right)^{1-{\tilde{s}}}<\infty.
	\end{align}
	Therefore, by $\eqref{inhomogenousStrichartzEstimate},$ $\eqref{HsEstimateNoise}$ and $\eqref{originalSmoothnessSolutions}$
	\begin{align*}
	\bigNorm{\int_0^t e^{-\im (t-\tau)A}\mu(u(\tau))\df \tau}_{L^r(\tilde{\Omega},Y_T)} &\lesssim \norm{\mu(u)}_{L^r(\tilde{\Omega}, L^1(0,T; H^{\tilde{s}}))}
	\lesssim \norm{u}_{L^r(\tilde{\Omega}, L^1(0,T; H^{\tilde{s}}))}
	\\&\lesssim \norm{u}_{L^{r\alpha}(\tilde{\Omega}, L^\beta(0,T; H^s)}<\infty.
	\end{align*}
	The estimates $\eqref{stochasticStrichartzEstimate},$ $\eqref{HsEstimateNoise}$ and $\eqref{originalSmoothnessSolutions}$ imply
	\begin{align*}
	\bigNorm{\int_0^t e^{-\im (t-\tau)A}B(u(\tau))\df W(\tau)}_{L^r(\tilde{\Omega},Y_T)}
	&\lesssim \norm{B(u)}_{L^r(\tilde{\Omega}, L^2(0,T; \HS(Y,H^{\tilde{s}}))}
	\lesssim \norm{u}_{L^r(\tilde{\Omega}, L^2(0,T; H^{\tilde{s}}))}\\
	&\lesssim \norm{u}_{L^{r\alpha}(\tilde{\Omega}, L^\beta(0,T; H^s))}<\infty.
	\end{align*}
	Hence, the mild equation $\eqref{mildEquation}$ holds almost surely in $H^{\tilde{s}}(M)$ for each $t\in [0,T]$ and thus, we get $\eqref{extraIntegrabilityInclusion}$ by the pathwise continuity of deterministic and stochastic integrals.
\end{proof}


As a preparation for the proof of pathwise uniqueness, we show a formula for the $L^2$-norm of the difference of two solutions of $\eqref{ProblemStratonovich}$.


\begin{Lemma}\label{SolutionDifferenceLemma}
	Let $\left(\tilde{\Omega},\tilde{\F},\tilde{\Prob},\tilde{W},\tilde{\Filtration},u_j\right),$ $j=1,2,$ be solutions of $\eqref{ProblemStratonovich}$ with $F=F_\alpha^\pm$ for $\alpha>1.$
	Then,
	\begin{align}\label{SolutionDifferenceFormula}
	\norm{u_1(t)-u_2(t)}_{L^2}^2=&2 \int_0^t \Real \skpLzwei{u_1(\tau)-u_2(\tau)}{-\im F_\alpha^\pm(u_1(\tau))+\im F_\alpha^\pm(u_2(\tau))} \df \tau
	\end{align}
	almost surely for all $t\in[0,T].$
\end{Lemma}

\begin{proof}
	The proof is similar to Proposition \ref{MassConservationMartingaleSolution}. In fact, it is even simpler, since the regularity of $F_\alpha^\pm$ due to Lemma $\ref{nonlinearEstimateUniqueness}$ simplifies the proof of the convergence for $\lambda\to \infty.$
\end{proof}


Finally, we are ready to prove the pathwise uniqueness of solutions to $\eqref{ProblemStratonovich}.$


\begin{Theorem}\label{Uniqueness2d}
	Let $d=2$ and $F(u)=F_\alpha^\pm(u)=\pm \vert u\vert^{\alpha-1}u$ with $\alpha\in (1,\infty).$
	Let $r> \alpha,$ $\beta\ge \max\{\alpha,2\}$ and
	\begin{align*}
	s \in \begin{cases}
	(1-\frac{1}{2\alpha},1] & \text{for } \alpha\in (1,3], \\
	(1-\frac{1}{\alpha(\alpha-1)},1] & \text{for } \alpha>3.
	\end{cases}
	\end{align*}
	Then, solutions of problem $\eqref{ProblemStratonovich}$ are pathwise unique in $L^{r}(\tilde{\Omega},L^\beta(0,T;H^s(M)))$,
	i.e. given two solutions $\left(\tilde{\Omega},\tilde{\F},\tilde{\Prob},\tilde{W},\tilde{\Filtration},u_j\right)$  with
	\begin{align*}
	u_j\in L^{r}(\tilde{\Omega},L^\beta(0,T;H^s(M))),
	\end{align*}
	for $j=1,2,$
	we have $u_1(t)=u_2(t)$ almost surely in ${L^2(M)}$ for all $t\in [0,T].$
\end{Theorem}

\begin{proof}\emph{Step 1.}
	Take two solutions $\left(\tilde{\Omega},\tilde{\F},\tilde{\Prob},\tilde{W},\tilde{\Filtration},u_j\right)$   of $\eqref{ProblemStratonovich}$ with $u_j\in L^{r}(\tilde{\Omega},L^\infty(0,T;H^s(M)))$ for $j=1,2,$ and define $w:=u_1-u_2.$ From Lemma $\ref{SolutionDifferenceLemma},$ we conclude		
	\begin{align*}
	\norm{w(t)}_{L^2}^2=&2 \int_0^t \Real \skpLzwei{w(\tau)}{-\im F(u_1(\tau))+\im F(u_2(\tau))} \df \tau
	\end{align*}
	almost surely for all $t\in[0,T].$
	The estimate
	\begin{align*}
	\vert F_\alpha^\pm(z_1)-F_\alpha^\pm(z_2)\vert\lesssim \left(\vert z_1\vert^{\alpha-1}+\vert z_2\vert^{\alpha-1}\right)\vert z_1-z_2\vert,\qquad z_1,z_2\in\C,
	\end{align*}
	yields
	\begin{align}\label{Gronwall}
	\norm{w(t)}_{L^2}^2\lesssim&  \int_0^t \int_M \vert w(\tau,x)\vert^2 \left[\vert{u_1(\tau,x)}\vert^{\alpha-1}+\vert{u_2(\tau,x)}\vert^{\alpha-1}\right] \df x \df \tau\nonumber\\
	\le& \int_0^t \norm{w(\tau)}_{L^2}^2 \left[\Vert{u_1(\tau)}\Vert_\LInfty^{\alpha-1}+\Vert{u_2(\tau)}\Vert_\LInfty^{\alpha-1}\right] \df \tau
	\end{align}
	almost surely for all $t\in[0,T].$ \\
	
	\emph{Step 2.}
	First, we deal with the case $\alpha\in (1,3].$
	By $s> 1-\frac{1}{2\alpha},$ we can choose $q>2$ and $\varepsilon\in (0,1)$ with
	\begin{align*}
	1-\frac{1}{2 \alpha}<1-\frac{1}{2 \alpha}+\frac{q-2+2\varepsilon}{2q\alpha}=1-\frac{1}{q\alpha}+\frac{\varepsilon}{q\alpha}<s.
	\end{align*}
	Hence, we have  $\frac{1+\varepsilon}{q}+1-\frac{2}{q}< 1-\alpha+s\alpha$ and in particular, there is $\tilde{s}\in (\frac{1+\varepsilon}{q}+1-\frac{2}{q}, 1-\alpha+s\alpha).$ 				
	If we choose $p>2$ according to $\frac{2}{p}+\frac{2}{q}=1,$ Proposition $\ref{PropertiesFractionalSobolev}$ leads to $H^{\tilde{s}-\frac{1+\varepsilon}{q},p}(M)\hookrightarrow \LInfty$ because of
	\begin{align*}
	\left(\tilde{s}-\frac{1+\varepsilon}{q}\right)-\frac{2}{p}=\tilde{s}-\frac{1+\varepsilon}{q}+\frac{2}{q}-1= \tilde{s}-\left(\frac{1+\varepsilon}{q}+1-\frac{2}{q}\right)>0.
	\end{align*}
	Moreover, we have $u_j\in L^q(0,T;H^{\tilde{s}-\frac{1+\varepsilon}{q},p}(M))$ almost surely for $j=1,2$ by Proposition \ref{extraqIntegrability}.
	Hence, the process $b$ defined by
	\begin{align}\label{definitionGronwallIntegrand}
	b(\tau):=\left[\norm{u_1(\tau)}_{L^\infty}^{\alpha-1}+\norm{u_2(\tau)}_{L^\infty}^{\alpha-1}\right],\qquad \tau\in[0,T],	
	\end{align}
	satisfies
	\begin{align}\label{bMajorante}
	\norm{b}_{L^1(0,T)}
	&\lesssim \norm{u_1}_{L^{q}(0,T;H^{s-\frac{1+\varepsilon}{q},p})}^{\alpha-1}+\norm{u_2}_{L^{q}(0,T;H^{s-\frac{1+\varepsilon}{q},p})}^{\alpha-1}
	<\infty \quad \text{a.s.},
	\end{align}
	where we used $q>2\ge \alpha-1$ and the H\"older inequality in time.
	Because of $\eqref{Gronwall},$ we can apply Gronwall's Lemma to get
	\begin{align*}
	u_1(t)=u_2(t)\quad \text{a.s. in ${L^2(M)}$ for all $t\in [0,T].$}
	\end{align*}
	
	\emph{Step 3.} Now, let $\alpha>3.$ Then, we set $q:=\alpha-1$ and choose $p>2$ with $\frac{2}{p}+\frac{2}{q}=1.$ Using $s> 1-\frac{1}{\alpha(\alpha-1)},$ we fix $\varepsilon\in (0,1)$ with
	\begin{align*}
	1-\frac{1}{\alpha(\alpha-1)}<1-\frac{1}{q\alpha}+\frac{\varepsilon}{q\alpha}<s.
	\end{align*}
	As above, we can choose $\tilde{s}\in (\frac{1+\varepsilon}{q}+1-\frac{2}{q}, 1-\alpha+s\alpha).$ We therefore get  $H^{\tilde{s}-\frac{1+\varepsilon}{q},p}(M)\hookrightarrow \LInfty$ and $u_j\in L^q(0,T;H^{\tilde{s}-\frac{1+\varepsilon}{q},p}(M))$ almost surely for $j=1,2.$ We obtain $b\in L^1(0,T)$ almost surely for $b$ from \eqref{definitionGronwallIntegrand} and Gronwall's Lemma implies
	\begin{align*}
	u_1(t)=u_2(t)\quad \text{a.s. in ${L^2(M)}$ for all $t\in [0,T].$}
	\end{align*}
\end{proof}

\begin{Remark}
	In \cite{BrzezniakStrichartz}, Brze{\'{z}}niak and Millet proved  pathwise uniqueness of solutions in the space $L^q(\Omega,C([0,T],H^1(M))\cap L^q([0,T],H^{1-\frac{1}{q},p}(M)))$ with $\frac{2}{q}+\frac{2}{p}=1$ and $q>\alpha+1.$ Since they used the deterministic Strichartz estimates from \cite{Burq} instead of \cite{Bernicot}, their result is restricted to compact manifolds $M.$
	Comparing the result in \cite{BrzezniakStrichartz} with Theorem \ref{Uniqueness2d} in the present article, we see that the assumptions of Theorem $\ref{Uniqueness2d}$ are weaker with respect to space and time. On the other hand, the assumptions on the required moments is slightly weaker in \cite{BrzezniakStrichartz}.
\end{Remark}

\begin{Remark}
	A similar Uniqueness-Theorem can also be proved on bounded domains in $\R^2$ using the Strichartz inequalities by Blair, Smith and Sogge from \cite{blairStrichartz}. We also want to mention the classical strategy by Vladimirov (see \cite{Vladimirov}, \cite{Ogawa}, \cite{OgawaOzawa} and \cite{Cazenave}) to prove uniqueness of $H^1$-solutions using Trudinger type inequalities which can be seen as the limit case of Sobolev's embedding, see also \cite{Adams}, Theorem 8.27. Since this proof only relies on the formula \eqref{SolutionDifferenceFormula} and the property of solutions to be in $H^1,$ it can be directly transfered to the stochastic setting. This strategy does not use Strichartz estimates, but it suffers from a restriction to $\alpha\in (1,3]$ and it cannot be transfered to $H^s$ for $s<1.$
\end{Remark}

Now, we give the definition of the concepts of strong solutions and uniqueness in law used in Corollary $\ref{CorollaryManifold2d}.$

\begin{Definition}\label{StrongSolutionDef}
	\begin{enumerate}
		\item[a)] 	Let $T>0$ and $u_0\in E_A.$
		Then, a \emph{strong solution} of the equation $\eqref{ProblemStratonovich}$ is  a continuous, $\tilde{\Filtration}$-adapted process with values in $\EAdual$ such that $u\in L^2(\Omega\times [0,T],\EAdual)$ and almost all paths are in $\weaklyContinousEA$
		with
		\begin{align*}
		u(t)=  u_0+ \int_0^t \left[-\im A u(s)-\im F(u(s))+\mu(u(s))\right] \df \tau- \im \int_0^t B u(s) \df W(s)
		\end{align*}
		almost surely in $\EAdual$ for all $t\in [0,T].$
		\item[b)] The solutions of $\eqref{ProblemStratonovichManifold}$ are called \emph{unique in law}, if for all martingale solutions \\ $\left(\Omega_j,\F_j,\Prob_j,W_j,\Filtration_j,u_j\right)$  with $u_j(0)=u_0,$ for $j=1,2,$ we have $\Prob_1^{u_1}=\Prob_2^{u_2}$ almost surely in $C([0,T],L^2(M)).$
	\end{enumerate}
	
\end{Definition}

We finish this section with the proof of Corollary $\ref{CorollaryManifold2d}.$

\begin{proof}[Proof of Corollary $\ref{CorollaryManifold2d}$]
	The existence of a martingale solution from Corollary $\ref{CorollaryManifold}$ and the pathwise uniqueness from Theorem  $\ref{Uniqueness2d}$ yield the assertion by \cite{OndrejatUniqueness}, Theorem 2 and 12.1.
\end{proof}

\appendix

\section{Auxilary Results from Functional Analysis}


In this appendix, we collect some abstract notions and results needed in Section $\ref{CompactnessSection}.$
For a  Banach space $X$ and $r>0,$ we denote
\begin{align*}
\mathbb{B}_X^r:=\left\{u\in X: \norm{u}_X\le r \right\}.
\end{align*}
The weak topology on $\mathbb{B}_X^r$ is metrizable if the dual $X^*$ is separable and a metric is given by
\begin{align*}
q(x_1,x_2)=\sum_{k=1}^\infty 2^{-k} \vert \duality{x_1-x_2}{x_k^*}\vert,\qquad x_1,x_2\in X,
\end{align*}
for a dense sequence $\left(x_k^*\right)_{k\in\N}\in \left(B_{X^*}^1\right)^\N,$ see \cite{Brezis}, Theorem 3.29. If $X$ is also separable, then $\stetigBallX$ is a complete separable metric space with metric
$\rho(u,v):=\sup_{t\in[0,T]} q(u(t),v(t))$
for $u,v\in \stetigBall.$

\begin{Definition}\label{DefinitionWeakTopologySpaces}
	We define
	\begin{align*}
	C_w([0,T],X):=\left\{ u: [0,T]\to X: [0,T]\ni t\to \duality{u(t)}{x^*}\in \C \text{ is cont. for all $x^*\in X^*$} \right\}
	\end{align*}
	and equip $C_w([0,T],X)$ with the locally convex topology induced by the family $P$ of seminorms given by
	\begin{align*}
	P:= \{p_{x^*}: x^* \in X^*\},\qquad p_{x^*}(u):=\sup_{t\in[0,T]} \left\vert \duality{u(t)}{x^*}\right\vert.
	\end{align*}
\end{Definition}

We continue with some auxiliary results.

\begin{Lemma}\label{ConvergenceInStetigBall}
	Let $r>0$ and $u_n, u\in C_w([0,T],X)$ with $\sup_{t\in[0,T]} \norm{u_n(t)}_X\le r$ and $u_n\to u$ in $C_w([0,T],X).$ Then, we have $u_n\to u$ in $\stetigBallX.$
\end{Lemma}

\begin{proof}
	By Lebesgue's Convergence Theorem,
	\begin{align*}
	\rho(u_n,u)\le \sum_{k=1}^\infty 2^{-k} \sup_{t\in [0,T]}\vert \duality{u_n(t)-u(t)}{x_k^*}\vert\to 0,\qquad n\to \infty,
	\end{align*}
	where we used the definition of convergence in $C_w([0,T],X)$ for fixed $k\in\N$ and
	\begin{align*}
	\sup_{t\in [0,T]}\vert \duality{u_n(t)-u(t)}{x_k^*}\le \left(\sup_{t\in [0,T]} \norm{u_n(t)}_X+\sup_{t\in [0,T]} \norm{u(t)}_X\right) \norm{x_k^*}_{X^*}\le 2 r.
	\end{align*}
\end{proof}

\begin{Lemma}[Strauss]\label{StraussLemma}
	Let $X,Y$ be Banach spaces with $X \hookrightarrow Y$ and $T>0.$ Then, we have the inclusion
	\begin{align*}
	L^\infty(0,T;X)\cap C_w([0,T],Y) \subset C_w([0,T],X).
	\end{align*}
\end{Lemma}	

\begin{proof}
	See \cite{Temam77}, Chapter 3, Lemma 1.4.
\end{proof}

\begin{Lemma}[Lions]\label{LionsLemma}
	Let $X,X_0,X_1$ be Banach spaces with $X_0 \hookrightarrow X \hookrightarrow X_1$ where the first embedding is compact. Assume furthermore that $X_0, X_1$ are reflexive and $p\in [1,\infty).$ Then, for each $\varepsilon>0$ there is $C_\varepsilon>0$ with
	\begin{align*}
	\Vert x \Vert_X^p \le \varepsilon \Vert x\Vert_{X_0}^p+C_\varepsilon \Vert x\Vert_{X_1}^p,\quad x\in X_0.
	\end{align*}
\end{Lemma}

\begin{proof}
	See \cite{Lions}, p. 58.
\end{proof}

\section{Sobolev Spaces on manifolds and Strichartz estimates}

In the Sections $\ref{ExampleSection}$ and $\ref{UniquenessSection},$ we need some results about Sobolev spaces on manifolds and their connection with the fractional domains of the Laplace-Beltrami operator. In this appendix, we recall the basic definitions and Sobolev embeddings. Moreover, we state the deterministic and stochastic Strichartz estimates for the Schr\"odinger group $\left(e^{\im t\Delta_g}\right)_{t\in\R}.$

Let $(M,g)$ be a $d$-dimensional riemannian manifold without boundary with
\begin{align}\label{ManifoldAssumptionAppendix}
\text{$M$ is complete, has a positive injectivity radius and a bounded geometry.}
\end{align}
We equip $M$ with the canonical volume $\mu$ and suppose that $M$ satisfies the doubling property: For all $x\in \tilde{M}$ and $r>0,$ we have
$\mu(B(x,r))<\infty$ and
\begin{align}\label{doublingAppendix}
\mu(B(x,2r))\lesssim \mu(B(x,r)).
\end{align}	
\begin{Definition}\label{DefinitionSobolevManifold}\begin{enumerate}
		\item[a)] Let $s\ge 0,$ $p\in (1,\infty),$ $\mathcal{A}:=\left(U_i,\kappa_i\right)_{i\in I}$ be an atlas of $M$ and $\left(\Psi_i\right)_{i\in I}$ a partition of unity subordinate to $\mathcal{A}.$ Then, we define the fractional Sobolev spaces $H^{s,p}(M)$ by
		\begin{align*}
		H^{s,p}(M):=\left\{f\in L^p(M): \norm{f}_{H^{s,p}(M)}:=\left(\sum_{i\in I} \norm{(\Psi_i f)\circ \kappa_i^{-1}}_{H^{s,p}(\Rd)}^p\right)^\frac{1}{p}<\infty\right\},
		\end{align*}
		where $H^{s,p}(\Rd)$ is the Sobolev space on $\Rd.$ For $p=2,$ we write
		$H^s(M):=H^{s,2}(M).$
		\item[b)] For $p\in[1,\infty),$ we define $W^{1,p}(M)$ as the completion of $C_c^\infty(M)$ in the norm
		\begin{align*}
		\norm{f}_{W^{1,p}(M)}:=\norm{f}_{L^p(M)}+\norm{\nabla f}_{L^p(M)},\qquad f\in C_c^\infty(M).
		\end{align*}
	\end{enumerate}
	
\end{Definition}

Note that in b), $\nabla f$ is an element of the tangential bundle of $M.$ We refer to \cite{lablee2015spectral} for further details.
A useful characterization of fractional Sobolev spaces is in terms of the fractional powers of the Laplace-Beltrami operator.
By Strichartz, \cite{Strichartz} Theorem 3.5, the restriction of $\left(e^{t\Delta_g}\right)_{t\ge 0}$
to $L^2(M)\cap L^p(M)$ extends to a strongly continuous semigroup on $L^p(M).$ We fix  $p\in (1,\infty)$ and $s>0.$ The generator $\left(\Delta_{g,p},\D(\Delta_{g,p})\right)$ is called the Laplace-Beltrami operator on $L^p(M).$ The negative fractional powers of $I-\Delta_{g,p}$ are defined by
\begin{align*}
\D((I-\Delta_{g,p})^{-\alpha}):=&\left\{f\in L^p(M):  \int_0^{\infty} t^{\alpha-1}e^{-t}e^{t\Delta_{g,p}}f \df t \quad \text{exists}\right\}\\
(I-\Delta_{g,p})^{-\alpha} f:=& \frac{1}{\Gamma(\alpha)}\int_0^{\infty} t^{\alpha-1}e^{-t}e^{t\Delta_{g,p}}f \df t
\end{align*}
for $\alpha>0.$
Note that in the case $p=2$ this coincides with the definition via the functional calculus because of the identity $\frac{1}{\Gamma(\alpha)}\int_0^{\infty} t^{\alpha-1}e^{-\lambda t}f \df t =\lambda^{-\alpha}$ for $\lambda >0.$\\
In the following Proposition, we list characterizations and embedding properties of the Sobolev spaces from Definition $\ref{DefinitionSobolevManifold}.$

\begin{Prop}\label{PropertiesFractionalSobolev}
	Let $\left(M,g\right)$ be a $d$-dimensional Riemannian manifold that satisfies \eqref{ManifoldAssumptionAppendix}. Let $s\ge 0$ and $p\in (1,\infty).$
	\begin{enumerate}
		\item[a)] We have $H^{s,p}(M)=R((I-\Delta_{g,p})^{-\frac{s}{2}}
		)$ with $\norm{f}_{H^{s,p}}\eqsim \norm{v}_{L^p}$ for $f=(I-\Delta_{g,p})^{-\frac{s}{2}} v.$\\
		Furthermore, we have $H^{1,p}(M)=W^{1,p}(M).$
		\item[b)] For $s>\frac{d}{p},$ we have $H^{s,p}(M)\hookrightarrow \LInfty.$
		\item[c)] Let $s\ge 0$ and $p\in (1,\infty).$ Suppose $p\in [2,\frac{2d}{(d-2s)_+})$ or $p=\frac{2d}{d-2s}$ if $s<\frac{d}{2}.$
		Then, the embedding $H^{s}(M)\hookrightarrow L^{p}(M)$ is continuous.
		If $M$ is compact and we have $0<s\le 1$ as well as $p\in [1,\frac{2d}{(d-2s)_+})$, the embedding $H^{s}(M)\hookrightarrow L^{p}(M)$ is compact.
		\item[d)] For $s,s_0,s_1\ge 0$ and $p,p_0,p_1\in(1,\infty)$ and $\theta\in(0,1)$ with
		\begin{align*}
		s=(1-\theta)s_0+\theta s_1,\qquad \frac{1}{p}=\frac{1-\theta}{p_0}+\frac{\theta}{p_1},
		\end{align*}
		we have
		$\left[H^{s_0,p_0}(M),H^{s_1,p_1}(M)\right]_\theta=H^{s,p}(M).$
	\end{enumerate}
\end{Prop}

\begin{proof}
	\emph{ad a):} See  \cite{Triebel}, Theorem 7.4.5. We remark that in the reference, $H^{s,p}$ is defined via the range identity from the Proposition and the identity from Definition \ref{DefinitionSobolevManifold} is proved. \\
	\emph{ad b):}  See \cite{Bolleyer}, Theorem III.1.2. d1). \\
	\emph{ad c):}
	For the first assertion, we refer to \cite{Bolleyer}, Theorem III.1.2. d1).
	If $M$ is  compact, we can choose a finite collection of charts and a finite partition of unity. Hence
	\begin{align}\label{chartsCompactnessSobolev}
	\norm{f}_{H^{s}(M)}:=\Big(\sum_{i=1}^N \norm{(\Psi_i f)\circ \varphi_i^{-1}}_{H^{s}(\Rd)}^2\Big)^\frac{1}{2}=\Big(\sum_{i=1}^N \norm{(\Psi_i f)\circ \varphi_i^{-1}}_{H^{s}(\mathcal{O})}^2\Big)^\frac{1}{2}
	\end{align}
	for a sufficiently large smooth bounded domain $\mathcal{O}\subset \Rd.$
	By \cite{HitchhikersGuideSobolev}, Corollary 7.2 and Theorem 8.2, the embedding
	$H^s(\mathcal{O})\hookrightarrow L^p(\mathcal{O})$
	is compact for $s\in (0,1)$ with $s<\frac{d}{2}$ and $p\in[1,\frac{2d}{d-2s}).$ Note that in the reference, the result is proved in terms of the Slobodetski space $W^{s,2}(\mathcal{O}),$ but we can use the identity $W^{s,2}(\mathcal{O})=H^s(\mathcal{O}).$ The embedding result combined with \eqref{chartsCompactnessSobolev} yields the assertion.\\	
	\emph{ad d):} See \cite{Triebel}, Section 7.4.5, Remark 2.
\end{proof}

In the next Lemma, we recall the  deterministic homogeneous Strichartz estimate due to Bernicot and Samoyeau, see \cite{Bernicot}, Corollary 6.2.

\begin{Lemma}\label{homogenousStrichartzLemma}
	Let $\varepsilon>0,$ $T>0$ and $2<p<\infty,$ $2<q\le \infty$ with $\frac{2}{q}+\frac{d}{p}=\frac{d}{2}.$
	Then,
	\begin{align}\label{homogenousStrichartzEstimateOriginal}
	\norm{e^{\im  t\Delta_g}x}_{L^q(0,T;L^p(M))}\lesssim_{T,\varepsilon} \norm{x}_{H^{\frac{1+\varepsilon}{q}}(M)},\qquad x\in H^{\frac{1+\varepsilon}{q}}(M).
	\end{align}
\end{Lemma}


We remark that in the special case of compact $M,$ Burq, G\'erard and Tzvetkov proved \eqref{homogenousStrichartzEstimateOriginal} even for $\varepsilon=0.$ But for our application in Section \ref{UniquenessSection}, this is not needed, such that we can prove uniqueness on non-compact manifolds with $d=2$ and
\eqref{ManifoldAssumptionAppendix}.

From Lemma $\ref{homogenousStrichartzEstimateOriginal},$ one can deduce the following Strichartz estimates for the stochastic and deterministic convolutions in fractional Sobolev spaces. Note that we choose the probability space $\Omega$ and the $Y$-valued Wiener process $W$ as in Assumption $\ref{stochasticAssumptions}$.

\begin{Lemma}\label{StrichartzLemma}
	In the situation of Lemma  $\ref{homogenousStrichartzLemma},$ we take $s\in[\frac{1+\varepsilon}{q},1]$ and $r\in (1,\infty).$
	\begin{enumerate}\item[a)] We have the homogeneous Strichartz estimate
		\begin{align}\label{homogenousStrichartzEstimate-1}
		\norm{e^{\im  t\Delta_g}x}_{L^q(0,T;H^{s-\frac{1+\varepsilon}{q},p}(M))}\lesssim_{T,\varepsilon} \norm{x}_{H^{s}(M)}
		\end{align}
		for $x\in H^{s}(M)$ and the inhomogeneous Strichartz estimate
		\begin{align}\label{inhomogenousStrichartzEstimate}
		\bigNorm{\int_0^\cdot e^{\im  (\cdot-\tau)\Delta_g}f(\tau)\df \tau}_{L^q(0,T;H^{s-\frac{1+\varepsilon}{q},p}(M))}\lesssim_{T,\varepsilon} \norm{f}_{L^1(0,T;H^{s}(M))}
		\end{align}
		for $f\in L^1(0,T;H^{s}(M)).$
		\item[b)] We have the stochastic Strichartz estimate
		\begin{align}\label{stochasticStrichartzEstimate}
		\bigNorm{\int_0^\cdot e^{\im  (\cdot-\tau)\Delta_g}B(\tau)\df W(\tau)}_{L^r(\Omega,L^q( 0,T;H^{s-\frac{1+\varepsilon}{q},p}(M)))}\lesssim_{T,\varepsilon} \norm{B}_{L^r(\Omega;L^2(0,T;\HS(Y,H^{s}(M))}
		\end{align}
		for all adapted processes
		in $B\in L^r(\Omega;L^2(0,T;\HS(Y,H^{s}(M)).$
	\end{enumerate}
\end{Lemma}

\begin{proof}
	Proposition $\ref{PropertiesFractionalSobolev}$ a) and Lemma  $\ref{homogenousStrichartzLemma}$ yield
	\begin{align}\label{homogenousStrichartzEstimate-2}
	\norm{e^{\im  t\Delta_g}x}_{L^q(0,T;H^{s-\frac{1+\varepsilon}{q},p}(M))}
	&\eqsim\norm{ (1-\Delta_g)^{\frac{s}{2}-\frac{1+\varepsilon}{2q}} e^{\im  t\Delta_g}x}_{L^q(0,T;L^p(M))}\nonumber\\
	&=\norm{  e^{\im  t\Delta_g} (1-\Delta_g)^{\frac{s}{2}-\frac{1+\varepsilon}{2q}} x}_{L^q(0,T;L^p(M))}\nonumber\\
	&\lesssim_{T,\varepsilon} \norm{(1-\Delta_g)^{\frac{s}{2}-\frac{1+\varepsilon}{2q}}x }_{H^{\frac{1+\varepsilon}{q}}(M)}\eqsim\norm{x}_{H^s(M)}.
	\end{align}
	From \eqref{homogenousStrichartzEstimate-2}, we get
	\begin{align}\label{inhomogenousStrichartzEstimateZero}
	\bigNorm{\int_0^\cdot e^{\im  (\cdot-\tau)\Delta_g}f(\tau)\df \tau}_{L^q(0,T;H^{s-\frac{1+\varepsilon}{q},p}(M))}\lesssim_{T,\varepsilon} \bigNorm{\int_0^\cdot e^{-\im \tau\Delta_g}f(\tau)\df \tau}_{\Hs}
	\lesssim \norm{f}_{L^1(0,T;\Hs)}
	\end{align}
	and  Theorem 3.10 in \cite{BrzezniakStrichartz} implies
	\begin{align}\label{stochasticStrichartzEstimatezero}
	\bigNorm{\int_0^\cdot e^{\im  (\cdot-\tau)\Delta_g}B(\tau)\df W(\tau)}_{L^r(\Omega,L^q(0,T;L^p(M)))}\lesssim_{T,\varepsilon} \norm{B}_{L^r(\Omega;L^2(0,T;\HS(Y,H^{\frac{1+\varepsilon}{q}}(M))}.
	\end{align}
	With the same procedure as in $\eqref{homogenousStrichartzEstimate-2},$ one can deduce the estimate  $\eqref{stochasticStrichartzEstimate}.$ 	
\end{proof}

\section*{Acknowledgement}

We thank the anonymous referees for many valuable comments. Moreover, we thank Peer Kunstmann and Dorothee Frey for their help in the context of analysis on manifolds. Finally, we  thank Nimit Rana for careful reading of the manuscript.


\Addresses
\end{document}